\documentclass[12pt]{amsart}
\usepackage{tikz-cd}
\usepackage{enumitem}
\setlist[description]{leftmargin=\parindent,labelindent=\parindent}
\usetikzlibrary{matrix,arrows,decorations.pathmorphing}
\usepackage{amssymb}
\usepackage{amsmath,amscd}
\usepackage{mathrsfs}
\usepackage{url}
\usepackage{cite}
\usepackage{fullpage}
\usepackage{hyperref}
\usepackage{verbatim}
\usepackage{comment}
\usepackage{fancyvrb}
\usepackage{fvextra}
\newtheorem{thm}{Theorem}[section]
\newtheorem{prop}[thm]{Proposition}
\newtheorem{lem}[thm]{Lemma}
\newtheorem{cor}[thm]{Corollary}

\theoremstyle{definition}
\newtheorem{definition}[thm]{Definition}
\newtheorem{example}[thm]{Example}
\newtheorem{rem}[thm]{Remark}

\numberwithin{equation}{section}

\newcommand{\X}{\mathcal{X}}
\newcommand{\Y}{\mathcal{Y}}

\newcommand{\V}{\mathcal{V}}

\newcommand{\N}{\mathcal{N}}
\newcommand{\B}{\mathcal{B}}

\newcommand{\zz}{\mathbb{Z}}
\newcommand{\qq}{\mathbb{Q}}

\newcommand{\C}{\mathcal{C}}
\newcommand{\M}{\mathcal{M}}
\newcommand{\p}{\mathbb{P}}
\newcommand{\pp}{\mathbb{P}}

\renewcommand{\H}{\mathcal{H}}
\newcommand{\F}{\mathcal{F}}
\renewcommand{\P}{\mathcal{P}}
\newcommand{\E}{\mathcal{E}}
\renewcommand{\O}{\mathcal{O}}

\newcommand{\Mg}{\mathcal{M}_g}
\renewcommand{\tilde}{\widetilde}
\DeclareMathOperator{\rank}{rank}

\DeclareMathOperator{\Aut}{Aut}
\DeclareMathOperator{\GL}{GL}
\DeclareMathOperator{\SL}{SL}
\DeclareMathOperator{\SP}{Sp}
\DeclareMathOperator{\Pic}{Pic}
\DeclareMathOperator{\Spec}{Spec}
\DeclareMathOperator{\Supp}{Supp}

\DeclareMathOperator{\coker}{coker}

\DeclareMathOperator{\Sym}{Sym}

\DeclareMathOperator{\PGL}{PGL}
\DeclareMathOperator{\BSL}{BSL}

\DeclareMathOperator{\BGL}{BGL}
\DeclareMathOperator{\codim}{codim}

\DeclareMathOperator{\PSp}{PSp}
\DeclareMathOperator{\PGSp}{PGSp}

\DeclareMathOperator{\BSp}{BSp}
\renewcommand{\gg}{\mathbb{G}}
\newcommand{\ee}{{\vec{e}} \nobreak\hspace{.16667em plus .08333em}'}

\begin{document}
\title{The Chow Rings of the Moduli Spaces of Curves of Genus 7, 8, and 9}

\author{Samir Canning}
\address{Department of Mathematics, ETH Z\"urich, R\"amistrasse 101, 8092 Z\"urich, Switzerland.}
\author{Hannah Larson}
\address{Department of Mathematics, Harvard University, One Oxford Street, Cambridge, MA 02138.}
\thanks{During the preparation of this article, S.C. was partially supported by NSF RTG grant DMS-1502651. H.L. was supported by the Hertz Foundation and NSF GRFP under grant DGE-1656518. This work will be part of S.C.'s and H.L.'s Ph.D. theses.}
\email{samir.canning@math.ethz.ch}
\email{hlarson@math.harvard.edu}
\maketitle

\begin{abstract}
The rational Chow ring of the moduli space $\M_g$ of curves of genus $g$ is known for $g \leq 6$.
Here, we determine the rational Chow rings of $\M_7, \M_8,$ and $\M_9$ by showing they are tautological. 
The key ingredient is intersection theory on Hurwitz spaces of degree $4$ and $5$ covers of $\pp^1$ via their associated vector bundles. 
The main focus of this paper is a detailed geometric analysis of special tetragonal and pentagonal covers whose associated vector bundles on $\pp^1$ are highly unbalanced, expanding upon previous work of the authors in the more balanced case. 
In genus $9$, we use work of Mukai 
to present the locus of hexagonal curves as a global quotient stack,
and, using equivariant intersection theory, we show its Chow ring is generated by restrictions of tautological classes.
\end{abstract}

\section{Introduction}
In his landmark paper \cite{Mum}, Mumford introduced the Chow ring of the moduli space $\M_g$ of genus $g$ curves. Since then, much progress has been made on the determination of $A^*(\M_g)$ in low genus, which we summarize below.

\begin{itemize}[leftmargin=50pt]
    \item[($g = 2$)] Mumford \cite{Mum} in 1983, determined $A^*(\overline{\M}_2)$ with rational coefficients. \\
    Vistoli \cite{V3} in 1998, determined $A^*(\M_2)$ with integral coefficients, \\
    E. Larson \cite{L2} in 2020, determined $A^*(\overline{\M}_2)$ with integral coefficients.
    \item[($g = 3$)] Faber \cite{F2} in 1990, determined $A^*(\overline{\M}_3)$ with rational coefficients, \\
    Di Lorenzo--Fulghesu--Vistoli \cite{DFV} in 2020, determined the integral Chow ring of the locus of smooth plane quartics.
    \item[($g=4$)] Faber \cite{F3} in 1990, determined $A^*(\M_4)$ with rational coefficients.
    \item [($g=5$)] Izadi \cite{Iz} in 1995, determined $A^*(\M_5)$ with rational coefficients.
    \item[($g=6$)] Penev--Vakil \cite{PV} in 2015, determined $A^*(\M_6)$ with rational coefficients.
\end{itemize}

In each of the above cases,
the rational Chow ring of $\M_g$ is equal to the \emph{tautological subring} $R^*(\Mg) \subseteq A^*(\M_g)$, a subring generated by certain natural classes which we now define.
Let $f: \mathcal{C} \rightarrow \M_g$ be the universal curve. 
The tautological subring is the subring of $A^*(\M_g)$ generated by the \emph{kappa classes}, $\kappa_i := f_*(c_1(\omega_f)^{i+1})$.

In this paper, we
tackle the next open cases of genus $7, 8$, and $9$ using the new machinery of tautological classes on the Hurwitz space \cite{part1,part2}. We prove that the rational Chow rings of $\M_7, \M_8,$ and $\M_9$ are all generated by tautological classes, and thereby determine these Chow rings using work of Faber \cite{F}.
In addition to our theorems in genus $7, 8$, and  $9$, our techniques give new and much simpler proofs of the genus $5$ and $6$ cases (see Section \ref{tet56}).
In particular, in genus $6$, we establish that all classes supported on the
bielliptic locus are tautological, which was not fully explained in \cite{PV}. 
\begin{thm} \label{main}
The Chow ring of the moduli space of genus $7$ curves is generated by tautological classes. Hence,
\[
A^*(\mathcal{M}_7)\cong \qq[\kappa_1,\kappa_2]/I_7,
\]
where $I_7$ is the ideal generated by the classes
\[
\begin{cases}
    2423\kappa_1^2\kappa_2-52632\kappa_2^2 \\
    1152000\kappa_2^2-2423\kappa_1^4\\
    16000\kappa_1^3\kappa_2-731\kappa_1^5.
\end{cases}
\]

\end{thm}
The computation of the tautological ring is originally due to Faber \cite{F}. We used the Sage \cite{sage} package admcycles \cite{admcycles} and a program of Pixton \cite{Pixton} to obtain the above presentation and those below. 


\begin{thm} \label{main8}
The Chow ring of the moduli space of genus $8$ curves is generated by tautological classes. Hence,
\[
A^*(\M_8)=\qq[\kappa_1,\kappa_2]/I_8,
\]
where $I_8$ is the ideal generated by the classes
\[
\begin{cases}
714894336\kappa_2^2-55211328\kappa_1^2\kappa_2+1058587\kappa_1^4\\
62208000\kappa_1\kappa_2^2-95287\kappa_1^5 \\
144000\kappa_1^3\kappa_2-5617\kappa_1^5. \\

\end{cases}
\]
\end{thm}


\begin{rem}
The authors would like to point out contemporaneous work of Maxwell da Paixão de Jesus Santos, which, using different techniques, makes significant progress towards showing $A^*(\M_8)$ is tautological (it is proved that non-tautological classes must be supported on the bielliptic locus).
\end{rem}

\begin{thm} \label{main9}
The Chow ring of the moduli space of genus $9$ curves is generated by tautological classes. Hence,
\[
A^*(\M_9)=\qq[\kappa_1,\kappa_2,\kappa_3]/I_9,
\]
where $I_9$ is the ideal generated by the classes
\[
\begin{cases}
    5195\kappa_1^4+3644694\kappa_1\kappa_3+749412\kappa_2^2-265788\kappa_1^2\kappa_2
    \\ 
    33859814400\kappa_2\kappa_3-95311440\kappa_1^3\kappa_2+2288539\kappa_1^5
    \\
    19151377\kappa_1^5+16929907200\kappa_1\kappa_2^2-1142345520\kappa_1^3\kappa_2
    \\
    1422489600\kappa_3^2-983\kappa_1^6
    \\
    1185408000\kappa_2^3-47543\kappa_1^6.
    \\
\end{cases}
\]
\end{thm}

\begin{rem}
Despite their complicated looking presentations, the tautological rings above have many nice properties. Faber proved that they are Gorenstein rings with socle in degree $g-2$ \cite{F}.  He also points out that $g=9$ is the first case in which the tautological ring is not a complete intersection ring.
Several different methods of producing relations among tautological classes in arbitrary genus have found only the Faber--Zagier relations, which may suggest that the Gorenstein property only occurs in low genus cases (see \cite{Pan} for a discussion).
\end{rem}

\begin{rem}
An interesting consequence of Theorems \ref{main}, \ref{main8}, and \ref{main9} is that for $g=7,8,9$, the cycle class map $A^*(\Mg)\rightarrow H^{2*}(\Mg,\qq)$ is injective. It is unknown whether this holds in general, and it could even fail quite dramatically: when $g$ is large, it is unknown whether $A^*(\M_g)$ is finite or infinite dimensional as a $\qq$-vector space, whereas cohomology is finitely generated for any algebraic variety. On the moduli space of stable curves, Pikaart \cite[Corollary 4.7]{Pikaart} has shown that $H^{33}(\overline{\M}_g,\qq)\neq 0$ for $g$ sufficiently large. It then follows from work of Jannsen \cite[Theorem 3.6]{Jannsen} that the map $A^*(\overline{\M}_g)\rightarrow H^{2*}(\overline{\M}_g,\qq)$ is not injective for $g$ sufficiently large.
\end{rem}


In our previous work about tautological classes on the Hurwitz space \cite{part1,part2}, degree four covers $C\rightarrow \p^1$ that factor through a lower genus curve presented a major difficulty. The primary example of this issue is when $C$ is \emph{bielliptic}, where degree four covers $C\rightarrow \p^1$ arise from the double cover $C\rightarrow D$ and any double cover $D\rightarrow \p^1$, where $D$ is an elliptic curve.
A main challenge of this paper is therefore to prove that classes supported on the bielliptic locus of $\M_g$ for $g \leq 9$ are tautological.
Indeed, the bielliptic locus is the source of the first known example of a \emph{nontautological} algebraic class on $\M_g$: in \cite{VZ}, van Zelm proves that the fundamental class of the locus of bielliptic curves $\B_{12} \subset \M_{12}$ is nontautological.
The techniques we develop for the bielliptic locus in genus $g \leq 9$
also extend to genus $10$.
\begin{thm}\label{bielliptic10}
The fundamental class of the bielliptic locus $\B_{10} \subset \M_{10}$ is tautological (hence equal to zero).
\end{thm}

\subsection{Overview of the proof} Our basic approach is to use the stratification of $\M_g$ by \emph{gonality}, the minimal degree of a map $C \to \pp^1$. 
Precisely, let us define
\[
\M_g^k:=\{[C]\in \M_g: C \text{ has a } g^1_k\},
\]
which is the locus of curves of gonality less than or equal to $k$. For $g=7, 8$, a general curve of genus $g$ has gonality $5$, so our stratification takes the form
\[
\M^2_g\subset \M^3_g\subset \M^4_g\subset \M^5_g=\M_g.
\]
In genus $9$, a general curve has gonality $6$, so we have one more stratum
\[
\M^2_9\subset \M^3_9\subset \M^4_9\subset \M^5_9 \subset \M_9^6 =\M_9.
\]
It suffices to show for each $k$ that all classes supported on $\M_g^k$ are tautological up to classes supported on $\M_g^{k-1}$. In other words, we must show that every class in $A^*(\M_g^k \smallsetminus \M_g^{k-1})$ pushes forward to a class in $A^*(\M_g \smallsetminus 
\M_g^{k-1})$ that is the restriction of a tautological class on $\M_g$. 

We shall call a class on $A^*(\M_g \smallsetminus 
\M_g^{k-1})$ tautological if it is the restriction of a tautological class from $\M_g$. 
As $k$ increases, each stratum $\M_g^k \smallsetminus \M_g^{k-1}$ becomes more complicated. 
Our main contribution is a better understanding of the strata for $k = 4, 5$, which was the main stumbling block in extending previous work on low genus curves.
We now explain our process in more detail, starting with the curves of lowest gonality and working upwards.

\subsection*{(1) An easy start}
Faber \cite{F} showed that the fundamental class of any Brill--Noether locus of the expected dimension is tautological. In particular the fundamental class of $\M_g^k$ is tautological.
It is well known that $A^*(\M_g^2) \cong \qq$ for all $g$. By a result of Patel-Vakil \cite{PV2}, $A^*(\M_g^3\smallsetminus \M_g^2)$ is generated by the restriction of $\kappa_1$ for all $g \neq 3$.
Using Faber's result and the push-pull formula this establishes for $g \neq 3$ that
\begin{equation} \label{Mg3}
\text{all classes supported on $\M_g^3$ are tautological.}
\end{equation}
(See Remark \ref{g3rem} for an alternative argument when $g = 3$.)
Note that \eqref{Mg3} already establishes that $A^*(\M_g) = R^*(\M_g)$ for $g \leq 4$.
More generally, using the push-pull formula and Faber's result, if the Chow ring of each locally closed stratum $A^*(\M_g^k \smallsetminus \M_g^{k-1})$ were generated by the restrictions of tautological classes for all $k$, we would be done. However, this is not the case for $k > 3$.


\subsection*{(2) Why it must get harder}
For $k = 4, 5$,
the Chow ring of $\M_g^k\smallsetminus \M_g^{k-1}$ 
is \emph{not} in general generated by restrictions of tautological classes. By considering curves of bidegree $(4, 4)$ on $\pp^1 \times \pp^1$, one can show that when $g \geq 8$, the map $\beta: \H_{4,g} \to \M_g^4$ is an isomorphism away from loci of codimension $2$ in both spaces. Therefore, the Picard rank conjecture, proved by Deopurkar--Patel \cite{DP} for $k = 4, 5$, shows that 
\[\dim A^1(\M_g^4\smallsetminus \M_g^3) = \dim A^1(\H_{4,g}) = 2.\]
Hence, the first Chow group of the locally closed stratum $\M_g^4 \smallsetminus \M_g^3$ cannot be generated by the restriction of $\kappa_1$. The analogous result holds for $k = 5$ when $g \geq 10$.
Furthermore, it is known that there exist classes supported on $\M_g^4$ that are not tautological in some genera:
van Zelm \cite{VZ} has shown that the fundamental class of the bielliptic locus $\B_{12} \subset \M_{12}$ is \emph{not} a tautological class. 
This makes the tetragonal locus (Section \ref{tet}) one of the most interesting parts, and it will, of course, require some special observations about genus $7$, $8$, and $9$ curves. (In Section \ref{tet56}, we also explain how to prove $A^*(\M_g) = R^*(\M_g)$ for $g = 5$ and $6$ using our techniques.) In Section \ref{noluck}, we discuss why our techniques cannot access the bielliptic locus when $g \geq 11$. In the case $g = 10$, we prove Theorem \ref{bielliptic10}: the class of the bielliptic locus $\B_{10} \subset \M_{10}$ is tautological.

\subsection*{(3) Using the Hurwitz space}
Our approach is to study the Chow rings of the Hurwitz stacks $\H_{4,g}$ and $\H_{5,g}$ parametrizing degree $4$ and $5$ covers, respectively, of the projective line. Let $\beta: \H_{k,g} \to \M_g$ denote the forgetful map.  The induced map $\H_{k,g} \smallsetminus \beta^{-1}(\M_g^{k-1}) \to \M_g^k \smallsetminus \M_g^{k-1}$ is proper and surjective, and thus induces a surjection on rational Chow groups.
In \cite{part2}, we showed that for $k = 4, 5$, classes in the \emph{tautological ring of} $\H_{k,g} \smallsetminus \beta^{-1}(\M^{k-1}_g)$ (see Definition \ref{tdef}) push forward to tautological classes on $\M_g \smallsetminus \M^{k-1}_g$ (see Theorem \ref{pushforward}). 
This is a useful (and non-trivial) tool because there are tautological classes on the Hurwitz space which are \emph{not} pullbacks of tautological classes on $\M_g$.
Thus, we wish to show that the Chow rings of $\H_{k,g} \smallsetminus \beta^{-1}(\M^{k-1}_g)$ are tautological for $k = 4, 5$ and $g = 7, 8, 9$. We succeed in proving this in each of these cases except $k = 4, g = 9$, where some additional special arguments are used. These arguments are carried out in Section \ref{tet} for $k=4$ and Section \ref{pent} for $k=5$.

To accomplish this, we further stratify $\H_{k,g}$.
Given a cover $\alpha: C \to \pp^1$, we define $E_\alpha := (\alpha_* \O_C/\O_{\pp^1})^\vee$ and $F_\alpha := \ker(\Sym^2 E_\alpha \to \alpha_*\omega_\alpha^{\otimes 2})$, which are vector bundles on $\pp^1$.
See Section \ref{CEsec} for an elaboration of the properties of $E_{\alpha}$ and $F_{\alpha}$.
We then stratify $\H_{k,g}$ by the pair of splitting types of $E_\alpha$ and  $F_\alpha$.
Each of these ``pair splitting loci" has a nice description as a quotient stack (Lemmas \ref{squo} and \ref{squo5}).
As a starting point, our previous work \cite{part1,part2} shows that the Chow ring of a union $\Psi$ of the several largest strata is generated by tautological classes (Proposition \ref{Hprime}). This result allows us to narrow down the possible sources of non-tautological classes: they all occur on the complement of $\Psi$. 
Some ``bad" pair splitting loci $\Sigma_i$ remain outside 
of $\Psi$ and not inside $\beta^{-1}(\M_g^{k-1})$. These bad $\Sigma_i$ are the main focus of this paper. Part of the difficulty of these strata is that they all occur in the ``unexpected (pair) codimension" in the sense of Deopurkar--Patel \cite[Remark 4.2]{DP}.

\subsection*{(4) The key coincidence and work to be done.} Using universal degeneracy formulas from \cite{L}, we show that the fundamental class of a single splitting locus (i.e. where \emph{one} of the two vector bundles has a given splitting type) is tautological if it occurs in the ``expected codimension." 
Perhaps the most surprising part of the proof is the following coincidence (when $(k, g) \neq (4, 9)$): after excising strata contained in $\beta^{-1}(\M^{k-1}_g)$, every bad $\Sigma_i$ can be realized as a single splitting locus, and that single splitting locus occurs in the expected codimension (proofs of Lemmas \ref{H47}, \ref{H48}, \ref{H58}, \ref{H59}). Hence, the fundamental class of the closure of each bad $\Sigma_i$ is tautological in $A^*(\H_{k,g} \smallsetminus \beta^{-1}(\M^{k-1}_g))$. That these fundamental classes are tautological is a ``coincidence of small numbers." It in fact fails for $k=4$, $g=12$ by the result of van Zelm \cite{VZ}. 

We then study the Chow rings of the locally closed strata $\Sigma_i$.
Using the description of $\Sigma_i$ as a quotient, we show that the Chow ring of each stratum $A^*(\Sigma_i \smallsetminus \beta^{-1}(\M_g^{k-1}))$ is generated by restrictions of tautological classes on $\H_{k,g}$ (Sections \ref{chows} and \ref{stratasec}). This last step requires a geometric understanding of the equations that define $C$ inside the associated scroll $\pp E_\alpha^\vee$, and when a collection of equations of this type fail to define a smooth curve or produce a curve of gonality less than $k$. These ideas do extend to arbitrary genus, unlike the results in the previous paragraph concerning fundamental classes. We state them as broadly as possible for arbitrary genera as they may be of future use. 

\subsection*{(5) The further work in genus $9$} As the genus increases, the luck with fundamental classes starts to run out and more subtle arguments are required.
In the case $g = 9, k = 4$, we encounter two bad pair strata $\Sigma_i$ that occur in unexpected codimension and cannot be realized as a single splitting locus. Although we do not compute their classes on $\H_{4,9}$, we still manage to show that their push forwards to $\M_9 \smallsetminus \M_9^3$ are tautological. For example, one of these problem strata corresponds to the locus of plane sextics with one double point (Lemma \ref{hasg26}). The class of this locus is tautological on $\M_9 \smallsetminus \M_9^3$ because it is a Brill--Noether locus of the expected codimension. We then show that the Chow ring of this stratum is generated by the pullback of $\kappa_1$ and $\kappa_2$, which is a stronger statement than being generated by restrictions of tautological classes on $\H_{4,9}$. By the push-pull formula, the push forward of every class supported on this stratum is tautological on $\M_9 \smallsetminus \M_9^3$ (though we remain unsure if they are tautological upstairs on $\H_{4,9}$). We deal with the other problem stratum by showing that its union with the bielliptic locus $\beta^{-1}(\B_9)$ has tautological fundamental class on $\H_{4, 9} \smallsetminus \beta^{-1}(\M_9^3)$ and a trick explained in Figure \ref{78pic}.

In genus $9$, we must also deal with curves of gonality $6$. 
The approach we take to these curves in Section \ref{Gen} is quite different from the approach taken to curves of gonality $5$ and below because there is no uniform description for degree $6$ covers in terms of associated vector bundles.
Instead, using results of Mukai \cite{Muk}, we realize $\M_9 \smallsetminus \M_9^5$ (up to a $\mu_2$ gerbe) as a global quotient of an open subvariety of a Grassmannian by $\SP_6$. The tautological subbundle on the resulting Grassmann bundle over $\BSp_6$ is the Hodge bundle (up to possibly twisting by a line bundle, see Lemma \ref{Sbundle}). It then remains to see that Chern classes of the rank $6$ vector bundle $\V$ associated to this quotient are tautological.
We see this by proving that the rank $21$ vector bundle $\Sym^2 \V$ is the bundle of $21$ quadrics that cut out a canonical genus $9$ curve (Lemma \ref{quadricsbundle}). From this, it follows that the Chern classes of $\Sym^2 \V$ are tautological, and, using the splitting principle, the Chern classes of $\V$ are seen to be tautological as well.
\subsection{Notations and conventions.}
All schemes in this paper are taken over a fixed algebraically closed field of characteristic $0$ or $p>5$. All Chow rings are taken with rational coefficients. We use the subspace convention for projective bundles and Grassmann bundles. 

\subsection*{Acknowledgments} We would like to thank Jarod Alper, Andrea Di Lorenzo, Giovanni Inchiostro, Elham Izadi,  Aaron Landesman, and Ravi Vakil for helpful conversations. We thank Gregor Botero and Maxwell da Paixão de Jesus Santos for their correspondence. We are also grateful to Johannes Schmitt and Jason van Zelm for their advice and help on computing the tautological rings with admcycles, and to David Holmes and Johannes Schmitt for correcting a miscalculation in an earlier version of the presentation of the tautological rings. In addition, we thank Rahul Pandharipande, Dan Petersen, and Burt Totaro for their comments on an early draft of this article. We thank the referee for insightful comments that helped improve this paper.

\section{Hurwitz Schemes and the Tautological Ring}\label{Hurwitz}
In order to study the loci $\M_g^4\smallsetminus \M_g^3$ and $\M_g^5\smallsetminus \M_g^4$, we will study the Hurwitz stacks $\H_{4,g}$ and $\H_{5,g}$ parametrizing degree $4$ and $5$ covers, respectively, of the projective line. 

\begin{definition}
The unparametrized Hurwitz stack $\H_{k,g}$ is the stack whose objects over a scheme $S$ are of the form $(C\rightarrow P\rightarrow S)$ where $P\rightarrow S$ is a $\pp^1$-fibration, $C\rightarrow P$ is a finite flat finitely presented morphism of constant degree $k$, and the composition $C\rightarrow S$ is a family of smooth genus $g$ curves. 
\end{definition}
The Hurwitz stack $\H_{k,g}$ admits a universal diagram
\[
\begin{tikzcd}
\C \arrow[r, "\alpha"] \arrow[rd, "f"'] & \P \arrow[d, "\pi"] \\
                                        & {\H_{k,g}}         
\end{tikzcd}
\]
The universal diagram furnishes several natural classes in the Chow ring of $\H_{k,g}$.
\begin{definition} \label{tdef}
The tautological ring $R^*(\H_{k,g})$ is the subring of $A^*(\H_{k,g})$ generated by classes of the form
\[
f_*(c_1(\omega_f)^i\cdot\alpha^*c_1(\omega_\pi)^j).
\]
If $U \subseteq \H_{k,g}$ is an open substack of $\H_{k,g}$, we define the tautological ring of the open $R^*(U)$ to be the image of the tautological ring under the restriction map
\[
A^*(\H_{k,g})\rightarrow A^*(U).
\]
\end{definition}

\begin{rem}[A note on the $\SL_2$ quotient]  \label{sl2rem}
The Hurwitz stack $\H_{k,g}$ is the $\PGL_2$ quotient of the parametrized Hurwitz scheme $\H_{k,g}^\dagger$. 
One can also take the quotient of $\H_{k,g}^\dagger$ by $\SL_2$. The map $[\H_{k,g}^\dagger/\SL_2] \to [\H_{k,g}^\dagger/\PGL_2] = \H_{k,g}$ is
a $\mu_2$-banded gerbe. It is a general fact that, \emph{with rational coefficients}, the pullback map along any gerbe banded by a finite group induces an isomorphism on Chow rings \cite[Section 2.3]{PV}.
In particular, $A^*(\H_{k,g}) \cong A^*([\H_{k,g}^\dagger/\SL_2])$.
The benefit of the $\SL_2$ quotient is that the pullback of the universal $\pp^1$-fibration to the $\SL_2$ quotient is a $\pp^1$ bundle, i.e. it is equipped with a line bundle of relative degree $1$.
Since we work with rational coefficients throughout, we do not distinguish the $\PGL_2$ and $\SL_2$ quotients and freely assume that $\P$ is equipped with a line bundle $\O_{\P}(1)$ of relative degree $1$. The push forward $\pi_* \O_{\P}(1)$ is the pullback of the universal rank $2$ vector bundle on $\BSL_2$.
\end{rem}

By forgetting the map $C\rightarrow P$, we obtain a morphism
\[
\beta:\H_{k,g}\rightarrow \M_g.
\]
Let 
\[
\beta':\H_{k,g}\smallsetminus \beta^{-1} (\M^{k-1}_g)\rightarrow \M_g\smallsetminus \M^{k-1}_g
\]
be the restriction of $\beta$ to $\H_{k,g}\smallsetminus \beta^{-1} (\M^{k-1}_g)$.
In \cite[Theorem 1.5]{part2}, we showed the following result relating the tautological rings of the relevant Hurwitz stacks and $\M_g$.
\begin{thm} \label{pushforward}
Let $k=4,5$. The map $\beta'$ is proper, so the induced push forward map 
\[\beta_*':A^*(\H_{k,g}\smallsetminus \beta^{-1} (\M^{k-1}_g))\rightarrow A^*(\M_g^k \smallsetminus \M^{k-1}_g)\] is surjective. Moreover, $\beta_{*}'(R^*(\H_{k,g}\smallsetminus \beta^{-1} (\M^{k-1}_g)))\subseteq R^*(\M_g\smallsetminus\M_g^{k-1})$.
\end{thm}
\begin{rem} \label{g3rem}
In the case $k=3$, it is also true that $\beta'_*(R^*(\H_{3,g} \smallsetminus \beta^{-1}(\M_g^2))) \subset R^*(\M_g \smallsetminus \M_g^2)$. For $g \neq 3$, this follows from work of Patel--Vakil \cite{PV} which shows $A^*(\H_{3,g}) = R^*(\H_{3,g})$ is generated by $\beta^*\kappa_1$. In genus $3$, it turns out $\beta^* \kappa_1 = 0$, so we instead prove the claim as follows. (The following argument does not presuppose $A^*(\M_3) = R^*(\M_3)$ and therefore provides a new proof of this fact in line with our approach.) Let $T \in A^1(\H_{3,3})$ be the class of the locus of covers with a point of triple ramification. By \cite[Proposition 2.15]{DP}, we have $A^1(\H_{3,3}) = R^1(\H_{3,3}) = \qq \cdot  T$. By \cite[Theorem 1.1 (1)]{part2}, we have $A^i(\H_{3,3}) = 0$ for all $i \geq 2$ and $R^*(\H_{3,3}) = A^*(\H_{3,3})$. By \cite[Corollary 7.5]{part2}, $\beta'_*( T)$ is tautological. Hence, the push forwards of all classes from $\H_{3,3}$ are tautological on $\M_3 \smallsetminus \M_3^2$. This argument is representative of the ideas that were used to prove Theorem \ref{pushforward} in \cite[Theorem 1.5]{part2}.
\end{rem}

In light of Theorem \ref{pushforward}, in order to prove Theorems \ref{main} and \ref{main8}, it suffices to show that $A^*(\H_{4,g}\smallsetminus \beta^{-1} (\M^3_g))$ and $A^*(\H_{5,g}\smallsetminus \beta^{-1} (\M^4_g))$ are generated by tautological classes for $g = 7, 8$. We will prove this in Sections \ref{tet} and \ref{pent}. More work is required when $g = 9$.

\subsection{The Casnati--Ekedahl structure theorem} \label{CEsec}
Here, we recall the Casnati--Ekedahl structure theorems for finite Gorenstein covers. The structure theorems furnish distinguished tautological classes, which we call the Casnati--Ekedahl classes, abbreviated CE classes. 
\par We begin with the most general statement, which holds for covers of every degree. Given a degree $k$ cover $\alpha:X\rightarrow Y$ where $Y$ is integral, one obtains an exact sequence
\begin{equation}\label{eqn:defining}
0\rightarrow \mathcal{O}_Y\rightarrow \alpha_*\mathcal{O}_X\rightarrow E_\alpha^{\vee}\rightarrow 0,
\end{equation}
where $E_\alpha$ is a vector bundle of rank $k-1$. 
When $\alpha$ is Gorenstein, $\alpha_* \O_X \cong (\alpha_* \omega_{\alpha})^\vee$. Pulling back and using adjunction, we therefore obtain a map 
\[
\omega_{\alpha}^\vee \to (\alpha^*\alpha_* \omega_{\alpha})^\vee \rightarrow \alpha^*E_\alpha^{\vee},
\]
which induces a map $X\rightarrow \p E_\alpha^{\vee}$ that factors $\alpha:X\rightarrow Y$.
The Casnati--Ekedahl structure theorem gives a resolution of the ideal sheaf of $X$ inside of $\p E_\alpha^{\vee}$ \cite{CE}.
\begin{thm}[Casnati--Ekedahl \cite{CE}]\label{CEstructure}
Let $X$ and $Y$ be schemes, $Y$ integral and let $\alpha:X\rightarrow Y$ be a Gorenstein cover of degree $k\geq 3$. There exists a unique $\p^{k-2}$-bundle $\gamma:\p\rightarrow Y$ and an embedding $i:X\hookrightarrow \p$ such that $\alpha=\gamma\circ i$ and $X_y:=\alpha^{-1}(y)\subset \gamma^{-1}(y)\cong \p^{k-2}$ is a nondegenerate arithmetically Gorenstein subscheme for each $y\in Y$. Moreover, the following properties hold.
\begin{enumerate}
    \item $\p \cong \p E_\alpha^{\vee}$ where $ E_\alpha^{\vee} :=\coker(\mathcal{O}_Y\rightarrow \alpha_*\O_X)$, and $i^*\O_{\pp}(1) \cong \omega_\alpha$.
    \item There is a unique up to unique isomorphism exact sequence of locally free $\mathcal{O}_\p$ sheaves
    \begin{equation}\label{CEseq}
        0\rightarrow \gamma^*F_{k-2}(-k)\rightarrow \gamma^*F_{k-3}(-k+2)\rightarrow \cdots\rightarrow \gamma^*F_{1}(-2)\rightarrow \mathcal{O}_\p\rightarrow \mathcal{O}_X\rightarrow 0.
    \end{equation}
    where $F_i$ is locally free on $Y$.
    The restriction of the exact sequence above to a fiber gives a minimal free resolution of $X_y:=\alpha^{-1}(y)$. Moreover the resolution is self-dual, so there is a canonical isomorphism $\H om_{\O_{\pp}}(F_i, F_{k-2}) \cong F_{k-2-i}$.
    \item The ranks of the $F_i$ are 
    \[ 
    \rank F_i=\frac{i(k-2-i)}{k-1}{{k}\choose{i+1}}.
    \]
    \item There is a canonical isomorphism $F_{k-2}\cong \det  E_\alpha$.
\end{enumerate}
\end{thm}
In the cases $k = 4, 5$, self-duality of the resolution determines all of the bundles $F_i$ in terms of $E_\alpha$ and $F_\alpha := F_1$ and tensor products and determinants thereof.
Twisting \eqref{CEseq} by $\O_{\pp}(2)$ and pushing forward, we see that $F_\alpha = \ker(\Sym^2 E_\alpha \to \alpha_*\omega_{\alpha}^{\otimes 2})$. We shall use this notation throughout.

Applying this to the universal cover
$\alpha: \C \to \P$ over $\H_{k,g}$, we obtain vector bundles $\E := E_\alpha$ and $\F := F_\alpha$ on $\P$.
The bundle $\E$ is sometimes called the ``universal Tschirnhausen bundle" and has degree $g + k - 1$ on the fibers of $\pi: \P \to \H_{k,g}$ (see e.g. \cite[Example 3.1]{part1}).
Next, let $z := -\frac{1}{2}c_1(\omega_{\pi}) = c_1(\O_{\P}(1))$. 
For $i = 1, \ldots, k-1$, we define classes $a_i \in A^i(\H_{k,g})$ and $a_i' \in A^{i-1}(\H_{k,g})$ by the formula
\begin{equation} \label{adef}
a_i := \pi_*(z \cdot c_i(\E)), \quad a_i' := \pi_*(c_i(\E)) 
\qquad \Rightarrow \qquad 
c_i(\E) = \pi^* a_i + \pi^* a_i' z.
\end{equation}
Similarly, we define 
\begin{equation} \label{bdef}
b_i := \pi_*(z \cdot c_i(\F)), \quad b_i' := \pi_*(c_i(\F)) 
\qquad \Rightarrow \qquad 
c_i(\F) = \pi^* b_{i} + \pi^* b_{i}' z .
\end{equation}
Finally, we set $c_2 := c_2(\pi_*\O_{\P}(1)) \in A^*(\H_{k,g})$, so $z^2 = -\pi^*c_2 \in A^2(\P)$.
\begin{definition}\label{CEring}
We define $a_i, a_i', b_i, b_i', c_2$ to be the \emph{Casnati--Ekedahl (CE) classes}.
\end{definition}
Note that the CE classes generate all $\pi_*$'s of polynomials in $z$ and the Chern classes of $\E$ and $\F$.
In \cite[Theorem 3.10]{part1}, we proved that the CE classes (together with some suitable generalizations when $k > 5$) are generators for the tautological ring.
\begin{lem} \label{CEinR}
The Casnati--Ekedahl classes lie in the tautological ring $R^*(\H_{k,g})$. Conversely, when $k = 4, 5$, every tautological class is a polynomial in the above Casnati--Ekedahl classes.
\end{lem}
Furthermore, we proved in \cite[Lemmas 5.3 and 5.11]{part1} that the CE classes are generators for the entire Chow ring of a certain open substack of $\H_{k,g}$ when $k=4,5$.
\begin{prop}\label{Hprime}
Let $g\geq 2$ be an integer. Then the following hold:
\begin{enumerate}
    \item The Chow ring of $\Psi = \H_{4,g}\smallsetminus \Supp R^1\pi_*(\F^{\vee}\otimes \Sym^2 \E)$ is generated by the restrictions of CE classes.
    \item The Chow ring of $\Psi = \H_{5,g}\smallsetminus \Supp R^1\pi_*(\wedge^2 \F\otimes \E \otimes \det \E^{\vee})$ is generated by the restrictions of CE classes.
\end{enumerate}
\end{prop}

\begin{rem}
Combining Theorem \ref{pushforward}, Lemma \ref{CEinR}, and Proposition \ref{Hprime}, if we knew that $\Supp R^1\pi_*(\F^\vee \otimes \Sym^2 \E)$ were contained in $\beta^{-1}(\M_g^3)$ and $\Supp R^1\pi_*(\wedge^2 \F \otimes \E \otimes \det \E^\vee)$ were contained in $\beta^{-1}(\M_g^4)$, we would be done. However, as we shall see, this is not the case (except when $k = 5, g = 7$, which seems mostly a coincidence).
\end{rem}

\section{Splitting Loci}
\label{slsec}
 Every vector bundle $E$ on $\pp^1$ splits as a direct sum of line bundles, $E \cong \O(e_1) \oplus \cdots \oplus \O(e_r)$. We call the tuple of integers $\vec{e} = (e_1, \ldots, e_r)$ with $e_1 \leq \cdots \leq e_r$ the \emph{splitting type} of $E$ and abbreviate the corresponding sum of line bundles by $\O(\vec{e}) := \O(e_1) \oplus \cdots \oplus \O(e_r)$.
 
 If $E$ is a vector bundle on a $\pp^1$ bundle $\pi: P \to B$, then the base $B$ is stratified by locally closed subvarieties called \emph{splitting loci}
 \[\Sigma_{\vec{e}}(E) := \{b \in B : E|_{\pi^{-1}(b)} \cong \O(\vec{e})\}.\]
The above equation describes splitting loci set-theoretically. Below, we give a moduli-theoretic interpretation.
Though not necessary here, equations giving a subscheme structure to $\Sigma_{\vec{e}}(E) \subset B$ in terms of Fitting supports can be found in \cite[Section 2]{L}.
 
Suppose $W$ is a rank $2$ vector bundle with trivial determinant.
We say that a vector bundle $E$ on $\pi: \pp W \to B$ is a \emph{family of vector bundles of splitting type $\vec{e}$} if $B$ admits a cover $U_i$ so that:
\begin{itemize}
    \item there exist isomorphisms $\psi_i: W|_{U_i} \cong \mathbb{A}^2 \times U_i$, linear in $\mathbb{A}^2$ (and therefore $\pi^{-1}(U_i) \cong U_i \times \pp^1$).
    \item there exist isomorphisms $\phi_i: E|_{\pi^{-1}(U_i)} \cong q_i^* \O(\vec{e})$, where $q_i: \pi^{-1}(U_i) \cong U_i \times \pp^1 \to \pp^1$ is the composition of the isomorphism above with the second projection.
\end{itemize}
This gives rise to gluing data on the overlaps which satisfy a cocycle condition on the triple overlaps.
The data of the vector bundle $W$ is equivalent to the data of a principal $\SL_2$ bundle. A family of vector bundles of splitting type $\vec{e}$, is equivalent to the data of:
\begin{itemize}
    \item transition functions for $W$ over $U_i \cap U_j$, i.e. maps $\psi_{ij}: U_i \cap U_j \to \SL_2$ satisfying the cocycle condition $\psi_{ik} = \psi_{ij} \circ \psi_{jk}$ on $U_i \cap U_j \cap U_k$
    \item transition functions for $E$ over $U_i \cap U_j$, i.e. maps $\phi_{ij} : U_i \cap U_j \to \Aut(\O(\vec{e}))$ such that when restricted to the triple overlap $U_i \cap U_j \cap U_k$ we have
    $\phi_{ik} = \phi_{ij} \circ (\psi_{ij} \cdot  \phi_{jk})$ where $\psi_{ij}$ acts on $\phi_{jk}$ by change of coordinates (made precise below).
\end{itemize}
The action of $\SL_2$ on $\Aut(\O(\vec{e}))$ that arises above can be described concretely as follows.
We have
\[\Aut(\vec{e}) := \Aut(\O(\vec{e})) \subset H^0(\pp^1, \E nd(\O(\vec{e})) = \bigoplus_{i, j} H^0(\pp^1, \O_{\pp^1}(e_j - e_i)).\]
We let $\SL_2$ act on a factor $H^0(\pp^1, \O(e_j - e_i))$ via the $(e_j - e_i)$th symmetric power of the standard representation (if $e_j - e_i < 0$ then this cohomology group is $0$).
The cocycle conditions above are described by multiplication in the semidirect product $\SL_2 \ltimes \Aut(\vec{e})$.

By this discussion, a family of vector bundles of splitting type $\vec{e}$ over $B$ determines a principal $\SL_2 \ltimes \Aut(\vec{e})$ bundle on $B$ and vice versa. In other words, the universal $\vec{e}$ splitting locus is the classifying stack $B(\SL_2 \ltimes \Aut(\vec{e}))$. 
Let us write $\pi: \P \to B(\SL_2 \ltimes \Aut(\vec{e}))$ for the universal $\pp^1$ bundle (which is pulled back from $\BSL_2$),
and let $\V(\vec{e})$ denote the universal vector bundle of splitting type $\vec{e}$ on $\P$.
 
 Suppose that $\vec{e} = (e_1, \ldots, e_r)$ consists of distinct degrees $d_1 < \cdots < d_s$ and that $d_i$ occurs with multiplicity $n_i$.
 Then, we have
 \[\Aut(\vec{e}) = \prod_{i=1}^s \GL_{n_i} \ltimes \prod_{i < j} H^0(\pp^1, \O_{\pp^1}(d_j - d_i))^{\oplus (n_i n_j)}. \]
Elements of $\Aut(\vec{e})$ can be represented by block upper triangular matrices where the off diagonal entries are polynomials of the specified degrees on $\pp^1$.

 The $\SL_2$ action is trivial on the block diagonal matrices (the product of $\GL_{n_i}$ subgroup). It follows that 
 \begin{equation} \label{sdir}
 \SL_2 \ltimes \Aut(\vec{e}) \cong \left(\SL_2 \times \prod_{i=1}^s \GL_{n_i}\right) \ltimes \prod_{i < j} H^0(\pp^1, \O(d_j - d_i))^{\oplus n_in_j}.
 \end{equation}
Hence, we have a map $\SL_2 \ltimes \Aut(\vec{e}) \to \prod_{i=1}^s \GL_{n_i}$.
Let $\N_i$ on $B(\SL_2 \ltimes \Aut(\vec{e}))$  be the pullback of the tautological rank $n_i$ vector bundle from the $\BGL_{n_i}$ factor.
 The Harder-Narasimhan filtration on the restriction of $\V(\vec{e})$ to each fiber of $\P \to B$ induces a filtration of $\V(\vec{e})$ where the successive quotients are $(\pi^* \N_i)(d_i)$. We call this the \emph{HN filtration} of $\V(\vec{e})$ and we call the bundles $\N_i$ on $B (\SL_2 \ltimes \Aut(\vec{e}))$ the \emph{HN bundles for $\V(\vec{e})$}.

 Meanwhile, we also have an inclusion $\SL_2 \times \prod_{i=1}^s \GL_{n_i} \to \SL_2 \ltimes \Aut(\vec{e})$. This induces a map $\varphi: \BSL_2 \times \prod_{i=1}^s \BGL_{n_i} \to B (\SL_2 \ltimes \Aut(\vec{e}))$, which by \eqref{sdir} is an affine bundle.
 The pullback $\varphi^*\N_i$ is again the tautological rank $n_i$ bundle coming from the $\BGL_{n_i}$ factor. 
We have the fiber diagram
\begin{center}
\begin{tikzcd}
\P' \arrow{d}[swap]{\pi'} \arrow{r}{\varphi'} & \P \arrow{d}{\pi} \\
\BSL_2 \times \prod_{i=1}^s \BGL_{n_i} \arrow{r}[swap]{\varphi} & B (\SL_2 \ltimes \Aut(\vec{e})).
\end{tikzcd}
\end{center}
We note in passing that the pullback $\varphi'^* \V({\vec{e}})$ on $\P'$ actually splits as a direct sum
 \[\varphi'^* \V(\vec{e}) \cong \bigoplus_{i=1}^s 
 \varphi'^*(\pi^* \N_i)(d_i).\]
 Since $\varphi$ is a vector bundle map, it induces an isomorphism on Chow. This establishes the following.
 
 \begin{lem}
 The Chow ring of $B(\SL_2 \ltimes \Aut(\vec{e}))$ is the free $\zz$ algebra on the universal $c_2$ pulled back from $\BSL_2$ and the Chern classes of the HN bundles $\N_1, \ldots, \N_s$.
 \end{lem}
 \begin{rem}
 The statement above holds with $\zz$-coefficients. We will only use it, however, with $\qq$-coefficients. 
 \end{rem}
 
 The above argument works just as well for a pair of splitting types.
 
 \begin{lem} \label{pair}
 The Chow ring of $B( \SL_2 \ltimes (\Aut(\vec{e}) \times \Aut(\vec{f})) )$ is the free $\zz$ algebra on the universal $c_2$ pulled back from $\BSL_2$, the Chern classes of the HN bundles for $\V(\vec{e})$, and Chern classes of the HN bundles for $\V(\vec{f})$.
 \end{lem}

From our description of the universal $\vec{e}$ splitting locus, one sees that its codimension in the moduli stack of vector bundles on $\pp^1$ bundles is
$h^1(\pp^1, \E nd(\O(\vec{e})))$. Given a family of vector bundles on $\pp^1$ bundles with splitting type $\vec{e}$, we say that $h^1(\pp^1, \E nd(\O(\vec{e})))$ is the \emph{expected codimension}.
It follows that, if non-empty, the codimension of a splitting locus is always at most the expected codimension.
There is a partial ordering on splitting types defined by $\ee \leq \vec{e}$ if $e_1' + \ldots + e_j' \leq e_1 + \ldots + e_j$ for all $j$.
On the moduli space of vector bundles on $\pp^1$ bundles, the $\ee$ splitting locus is in the closure of the $\vec{e}$ splitting locus if and only if $\ee \leq \vec{e}$. (Of course, this need not be the case in every family, as we shall see.) Since codimension can only decrease under pullback, this implies the following fact
\begin{equation} \label{clofact}
\text{Every component of } \bigcup_{\ee \leq \vec{e}} \Sigma_{\vec{e}}(E) \text{ has at most the expected codimension.}
\end{equation}
We note that the union $\bigcup_{\ee \leq \vec{e}} \Sigma_{\vec{e}}(E)$ may not be the closure of $\Sigma_{\vec{e}}(E)$, but it is always closed in the base.

\begin{definition}[Stratifications]
Throughout this paper a \emph{stratification} of $B$ shall mean a disjoint union $B = \bigsqcup_{S \in \mathcal{S}} S$ into locally closed subvarieties (or substacks) equipped with a partial ordering $S' \leq S$ such that for each $S \in \mathcal{S}$, the union $\bigcup_{S' \leq S} S'$ is closed in $B$.
\end{definition}

\begin{example}[Warning] \label{wex}
Our notion of stratification is weaker than some in the literature. For example, say $B$ is the union of the two coordinate axes $B = V(xy) \subset \Spec  k[x,y] \cong \mathbb{A}^2$.
Then $\mathcal{S} = \{V(y), V(x) \smallsetminus (0, 0)\}$ is a stratification of $B$ with partial order $V(y) \leq V(x) \smallsetminus (0, 0)$. We represent this partial order diagramatically as pictured on the right.
\begin{center}
\begin{tikzpicture}
\draw[thick,color = blue] (0, -2) -- (0, -.05);
\draw[thick,color = blue] (0, 2) -- (0, .05);
\node[color = blue] at (0, 2.3) {$V(x) \smallsetminus (0, 0)$};
\draw[thick, color = red] (-2, 0) -- (2, 0);
\node[color = red] at (-2.55, 0) {$V(y)$};

\node[color = blue] at (9, 1) {$V(x) \smallsetminus (0, 0)$};
\node[color = red] at (9, -1) {$V(y)$};
\draw (9, -.6) -- (9, .6);
\end{tikzpicture}
\end{center}

\end{example}

We shall make use of the following key result from \cite{L}.

\begin{thm}[Theorem 1.2 of \cite{L}] \label{degenformulas}
Let $E$ be a vector bundle on a $\pp^1$ bundle $\pi : P \to B$.
Suppose that $\Sigma_{\vec{e}}(E)$ occurs in the expected codimension. Then, modulo classes supported on $\Sigma_{\ee}(E)$ for $\ee < \vec{e}$, the fundamental class of the closure of $\Sigma_{\vec{e}}(E)$ is given by a universal formula in terms of the Chern classes of $\pi_* \O_{P}(1), \pi_*E(m-1)$ and $\pi_*E(m)$ for some $m$ suitably large.
\end{thm}

Applying this to the universal CE bundles, we obtain the following.
 \begin{lem} \label{sigmaCE}
 Let $\E$ and $\F$ be the universal CE bundles on $\P \to \H_{k,g}$. If $\Sigma_{\vec{e}}(\E)$ occurs in the expected codimension, then, modulo classes supported on $\Sigma_{\ee}(\E)$ for $\ee < \vec{e}$ the fundamental class of its closure is expressible in terms of CE classes. The analogous statement holds for the classes of $\Sigma_{\vec{f}}(\F)$.
 \end{lem}
 \begin{proof}
Recall that the class $c_2 = c_2(\pi_* \O_{\P}(1))$ is a CE class by definition (and $c_1(\O_{\P}(1)) = 0$ because we are working over $\BSL_2$). 
 For $m$ suitably large, we have  $R^1\pi_* \E(m) = 0$ and $R^1\pi_* \E(m-1) = 0$, so
 by Grothendieck--Riemann--Roch, the Chern classes of $\pi_*\E(m)$ and $\pi_*\E(m-1)$ are polynomials in the CE classes. (See \cite[Example 3.4]{part3} for an example in codimension $1$). Similarly, the Chern classes of
 $\pi_*\F(m)$ and $\pi_*\F(m-1)$ are polynomials in the CE classes for $m$ suitably large.
 The result now follows from Theorem \ref{degenformulas}.
 \end{proof}

 We shall need a slight variant of the above lemma concerning particular unions of splitting loci (some of which will be allowed to occur in the wrong codimension).
Let us define
\[\Sigma_{(n, *, \ldots, *)}(E) := \bigcup_{e_1 = n} \Sigma_{\vec{e}}(E)\]
In \cite[Lemma 5.1]{L}, it was shown that if the above union occurs in its expected codimension, equal to $\deg(\vec{e}) + 1 - (n+1)r$, then --- modulo classes supported on $\Sigma_{(n', *, \ldots, *)}(E)$ for $n' < n$ --- its fundamental class can be computed with the Porteous formula. In particular, it is expressible in terms of the Chern classes of $\pi_* E(m-1), \pi_*E(m)$ and $\pi_* \O_{P}(1)$ for $m$ suitably large. Arguing as in Lemma \ref{sigmaCE}, we obtain the following result.

\begin{lem} \label{2star}
Suppose that every component of the union $\Sigma_{(n,*, \ldots, *)}(\E)$ occurs in codimension $\deg(\vec{e}) + 1 - (n+1)r$. 
Then, modulo classes supported on $\Sigma_{(n', *, \ldots, *)}(\E)$ for $n' < n$,
the fundamental class of the
closure of $\Sigma_{(n,*, \ldots, *)}(\E)$
is expressible in terms of CE classes.
\end{lem}
This is useful to us as illustrated in the following example.

\begin{example} \label{246ex}
The expected codimension for splitting type $(2, 4, 6)$ is $5$, but suppose that $\Sigma_{(2, 4, 6)}(\E)$ occurs in codimension $4$. Suppose $\Sigma_{(2, 5, 5)}(\E)$ also occurs in codimension $4$ and $\Sigma_{(2, 3, 7)}(\E)$ and $\Sigma_{(2, 2, 8)}(\E)$ are empty. Then, we have $\Sigma_{(2, *, *)}(\E) = \Sigma_{(2, 4, 6)}(\E) \cup \Sigma_{(2, 5, 5)}(\E)$, and every component occurs occurs in codimension $4 = 13 - 9 = \deg(\vec{e})+1 - (2+1)(3)$. Thus, the above lemma shows that the fundamental class of the union $\Sigma_{(2, 4, 6)}(\E) \cup \Sigma_{(2, 5, 5)}(\E)$ is expressible in terms of CE classes (modulo classes supported on $\Sigma_{(n', *, *)}(\E)$ for $n' < 2$).
\end{example}

\subsection{Pair splitting loci on $\H_{4,g}$}
Let $\E$ and $\F$ be the universal CE bundles on $\P$, the universal $\p^1$-bundle on $\H_{4,g}$.
Let $\vec{e}$ be a splitting type of rank $3$ and degree $g + 3$, and let $\vec{f}$ be a splitting type of rank $2$ and degree $g + 3$. 
Each splitting locus of the form $\Sigma := \Sigma_{\vec{e}}(\E) \cap \Sigma_{\vec{f}}(\F)$ has a concrete description as a quotient stack. This description seems well-known in the literature, but with slightly different presentations (see for example \cite[p. 20]{DP}, \cite[Theorem 4.4]{CE}, and \cite[Section 3]{CDC}). Here, we outline our preferred way of thinking about this quotient, following our set-up in \cite[Section 3]{part1}.

The vector space 
\[\Phi: H^0(\pp^1, \O(\vec{f})^\vee \otimes \Sym^2 \O(\vec{e})) \xrightarrow{\sim} H^0(\pp \O(\vec{e})^\vee, \gamma^*\O(\vec{f})^{\vee}\otimes\O_{\pp \O(\vec{e})^\vee}(2))\]
parametrizes pencils of relative quadrics on the $\pp^1$ bundle $\pp \O(\vec{e})^\vee$.
Let 
\[U \subset H^0(\pp^1, \O(\vec{f})^\vee \otimes \Sym^2 \O(\vec{e}))\]
be the open subset of sections $\eta$ whose vanishing locus $V(\Phi(\eta)) \subset \pp \O(\vec{e})^\vee$ defines a smooth, irreducible quadruple cover of $\pp^1$. Considering its Hilbert polynomial, one can show that such a cover will have genus $g$. It turns out --- essentially from the Casnati--Ekedahl structure theorem --- that all degree $4$, genus $g$ covers $\alpha: C \to \pp^1$ with $E_\alpha \cong \O(\vec{e})$ and $F_\alpha \cong \O(\vec{f})$ arise in this way. We make this precise below.

There is a natural action of
$\SL_2\ltimes(\Aut(\vec{e}) \times \Aut(\vec{f}))$ on $U$. Since $\vec{e}$ and $\vec{f}$ are the same degree, we have $\det \O(\vec{e}) \otimes \det \O(\vec{f})^\vee \cong \O_{\pp^1}$, so $\SL_2\ltimes(\Aut(\vec{e}) \times \Aut(\vec{f}))$ also acts on a copy of $\gg_m \subset H^0(\pp^1, \det \O(\vec{e}) \otimes \det \O(\vec{f})^\vee)$. Our discussion will be simplified slightly by considering also the framed Hurwitz space $\rho: \H_{k,g}^\dagger \to \H_{k,g}$ (see Remark \ref{sl2rem}). Let us write $\Sigma^\dagger := \rho^{-1}(\Sigma)$, so $\Sigma = [\Sigma^\dagger/\SL_2]$. This allows us to think about the quotient in two steps.

\begin{lem} \label{squo}
We have $\Sigma^\dagger \cong [(U \times \gg_m)/\Aut(\vec{e}) \times \Aut(\vec{f})]$, and therefore 
\[\Sigma = [\Sigma^\dagger/\SL_2] \cong [(U\times \gg_m)/\SL_2 \ltimes(\Aut(\vec{e}) \times \Aut(\vec{f}))].\]
\end{lem}
\begin{proof}
We shall prove the statement for the framed stratum $\Sigma^\dagger$, from which the second statement follows.
By definition, $\Sigma^\dagger \subset \H_{4,g}$ parametrizes covers $\alpha: C \to \pp^1$ such that $E_\alpha \cong \O(\vec{e})$ and $F_\alpha \cong \O(\vec{e})$. As a fiber category, the objects of $\Sigma^\dagger(S)$ for a scheme $S$ are degree $4$ covers $\alpha: C \to \pp^1 \times S$ such that
\begin{enumerate}
    \item $E_\alpha$ on $\pp^1 \times S$ is a family of vector bundles of splitting type $\vec{e}$ 
    \item $F_\alpha$ on $\pp^1 \times S$ is a family of vector bundles on splitting type $\vec{f}$.
    \item $C \to S$ is a family of smooth genus $g$ curves. 
\end{enumerate}
The morphisms in $\Sigma^\dagger(S)$ are isomorphisms of covers over $\pp^1 \times S$. The category $\Sigma^\dagger(S)$ is the subcategory of $\mathrm{Quad}(\pp^1 \times S)$ from \cite[Section 3.2]{part1} where we impose the additional conditions on the splitting types in (1) and (2), and the smoothness in condition (3).

Let $G := \Aut(\vec{e}) \times \Aut(\vec{f})$.
As explained in the start of this section, a principal $\Aut(\vec{e})$ bundle is equivalent to a family of vector bundles of splitting type $\vec{e}$ on $\pp^1 \times S$.
Via this identification, $[(U \times \gg_m)/G](S)$ is the category whose objects are tuples $(E, F, \phi, \eta)$ where
\begin{enumerate}
    \item[($1'$)] $E$ on $\pp^1 \times S$ is a family of vector bundles of splitting type $\vec{e}$
    \item[($2'$)] $F$ on $\pp^1 \times S$ is a family of vector bundles of splitting type $\vec{f}$
    \item[($3'$)] $\phi$ is an isomorphism $\det E \cong \det F$
    \item[($4'$)] $\eta$ is a global section of $F^\vee \otimes \Sym^2 E$ such that $V(\eta) \subset \pp E^\vee \to \pp^1 \times S$ is a degree $4$ cover over $S$, and the composition $V(\eta) \to S$ is a family of smooth curves.
\end{enumerate}
An arrow $(E_1, F_1, \phi_1, \eta_1)$ to $(E_2, F_2, \phi_2, \eta_2)$ is a pair of isomorphisms $\xi: E_1 \to E_2$, and $\psi: F_1 \to F_2$, such that the following diagrams commute
\begin{center}
\begin{tikzcd}
F_1 \arrow{d}[swap]{\psi} \arrow{r}{\eta_1} &\Sym^2 E_1 \arrow{d}{\Sym^2 \xi} \\
F_2 \arrow{r}{\eta_2} &\Sym^2 E_2
\end{tikzcd}
\hspace{1in}
\begin{tikzcd}
\det F_1 \arrow{r}{\phi_1} \arrow{d}[swap]{\det \psi} & \det E_1 \arrow{d} {\det \xi} \\
\det F_2 \arrow{r}{\phi_2} & \det E_2.
\end{tikzcd}
\end{center}
Thus, the category $[(U \times \gg_m)/G](S)$ is the subcategory of $\mathrm{Quad}'(\pp^1 \times S)$ from \cite[Section 3.2]{part1} where we impose the additional conditions on the splitting types in ($1'$) and ($2'$) and the smoothness in condition ($4'$).

There is a natural map $[(U \times \gg_m)/G] \to \Sigma^\dagger$ that sends a tuple $(E, F, \phi, \eta)$ over $S$ to the degree $4$ cover $V(\Phi(\eta)) \subset \pp E^\vee \to \pp^1 \times S$.
Theorem 3.6 of \cite{part1} showed that a corresponding map $\mathrm{Quad}'(\pp^1 \times S) \to \mathrm{Quad}(\pp^1 \times S)$ is an equivalence of categories. The argument there restricts to give an equivalence of the subcategories $[(U \times \gg_m)/G](S)$ and $\Sigma^\dagger(S)$.
\end{proof}

It follows from Lemma \ref{squo} that the $\vec{e}, \vec{f}$ splitting locus is irreducible of codimension
\begin{equation} \label{codimeq}
h^1(\pp^1, \E nd (\O(\vec{e}))) + h^1(\pp^1, \E nd(\O(\vec{f}))) - h^1(\pp^1, \O(\vec{f})^\vee \otimes \Sym^2 \O(\vec{e}))
\end{equation}
inside $\H_{4,g}$ (see also \cite[Remark 4.2]{DP}).
In light of Proposition \ref{Hprime}, we are primarily concerned with the $\vec{e}, \vec{f}$ splitting loci for which $h^1(\pp^1, \O(\vec{f})^\vee \otimes \Sym^2 \O(\vec{e})) \neq 0$, equivalently $2e_1 - f_2 \leq -2$.  By \eqref{codimeq} these are the pair splitting loci whose codimension is \emph{not} the sum of the expected codimensions for $\vec{e}$ and $\vec{f}$.


\subsection{Pair splitting loci on $\H_{5,g}$} 
Let $\E$ and $\F$ be the universal CE bundles on $\pi: \P \to \H_{5, g}$.
Let $\vec{e}$ be a splitting type of rank $4$ and degree $g + 4$, and let $\vec{f}$ be a splitting type of rank $5$ and degree $2(g + 4)$.
Similar to the previous subsection, we describe each splitting locus of the form $\Sigma := \Sigma_{\vec{e}}(\E) \cap \Sigma_{\vec{f}}(\F)$ as a quotient stack. Again this description is well-known, though in varying language (see for example \cite[p. 24]{DP}, \cite[Theorem 3.8]{C}). We give a presentation following our set up in \cite[Section 3]{part1}.

In degree $5$, the relevant space of section is
\begin{align*}
&\Phi: H^0(\pp^1, \O(\vec{e}) \otimes \O(-g-4) \otimes \wedge^2 \O(\vec{f})) \\
&\qquad \xrightarrow{\sim} H^0(\pp \O(\vec{e})^\vee, \O_{\pp \O(\vec{e})^\vee}(1) \otimes \gamma^*(\O_{\pp^1}(-g-4)  \otimes \wedge^2 \O(\vec{f})))
\end{align*}
Sections of the right-hand side are represented by $5 \times 5$ skew-symmetric matrices $M$ of linear forms on $\pp \O(\vec{e})^\vee$.
Given
such a matrix $M$, we write $D(M) \subset \pp \O(\vec{e})^\vee$ to mean the subscheme defined by the $4 \times 4$ Pfaffians of $M$ (see Section \ref{pent} for explicit equations in coordinates). These Pfaffians correspond to the equations of the Grassmann bundle $G(2, \O(\vec{f})) \subset \pp(\wedge^2 \O(\vec{f}))$ under its relative Pl\"ucker embedding, as we now explain. A section 
\[\eta \in H:= H^0(\pp^1, \O(\vec{e}) \otimes \O(-g-4) \otimes \wedge^2 \O(\vec{f}))\]
can be viewed as a linear map $\eta: \O(\vec{e})^\vee \otimes \O(g+4) \to \wedge^2 \O(\vec{f})$. If this map is injective with locally free cokernel, then $D(\Phi(\eta)) \subset \pp \O(\vec{e})^\vee$ is the intersection of $\eta(\pp \O(\vec{e})^\vee)$ with $G(2, \O(\vec{f})) \subset \pp(\wedge^2 \O(\vec{f}))$. The Grassmann bundle $G(2, \O(\vec{f})) \subset \pp(\wedge^2 \O(\vec{f}))$ has degree $5$ and codimension $3$ in each fiber over $\pp^1$, so one expects this intersection to be a degree $5$ cover of $\pp^1$.

Let $U \subset H$ be the open subvariety of sections $\eta$ such that $D(\Phi(\eta))$ is a smooth, irreducible degree $5$ cover of $\pp^1$. 
Considering its Hilbert polynomial, one can show that such a cover will have genus $g$. It turns out --- essentially from the Casnati--Ekedahl structure theorem and further work of Casnati \cite{C} --- that all degree $5$, genus $g$ smooth covers $\alpha: C \to \pp^1$ with $E_\alpha \cong \O(\vec{e})$ and $F_\alpha \cong \O(\vec{f})$ arise in this way.

Precisely, there is a natural action of $\SL_2 \ltimes (\Aut(\vec{e}) \times \Aut(\vec{f}))$ on $U$. Since $\deg(\vec{f}) = 2\deg(\vec{e})$, we have $\det \O(\vec{e})^{\otimes 2} \otimes \det \O(\vec{f})^\vee \cong \O_{\pp^1}$, so $\SL_2 \ltimes (\Aut(\vec{e}) \times \Aut(\vec{f}))$ acts on a copy of $\gg_m \subset H^0(\pp^1, \det \O(\vec{e})^{\otimes 2} \otimes \det \O(\vec{f})^\vee)$. As in the previous subsection, we will consider the quotient in two steps. Let $\rho: \H_{5,g}^\dagger \to \H_{5,g}$ be the parametrized Hurwitz space and set $\Sigma^\dagger := \rho^{-1}(\Sigma)$ so $\Sigma = [\Sigma^\dagger/\SL_2]$.

\begin{lem} \label{squo5}
We have
 $\Sigma^\dagger \cong [(U \times \gg_m)/\Aut(\vec{e}) \times \Aut(\vec{f})]$. Therefore,
 \[\Sigma = [\Sigma^\dagger/\SL_2] \cong [(U\times \gg_m)/\SL_2 \ltimes(\Aut(\vec{e}) \times \Aut(\vec{f}))].\]
\end{lem}
\begin{proof}
The proof is very similar to Lemma \ref{squo}. There is a map 
\[[(U \times \gg_m)/\Aut(\vec{e}) \times \Aut(\vec{f})] \to \Sigma^\dagger\]
that comes from sending a section $\eta \in U$ to the associated cover $D(\Phi(\eta)) \to \pp^1$. 
The categories $\Sigma^\dagger(S)$ and $[(U \times \gg_m)/\Aut(\O(\vec{e})) \times \Aut(\O(\vec{f}))](S)$ are readily seen to be subcategories of $\mathrm{Pent}(\pp^1 \times S)$ and $\mathrm{Pent}'(\pp^1 \times S)$ respectively, defined in \cite[Section 3.3]{part1}; these two subcategories are seen to be equivalent under the equivalence given in \cite[Theorem 3.8]{part1}.
\end{proof}

It follows from Lemma \ref{squo5} that the $\vec{e}, \vec{f}$ splitting locus is irreducible of codimension
\begin{equation} \label{codimeq5}
h^1(\pp^1, \E nd (\O(\vec{e}))) + h^1(\pp^1, \E nd(\O(\vec{f}))) - h^1(\pp^1, \O(\vec{e}) \otimes \O(-g-4) \otimes \wedge^2 \O(\vec{f}))
\end{equation}
inside $\H_{5,g}$. In light of Proposition \ref{Hprime}, our primary interest will be in strata where the last term $h^1(\pp^1, \O(\vec{e}) \otimes \O(-g-4) \otimes \wedge^2 \O(\vec{f})) \neq 0$, or equivalently $e_1 + f_1 + f_2 - (g+4) \leq -2$.

\section{The Tetragonal Locus}\label{tet}
In this section, we study the stratification of $\H_{4,g}$ by the pair splitting loci of the CE bundles $\E$ and $\F$.
Given a degree $4$, genus $g$ cover $\alpha: C \to \pp^1$, we let $E = E_\alpha$ and $F = F_\alpha$ be the associated vector bundles as in Section \ref{CEsec}.
Since they are vector bundles on $\pp^1$, the bundles $E$ and $F$ split.
\begin{align*}
E &= \O(e_1) \oplus \O(e_2) \oplus \O(e_3) &\qquad &e_1 \leq e_2 \leq e_3
\intertext{and}
F &= \O(f_1) \oplus \O(f_2) &\qquad &f_1 \leq f_2.
\end{align*}
In this section, we use the roman font, $E$ and $F$, to denote vector bundles of a fixed splitting type. By slight abuse of notation, we sometimes write $E = \vec{e}$ to mean $E \cong \O(\vec{e})$.

When $C$ is not hyperelliptic, the splitting type of $E$ can be interpreted geometrically as follows: under the canonical embedding, the fibers of $\alpha$ span a $2$-plane. The union of these two planes is called the \emph{associated $3$-fold scroll.} 
The embedding $C \subset \pp E^\vee$ given by the Casnati--Ekedahl theorem is constructed so that $\O_{\pp E^\vee}(1)|_C = \omega_{\alpha} = \omega_C \otimes \alpha^*\omega_{\pp^1}^\vee$. 
Let $\gamma:  \pp E^\vee \to \pp^1$ be the structure map.
Then, the associated scroll is the image of $\pp E^\vee \to \pp^{g-1}$ via the line bundle $\O_{\pp E^\vee}(1) \otimes \gamma^* \omega_{\pp^1}$ on $\pp E^\vee$.

Meanwhile, the bundle $F$ parametrizes the pencil of relative quadrics that define $C \subset \pp E^\vee$. If $X, Y, Z$ are relative coordinates on $\pp E^\vee$ corresponding to a splitting, then the pencil is generated by
\begin{align}
p &= p_{1,1}X^2 + p_{1,2} XY + p_{2,2} Y^2 + p_{1,3} XZ + p_{2,3} YZ + p_{3,3} Z^2 \label{peq} \\
q &= q_{1,1}X^2 + q_{1,2} XY + q_{2,2} Y^2 + q_{1,3} XZ + q_{2,3} YZ + q_{3,3}, Z^2, \label{qeq}
\end{align}
where $p_{i,j}$ and $q_{i, j}$ are polynomials on $\pp^1$ of degrees
 \[\deg(p_{i,j}) = e_i + e_j - f_1 \qquad \text{and} \deg(q_{i,j}) = e_i + e_j - f_2.\]

For a stratum to be non-empty, $\vec{e}$ and $\vec{f}$ must satisfy certain constraints, which we collect below.
Considering the defining sequence \eqref{eqn:defining} of $E_\alpha$, we see that $\deg(E) = -\deg(\alpha_* \O_C) = -\chi(\alpha_* \O_C) + 4 = g + 3$. By \cite[Theorem 4.4]{CE}, one must have $\det E \cong \det F$, so
 \begin{equation} \label{totaldeg}
e_1 + e_2 + e_3 = f_1 + f_2 = g+ 3.
\end{equation}
For a cover to be irreducible, we must have $1 = h^0(C, \O_C) = h^0(\pp^1, E^\vee) + 1$. This implies $e_1 \geq 1$. An upper bound on the largest part was given in \cite[Proposition 2.6]{DP}:
\begin{equation}
e_1 \geq 1 \qquad \qquad \text{and} \qquad \qquad e_3 \leq \frac{g + 3}{2}.
\end{equation}
It is well-known (see e.g. \cite[p. 127]{S}) that
\begin{equation} \label{hypeq}
\text{$e_1 = 1$ if and only if $C$ is hyperelliptic,}
\end{equation}
in which case $\alpha$ factors as $C \xrightarrow{h} \pp^1 \xrightarrow{i} \pp^1$, where $h: C \to \pp^1$ is the hyperelliptic map and $i: \pp^1 \to \pp^1$ is a degree $2$ cover.

We now turn to the geometry of the quadrics that cut out $C$.
If $p_{1,1} = 0$ and $q_{1,1} = 0$, then $V(p, q)$ contains the section $Y = Z = 0$. Thus,
\begin{equation} \label{no0}
    \text{$p_{1,1}$ and $q_{1,1}$ cannot both be $0$} \qquad \Rightarrow \qquad 2e_1 \geq f_1.
\end{equation}
If $q_{1,1} = q_{1,2} = q_{2,2} = 0$, then the quadric $q$ is divisible by $Z$. That is, $V(q)$ is reducible, so $C$, being irreducible, must lie in one component. Then fibers of $C \to \pp^1$ would then each span a line under the canonical embedding, giving $C$ a $g^2_4$, which is impossible when $g > 3$.
\begin{equation} \label{q12van}
 \text{$q_{1,1}, q_{1,2}$ and $q_{2,2}$ cannot all be $0$} \qquad \Rightarrow \qquad 2e_2 \geq f_2.
\end{equation}
On the other hand, if $q_{1,1} = q_{1,2} = q_{1,3} = 0$, then $V(q)$ is singular all along the section $Y = Z = 0$. Therefore, in order for $C$ to be smooth, no other quadric in the pencil can vanish at any point along the section $Y = Z = 0$:
\begin{equation} \label{q1jvan}
\text{if $q_{1,1} = q_{1,2} = q_{1,3} = 0$, then $p_{1,1}$ must be non-vanishing on $\pp^1$.}
\end{equation}
In terms of splitting types this implies,
 \begin{equation} \label{conditional}
 \text{if $f_2 > e_1 + e_3$, then $f_1 = 2e_1$.}
\end{equation}

Let us write $\Psi := \H_{4,g} \smallsetminus \Supp R^1\pi_*(\F^\vee \otimes \Sym^2 \E)$.
By Proposition \ref{Hprime}, we know that $A^*(\Psi)$ is generated by tautological classes.
The complement of $\Psi$ is the union of splitting loci which satisfy $2e_1 - f_2 \leq -2$. We will therefore need some results concerning the Chow rings of locally closed strata $\Sigma_{\vec{e}}(\E) \cap \Sigma_{\vec{f}}(\F)$ for such $\vec{e}, \vec{f}$, which we prove in Section \ref{chows}.
In Sections \ref{tet7}, \ref{tet8}, \ref{tet9}, we specialize to the cases $g = 7, 8, 9$ respectively.

\subsection{Strategy} \label{strategysec} Our basic strategy will be as follows:
\begin{enumerate}
\item Use 
conditions \eqref{totaldeg} -- \eqref{conditional}
to determine the allowed pairs of splitting types $\vec{e}, \vec{f}$. The partial order on splitting types of Section \ref{slsec} induces a partial order on pairs of splitting types by $(\ee, \vec{f}') \leq (\vec{e}, \vec{f})$ if $\ee \leq \vec{e}$ and $\vec{f}' \leq \vec{f}$.
\item Starting with strata at the bottom of our $\leq$ order and working upwards, show that for each stratum outside of $\Psi$, at least one of the following is satisfied:
\begin{enumerate}
    \item the stratum is contained in $\beta^{-1}(\M_g^{3})$
    \item its fundamental class in $\H_{4,g} \smallsetminus \beta^{-1}(\M_g^{3})$ is tautological (modulo classes supported on strata below it in the partial order) \textit{and} the Chow ring of the locally closed stratum is generated by the restrictions of CE classes.
    \item the push forward of its fundamental class to $\M_g \smallsetminus \M_g^3$ is tautological \textit{and} the Chow ring of the locally closed stratum is generated by the restrictions of $\kappa_1, \kappa_2$.
\end{enumerate} 
\end{enumerate}
Case (c) will only be needed in genus $9$;  thus, in genus $7$ and $8$, we will actually establish that $A^*(\H_{4,g} \smallsetminus \beta^{-1}(\M_g^3))$ is generated by CE classes.

\begin{rem}
We note the ``trade-off" between choices (b) and (c) above. If a class is tautological on $\H_{4,g} \smallsetminus \beta^{-1}(\M_g^3)$, then its push forward to $\M_g \smallsetminus \M_g^3$ is tautological by Theorem \ref{pushforward}. On the other hand, $\kappa_1$ and $\kappa_2$ are CE classes, but need not generate all CE classes. Therefore, in (b) if we prove the stronger statement about the fundamental class, we only need the weaker statement about the Chow ring; in (c) if we only prove the weaker condition about the fundamental class, we need the stronger statement about the Chow ring. 
\end{rem}

\subsection{Chow rings of very unbalanced splitting strata} \label{chows}
In Lemma \ref{squo}, each $\vec{e}, \vec{f}$ splitting locus $\Sigma = \Sigma_{\vec{e}}(\E) \cap \Sigma_{\vec{f}}(\F)$ was described as a quotient of the form $[(U \times \gg_m)/G]$, where $G := \SL_2 \ltimes(\Aut(\vec{e}) \times \Aut(\vec{f}))$.
The quotient $[(U \times \gg_m)/G]$ is a $\gg_m$ bundle over $[U/G]$, and $U \subset H:= H^0(\pp^1, \O(\vec{f})^\vee \otimes \Sym^2 \O(\vec{e}))$ is an open subvariety of affine space. Hence, there is a series of surjections
\begin{equation} \label{surjs}
A^*(BG) \twoheadrightarrow A^*([U/G]) \twoheadrightarrow A^*(\Sigma),
\end{equation}
We gave generators for $A^*(BG)$ in Lemma \ref{pair}. To show that $A^*(\Sigma)$ is generated by CE classes, it will suffice to show that the images of these generators under \eqref{surjs} can be written in terms of CE classes. Similarly, to show the stronger statement that $A^*(\Sigma)$ is generated by $\kappa_1$ and $\kappa_2$, we must show that the images of the generators of $A^*(BG)$ under \eqref{surjs} are all expressible in terms of $\kappa_1$ and $\kappa_2$. We first consider the case when $\vec{e}$ has a repeated part.

\begin{lem} \label{twoequalparts}
Let $\Sigma$ be the $\vec{e}, \vec{f}$ splitting locus and suppose $e_1 < e_2 = e_3$ and $f_1 < f_2$. Then, $A^*(\Sigma)$ is generated by the restrictions of CE classes. 
\end{lem}
\begin{proof}
Set $G := \SL_2 \ltimes (\Aut(\vec{e}) \times \Aut(\vec{f}))$ and let $\pi: \P \to BG$ be the $\pp^1$ bundle pulled back from $\BSL_2$. Let $L$ of rank $1$ and $R$ of rank $2$ be the HN bundles for $\vec{e}$ so that we have a filtration
\begin{equation} \label{Efilt}
0 \rightarrow (\pi^*R)(e_2) \rightarrow \V(\vec{e}) \rightarrow (\pi^* L)(e_1) \rightarrow 0.
\end{equation}
Similarly, let $M$ and $N$ be the rank $1$ HN bundles for $\vec{f}$ so that we have a filtration
\begin{equation} \label{Ffilt}
0 \rightarrow (\pi^* N)(f_2) \rightarrow \V(\vec{f}) \rightarrow (\pi^* M)(f_1) \rightarrow 0. 
\end{equation}

Let $r_i = c_i(R)$, and $\ell = c_1(L), m = c_1(M)$ and $n = c_1(N)$. Let $c_2$ be the second Chern class pulled back from $\BSL_2$. 
By Lemma \ref{pair}, the classes $r_1, r_2, \ell, m, n$ and $c_2$ freely generate $A^*(BG)$.
Setting $z = c_1(\O_{\P}(1))$, and using the splitting principle with \eqref{Efilt}. we obtain the identities
\begin{align} \label{c1E}
c_1(\V(\vec{e})) &= c_1(R(e_2)) + c_1(L(e_1)) = r_1 + 2e_2z + \ell + e_1z = (r_1 + \ell) + (g+3) z.
\intertext{Recalling that $z^2 = -c_2$ on $\P$, and using the splittng principle we also have}\label{c2E}
 c_2(\V(\vec{e}))&=(2e_2\ell+(e_1+e_2)r_1)z-(2e_1e_2+e_2^2)c_2+\ell r_1+r_2.
 \end{align}
 
Similarly, using the splitting principle on \eqref{Ffilt}, we obtain the identities
\begin{align} \label{c1F}
c_1(\V(\vec{f})) &= c_1(M(f_1)) + c_1(N(f_2)) = m + f_1 z + n + f_2 z = (m + n) + (g+3) z, \\
c_2(\V(\vec{f}))&=(f_2m+f_1n)z-f_2f_1c_2+mn. \label{c2F}
\end{align}

By slight abuse of notation, let us denote the images of $r_1, r_2, \ell, m, n$ and $c_2$ under the map \eqref{surjs} by the same letters. (The pullback of $c_2$ is the CE class $c_2$, as both are pulled back from $\BSL_2$.) These classes are generators for $A^*(\Sigma)$. 
By \eqref{c1E}, we have $a_1 = r_1 + \ell$. By \eqref{c2E}, we have $a_2' = e_1 r_1 + 2e_2 \ell$.
We have $e_1 < 2e_1$, so the classes $r_1$ and $\ell$ are expressible in terms of $a_1$ and $a_2'$. Next, \eqref{c2E} shows $a_2 = r_2 + r_1 \ell -(2e_1e_2+e_2^2)c_2$, so $r_2$ is also expressible in terms of CE classes. Finally, $b_1 = m + n$ by \eqref{c1F} and $b_2' = f_2 m + f_1 n$ by \eqref{c2F}, so $m$ and $n$ are expressible in terms of $b_1$ and $b_2'$ because $f_1<f_2$. Hence, the CE classes generate $A^*(\Sigma)$.
\end{proof}

Now we consider the case when all parts of $\vec{e}$ are distinct. 
The proof follows a similar set up, but requires that we also make use of some relations among the generators of $A^*(BG)$ when pulled back to $A^*([U/G])$, i.e. that the first map $v^*$ in \eqref{surjs} has a kernel. The classes in the kernel come from considering the complement of $U \subset H := H^0(\pp^1, \O(\vec{f})^\vee \otimes \Sym^2 \O(\vec{e}))$, which corresponds equations whose vanishing locus in $\pp E^\vee$ fails to be a smooth, irreducible curve. The second map in \eqref{surjs} also has a kernel. By Lemma \ref{squo}, we have that $\Sigma \to [U/G]$ is the $\gg_m$ bundle associated to the line bundle $\pi_*(\det \V(\vec{e}) \otimes \det\V(\vec{f})^\vee)$. Thus, by a theorem of Vistoli, the kernel of $A^*([U/G]) \to A^*(\Sigma)$ is
generated by $c_1(\pi_*(\det \V(\vec{e}) \otimes \det\V(\vec{f})^\vee)) = a_1 - b_1$.

\begin{lem} \label{distinctparts}
Let $\Sigma$ be the $\vec{e}, \vec{f}$ splitting locus and suppose $e_1 < e_2 < e_3$ and $f_1 < f_2$ and $2e_1 < f_2$. Then the following are true:
\begin{enumerate}
    \item If $2e_1 = f_1$, then $A^*(\Sigma)$ is generated by the restrictions of CE classes. \label{dp1}
    \item If $2e_1 = f_1$, and $e_1 + e_2 < 2e_2 = f_2$, then $A^*(\Sigma)$ is generated by $\kappa_1$ and $\kappa_2$. \label{for8}
    \item $\mathrm{(i)}$ If $2e_1 > f_1$, and $e_1 + e_2 < e_1 + e_3 = 2e_2 = f_2$, then $A^*(\Sigma)$ is generated by restrictions of CE classes.  \\
   $\mathrm{(ii)}$ Furthermore, if we also have $g \neq 9 - f_1$, then $A^*(\Sigma)$ is generated by $\kappa_1$ and $\kappa_2$. \label{for6}
\end{enumerate}
\end{lem}

\begin{proof}
Set $G = \SL_2 \ltimes (\Aut(\vec{e}) \times \Aut(\vec{f}))$ and let $\pi: \P \to BG$ be the $\pp^1$ bundle pulled back from $\BSL_2$ as before.
Let $L, S, T$ be the rank $1$ HN bundles on $BG$ for $\vec{e}$,
so that $\V(\vec{e})$ is filtered by $(\pi^*L)(e_1), (\pi^* S)(e_2),$ and $(\pi^*T)(e_3)$.
Similarly, let $M$ and $N$ be the rank $1$ HN bundles on $BG$ for $\vec{f}$ so that $\V(\vec{f})$ is filtered by $(\pi^*M)(f_1)$ and $(\pi^*N)(f_2)$.
Let $s = c_1(S), t = c_1(T), \ell = c_1(L), m = c_1(M)$ and $n = c_1(N)$. 
By Lemma \ref{pair}, the classes $s, t,\ell, m, n$ and $c_2$ freely generate $A^*(BG)$.

Using the splitting principle (and omitting $\pi$ pullbacks) as in Lemma \ref{twoequalparts}, we have
\begin{equation*} c_1(\V(\vec{e})) = c_1(L(e_1)) + c_1(S(e_2)) + c_1(T(e_3)) = (\ell + s + t) + (g+3) z.
\end{equation*}
Recalling that $z^2 = -c_2$ on $\P$, we also have
\begin{align*} 
    c_2(\V(\vec{e}))=((e_2+e_3)\ell+(e_1+e_3)t+(e_1+e_2)s)z-(e_1e_2+e_1e_3+e_2e_3)c_2+\ell (t+s)+ts
\end{align*}
Thus, we have 
\begin{equation} \label{as} a_1 = \ell + s + t \qquad\text{and} \qquad a_2' = (e_2 + e_3)\ell + (e_1+e_3)s + (e_1+e_2)t.
\end{equation}
Similarly, using the splitting principle, the Chern classes of $\V(\vec{f})$ satisfy the same identities as in \eqref{c1F} and \eqref{c2F}, so
\begin{equation} \label{bs}b_1 = m + n \qquad \text{and} \qquad b_2' = f_2m + f_1n.
\end{equation}

Notice that $A^*(BG)$ has $5$ generators in codimension $1$, but there are only $4$ codimension $1$ CE classes (namely $a_1, a_2', b_1, b_2'$). Thus, to have any hope of the CE classes generating $A^*(\Sigma)$, the first map \eqref{surjs} must have some kernel in codimension $1$. 
In each of the cases below, we describe one, or two such relations.

(1) Assume that $2e_1 = f_1$. Corresponding to our filtration of $\V(\vec{e})$, there is a quotient $\Sym^2 \V(\vec{e}) \to (\pi^* L)^{\otimes 2}(2e_1)$. Tensoring with $\V(\vec{f})^\vee$, we obtain a quotient
\begin{equation} \label{firstq}
\V(\vec{f})^\vee \otimes \Sym^2 \V(\vec{e}) \to \V(\vec{f})^\vee \otimes (\pi^*L)^{\otimes 2}(2e_1).
\end{equation}
Using our filtration of $\V(\vec{f})^\vee$, we see that the right hand term above is filtered by the line bundles $\pi^*(N^\vee \otimes L^{\otimes 2})(2e_1 - f_2)$ and $\pi^*(M^\vee \otimes L^{\otimes 2})(2e_1 - f_1) = \pi^*(M^\vee \otimes L^{\otimes 2})$. Noting that $2e_1 - f_2 < 0$, cohomology and base change then shows that the push forward of the right hand term of \eqref{firstq} is $\pi_*(\V(\vec{f})^\vee \otimes (\pi^*L)^{\otimes 2}(2e_1)) \cong M^\vee \otimes L^{\otimes 2}$.
Applying $\pi_*$ to \eqref{firstq}, we therefore obtain a surjection
\[\pi_*(\V(\vec{f})^\vee \otimes \Sym^2 \V(\vec{e})) \to M^\vee\otimes L^{\otimes 2}.\]
The total space of $\pi_*(\V(\vec{f})^\vee \otimes \Sym^2 \V(\vec{e}))$ is simply $[H^0(\pp^1, \O(\vec{f})^\vee \otimes \Sym^2 \O(\vec{e}))/G]$. In the notation of \eqref{peq}, the above surjection corresponds to projection onto the coefficient $p_{1,1}$. 
Since $2e_1 - f_2 < 0$, the coefficient $q_{1,1} = 0$. Thus by \eqref{no0}, we must have $p_{1,1} \neq 0$.
That is, $U$ lies in the complement of the kernel of this projection. Put differently, writing $v: [U/G] \to BG$ for the map to the base, the line bundle $v^*(M^\vee \otimes L^{\otimes 2})$ has a non-vanishing section on $[U/G]$. Hence, we have the relation
\begin{equation} \label{extra}
    0 = v^*(2\ell - m).
\end{equation}
We collect \eqref{as}, \eqref{bs} and \eqref{extra} into a $5 \times 5$ matrix equation in  $A^1(\Sigma)$:
\[
\left( \begin{matrix} a_1 \\ a_2' \\ b_1 \\ b_2' \\ 0 \end{matrix}\right) =
\left(\begin{matrix}
1 & 1 & 1 & 0 & 0 \\
e_2 + e_3 & e_1 + e_3 & e_1 + e_2 & 0 & 0 \\
0 & 0 & 0 & 1 & 1 \\
0 & 0 & 0 & f_2 & f_1 \\
2 & 0 & 0 & -1 & 0
\end{matrix}\right) \left(\begin{matrix}
\ell \\ s \\ t \\ m \\ n \end{matrix}\right).
\]
The matrix of coefficients above has determinant $2(e_3 - e_2)(f_2 - f_1)$.
Since $e_2 \neq e_3$ and $f_1 \neq f_2$, the matrix is invertible, so the images of the classes $\ell, s, t,m, n$ are expressible in terms of the CE classes $a_1, a_2', b_1, a_2'$. The images of $\ell, s, t,m, n$ and $c_2$ under \eqref{surjs} generate $A^*(\Sigma)$. Hence, $A^*(\Sigma)$ is generated by CE classes. 

(2) Suppose further that $2e_1 = f_1$, and $e_1 + e_2 < 2e_2 = f_2$.
Because $2e_1 < e_1 + e_2 < f_2$, both $q_{1,1}$ and $q_{1,2}$ are zero. By \eqref{q12van}, $q_{2,2}$ must be nonzero. Since $2e_2 = f_2$, the coefficient $q_{2,2}$ is degree $0$, so its vanishing is a codimension $1$ condition. Using an argument similar to the above, this gives rise to a non-vanishing global section of $v^*(S^{\otimes 2} \otimes N^\vee)$ on $[U/G]$. Hence, we obtain the relation $v^*(2 s - n) = 0$. 
As in (1) we still have the relation $v^*(2\ell - m) = 0$. 
Meanwhile, 
we also know of some relations among CE classes that hold in $A^1(\H_{4,g})$. First off, we have $0 = a_1 - b_1$, which corresponds to $\Sigma \to [U/G]$ being a $\gg_m$ bundle associated to a line bundle with first Chern class $a_1 - b_1$. Second, we have $0 = (8g + 20) a_1 - 8a_2' - b_2'$, by \cite[Lemma 5.4]{part3}, corresponding to the fundamental class of $\Delta := H^0(\pp^1, \O(\vec{f})^\vee \otimes \Sym^2 \O(\vec{e})) \smallsetminus U$. (Our other relations above $2\ell - m = 2s - n = 0$ came from certain components of $\Delta$).
Finally, by \cite[Lemma 7.6]{part2}, we have $\kappa_1 = (12g + 24)a_1 - 12a_2'$. Using \eqref{as} and \eqref{bs} to rewrite these in terms of $s, t, \ell, m, n$ (and that $e_1 + e_2 = e_3 = f_1 + f_2 = g+3$), we lay out a $5 \times 5$ matrix summarizing these relations that hold in $A^1(\Sigma)$:
\begin{equation} \label{bigmat}
\left(\begin{matrix}
 0 \\
 0 \\
 0 \\
 0 \\
 \kappa_1
\end{matrix} \right)=
\left(\begin{matrix}
2 & 0 & 0  & 0 & -1 \\
0 & 0 & 2 & -1 & 0 \\
1 & 1 & 1 & - 1 & -1 \\
8e_2 + 44 & 8 e_3 + 44 & 8e_1 + 44 & -f_2 & -f_1 \\
12e_2 + 60 & 12e_3 + 60 & 12 e_1 + 60 & 0 & 0
\end{matrix} \right)
\left(\begin{matrix} s \\ t \\ \ell \\ m \\ n \end{matrix}\right).
\end{equation}
Taking into account $f_2 = 2e_2$ and $f_1 = 2e_1$, the determinant of the above matrix is equal to $48(g+5)(e_2 - e_1) \neq 0$, so $\kappa_1$ is a generator for $A^1(\Sigma)$. 

Next, we want to show that the entire ring $A^*(\Sigma)$ is generated by $\kappa_1$ and $\kappa_2$. Because $A^*(\Sigma)$ is generated in codimension $1$ and $2$, it suffices to show that $A^2(\Sigma)$ is generated by $\kappa_2$, together with products of codimension $1$ classes. This in turn follows if we can write $c_2$ in terms of products of codimension $1$ classes and $\kappa_2$. Such an identity in fact holds in $A^2(\H_{4,g})$, as implied by \cite[Example 7.8]{part2}.
Precisely, combining \cite[Equations (7.3) and (7.4)]{part2}, we see
\begin{equation}\label{c2k2}
c_2 = -24(2g^3-32g^2+138g-12) \kappa_2 + \text{products of codimension $1$ classes} \in A^2(\H_{4,g}).
\end{equation}

(3) Now we assume that $2e_1 > f_1$ and $e_1 + e_2 < e_1 + e_3 = 2e_2 = f_2$. Since $e_1 + e_2 < f_2$, we have $q_{1,1} = q_{1,2} = 0$, so by \eqref{q12van}, we must have $q_{2,2} \neq 0$. Since $\deg(q_{2,2}) = 2e_2 - f_2 = 0$, this is also a codimension $1$ condition. The coefficient $q_{2,2}$ is corresponds to a non-zero section of $v^*(S^{\otimes 2} \otimes N^\vee)$, so we obtain the relation $v^*(2s - n) = 0$ as in (2). Collecting \eqref{as}, \eqref{bs}, and the relation $v^*(2s-n)=0$ in a matrix equation, we have
\[
\left( \begin{matrix} a_1 \\ a_2' \\ b_1 \\ b_2' \\ 0 \end{matrix}\right) =
\left(\begin{matrix}
1 & 1 & 1 & 0 & 0 \\
e_2 + e_3 & e_1 + e_3 & e_1 + e_2 & 0 & 0 \\
0 & 0 & 0 & 1 & 1 \\
0 & 0 & 0 & f_2 & f_1 \\
0 & 2 & 0 & 0 & -1
\end{matrix}\right) \left(\begin{matrix}
\ell \\ s \\ t \\ m \\ n \end{matrix}\right).
\]
The determinant of the above matrix is 
$-2(e_3-e_1)(f_2-f_1)$. Because $e_3>e_1$ and $f_2>f_1$, this determinant does not vanish, and so the codimension $1$ generators are expressible in terms of CE classes. This completes the proof of (3)(i).

To show (3)(ii), we will need to produce more relations.
Because $2e_1 > f_1$, the coefficient $q_{1,1}$ is a polynomial of positive degree on $\pp^1$, in particular it must vanish somewhere. Thus, by \eqref{q1jvan}, one of $q_{1,1}, q_{1,2}, q_{1,3}$ must be nonzero. However, we know $q_{1,1} = q_{1,2} = 0$, so we must have $q_{1,3} \neq 0$. Again, $\deg(q_{1,3}) = e_1 + e_3 - f_2 = 0$, so this is a codimension $1$ condition. This coefficient of $q_{1,3}$ gives a non-zero section of $v^*(L \otimes T \otimes N^\vee)$, on $[U/G]$, so we obtain the relation $v^*(\ell + t - n) = 0$.
Now, we can just replace the second row of the matrix in \eqref{bigmat} (the top row and bottom three are still valid relations), to get an equation in $A^1(\Sigma)$:
\begin{equation}
\left(\begin{matrix}
 0 \\
 0 \\
 0 \\
 0 \\
 \kappa_1
\end{matrix} \right)=
\left(\begin{matrix}
2 & 0 & 0  & 0 & -1 \\
0 & 1 & 1 & 0 & -1 \\
1 & 1 & 1 & - 1 & -1 \\
8e_2 + 44 & 8 e_3 + 44 & 8e_1 + 44 & -f_2 & -f_1 \\
12e_2 + 60 & 12e_3 + 60 & 12 e_1 + 60 & 0 & 0
\end{matrix} \right)
\left(\begin{matrix} s \\ t \\ \ell \\ m \\ n \end{matrix}\right).
\end{equation}
The determinant of the above matrix is $-12(f_1 + g - 9)(e_3 - e_1)$, which is non-zero given the hypotheses in the lemma. Thus $\kappa_1$ generates $A^1(\Sigma)$. By \eqref{c2k2}, we see that the codimension $2$ generator $c_2$ is expressible in terms of $\kappa_2$  and $\kappa_1^2$. Hence, $A^*(\Sigma)$ is generated by $\kappa_1$ and  $\kappa_2$, as desired.
\end{proof}
 
\subsection{Genus $5$ and $6$}\label{tet56}
As a warm-up, we will explain how the argument works in genus $5$ and $6$, thus giving new proofs of the results of Izadi \cite{Iz} and Penev-Vakil \cite{PV}, who proved that the Chow ring is equal to the tautological ring in genus $5$ and $6$, respectively.

Using \eqref{totaldeg}--\eqref{conditional}, we have the following allowed pairs of splitting types in genus $5$. We label a stratum with a $\Psi_i$ if it is contained within $\Psi$ (see Proposition \ref{Hprime}), equivalently if $2e_1 - f_2 \geq -1$:

\begin{enumerate}
    \item [($\Psi_0$)] $E=(2, 3, 3), F=(4,4)$.
    \item [($\Psi_1$)] $E=(2, 3, 3), F=(3,5)$.
    \item [($\Psi_2$)] $E=(2, 2, 4), F=(3,5)$.
    \item [($Z$)] $E=(1, 3, 4), F=(2,6)$.
\end{enumerate}

\begin{prop}
The Chow ring $A^*(\H_{4,5}\smallsetminus \beta^{-1}(\M^3_5))$ is generated by tautological classes. Hence, $A^*(\M_5)$ is tautological.
\end{prop}
\begin{proof}
By \eqref{hypeq}, the stratum $Z$ consists of entirely hyperelliptic curves. Hence, $\H_{4, 5} \smallsetminus \beta^{-1}(\M_5^3)$ is contained in $\H_{4,5} \smallsetminus Z = \Psi_0 \cup \Psi_1 \cup \Psi_2 = \Psi$.
In particular, by Proposition \ref{Hprime}, we see $A^*(\H_{4,5} \smallsetminus \beta^{-1}(\M_5^3))$ is generated by tautological classes.
By Theorem \ref{pushforward}, it follows that $A^*(\M_5\smallsetminus \M^3_5)$ is generated by tautological classes. Classes supported on $\M^3_5$ are known to be tautological \eqref{Mg3}, so we conclude that $A^*(\M_5)$ is tautological.
\end{proof}

The genus $6$ case is similar. By \eqref{totaldeg}--\eqref{conditional}, we have the following pairs of splitting types:

\begin{enumerate}
    \item[($\Psi_0$)] $E = (3, 3, 3), F = (4,5)$, generic stratum.
    \item[($\Psi_1$)] $E = (2, 3, 4), F = (4,5)$, codimension $1$, with $E$ unbalanced.
    \item[($\Psi_2$)] $E = (3, 3, 3), F = (3,6)$, codimension $2$, with $F$ unbalanced.
    \item[($\Sigma_3$)] $E = (2, 3, 4), F = (3, 6)$, codimension $2$, with $E$ and $F$ unbalanced. 
    \item[($Z$)] $E= (1, 4, 4), F = (2, 7)$, codimension $2$.
\end{enumerate}

We first identify the curves of lower gonality
\begin{lem}
We have $\beta^{-1}(\M_6^3) = Z \cup \Psi_2$
\end{lem}
\begin{proof}
By \eqref{hypeq}, we already know $Z = \beta^{-1}(\M_6^2)$, so must show that $\Psi_2 = \beta^{-1}(\M_6^3 \smallsetminus \M_6^2)$. We first show $\Psi_2 \subseteq \beta^{-1}(\M_6^3 \smallsetminus \M_6^2)$. On $\Psi_2$, we have $\pp E^\vee \cong \pp^1 \times \pp^2$. Since $\deg(q_{i,j}) =0$ and $\deg(p_{i,j}) = 3$ for all $i, j$, the projection onto the $\pp^2$ factor realizes $C$ as a degree $3$ cover of a conic in $\pp^2$. To show the reverse inclusion, suppose $\sigma: C \to \pp^1$ is a trigonal curve that also admits a degree $4$ map $\alpha: C \to \pp^1$. Then $(\alpha, \sigma): C \to \pp^1 \times \pp^1$ is birational onto its image, which is a curve of bidegree $(3, 4)$. By the genus formula, the genus of the image is $6$, so $(\alpha, \sigma): C \to \pp^1 \times \pp^1$ is an embedding. Composing with the degree $2$ Veronese on the second factor, we obtain a map $C \hookrightarrow \pp^1 \times \pp^1 \hookrightarrow \pp^1 \times \pp^2$ which is an embedding of $C$ in a $\pp^2$ bundle satisfying the properties of $\pp$ in Theorem \ref{CEstructure}. By its uniqueness, we see that $\pp E_\alpha \cong \pp^1 \times \pp^2$, i.e. $E_\alpha = (3, 3, 3)$. 
Meanwhile, the bundle $F_\alpha$ corresponds to the quadrics vanishing on $C \subset \pp E_\alpha^\vee \cong \pp^1 \times \pp^2$. The curve $C$ lies on a quadric whose equation is pulled back from the $\pp^2$ factor. Writing this quadric in the form \eqref{qeq}, we see that $\deg(q_{i,j}) = 0$, so $f_2 = 6$. Hence, $F_\alpha = (3, 6)$.
\end{proof}


\begin{prop}
 The Chow ring $A^*(\H_{4,6}\smallsetminus \beta^{-1}(\M^3_6))$ is generated by tautological classes. Hence, $A^*(\M_6)$ is tautological.
\end{prop}
\begin{proof}
Working on the complement of $\beta^{-1}(\M^3_6) = Z \cup \Psi_2$, we observe that $\Sigma_3$ is the $(3,6)$ splitting locus for $\F$, i.e. $\Sigma_3 = \Sigma_{(3,6)}(\F)$. Moreover,
\[
\codim \Sigma_3=2=h^1(\p^1,\E nd(\O(3,6))).
\]
Thus, by Lemma \ref{sigmaCE}, the fundamental class of $\Sigma_3 \subset \H_{4,6} \smallsetminus \beta^{-1}(\M_6^3)$ is expressible in terms of CE classes. By Lemma \ref{distinctparts} \eqref{for6} (i), we see $A^*(\Sigma_3)$ is generated by restrictions of CE classes.
Hence, using the push-pull formula, every class supported on $\Sigma_3 \subset \H_{4,6}\smallsetminus \beta^{-1}(\M_6^3)$ is tautological.

Meanwhile, $\H_{4,6}\smallsetminus (\beta^{-1}(\M^3_6) \cup \Sigma_3) = \Psi_0\cup \Psi_1 = \Psi$ is the open subset considered in Proposition \ref{Hprime}.
Hence, $A^*(\H_{4,6}\smallsetminus (\beta^{-1}(\M^3_6) \cup \Sigma_3))$ is generated by tautological classes.
By excision and the first paragraph of the proof, all of $A^*(\H_{4,6} \smallsetminus \beta^{-1}(\M_6^3))$ is tautological.
By Theorem \ref{pushforward},  $A^*(\M_6\smallsetminus \M^3_6)$ is generated by tautological classes. Combined with \eqref{Mg3}, we obtain that $A^*(\M_6)$ is tautological.
\end{proof}

\begin{rem}
(1) We note that the stratum $\Sigma_3$ consists of plane quintic curves. Indeed, on $\Sigma_3$, we have $p_{1,1} = 0$ and $\deg(q_{1,1}) = 1$ so the curve meets the line $Y = Z = 0$ in $\pp E^\vee$ in one point, say $\nu \in C$. The canonical line bundle on $C$ is the restriction of $\O_{\pp E^\vee}(1) \otimes \omega_{\pp^1}$, which contracts the line $Y = Z = 0$ in the map $\pp E^\vee \to \pp^5$. Thus, $\nu$ is contained in each of the planes spanned by the image of a fiber of $\alpha$ under the canonical embedding. Hence, by geometric Riemann--Roch, the $g^1_4$ plus $\nu$ is a $g^2_5$. The locus of genus $6$ curves possessing a $g^2_5$ is codimension $3$ in $\M_6$, but this stratum has codimension $2$ in $\H_{4,6}$ because projection from any point on a plane quintic gives a $g^1_4$.

(2) It turns out $\Sigma_3$ in genus $6$ is the only case where Lemma \ref{distinctparts} \eqref{for6}(i) holds but $g = 9 - f_1$. 
The fact that $\Sigma_3 \to \M_6$ has positive-dimensional fibers seems to provide some geometric intuition for this exception where we fail to obtain the stronger statement in (3)(ii).
\end{rem}

\subsection{Genus $7$ tetragonal curves}\label{tet7}
Using \eqref{totaldeg} -- \eqref{conditional}, the allowed splitting types in genus $7$ are as follows. We label a stratum with a $\Psi_i$ if it is contained within $\Psi$ (see Proposition \ref{Hprime}), equivalently if $2e_1 - f_2 \geq -1$.

\begin{enumerate}
    \item[($\Psi_0$)] $E = (3, 3, 4), F = (5,5)$: generic stratum; associated scroll is smooth.
    \item[($\Psi_1$)] $E = (3, 3, 4), F = (4, 6)$: associated scroll is smooth, $F$ unbalanced.
    \item[($\Sigma_2$)] $E = (2, 4, 4), F = (4,6)$: associated scroll is a cone over $\pp^1 \times \pp^1$ (which is embedded in a hyperplane in $\pp^6$ via $\O(2, 1)$). General bielliptics live in here as a proper closed subvariety, described in \cite[Theorem 2.3]{CDC}.
    \item[($\Sigma_3$)] $E = (2, 3, 5), F = (4,6)$: the associated scroll is a cone over the Hirzebruch surface $\mathbb{F}_2$ (embedded via $\O(1)$ on $\pp(\O(1) \oplus \O(3)$). The ``special bielliptics" live in here as a proper closed subvariety, described in \cite[Theorem 2.3]{CDC}.
    \item[($Z$)] $E = (1, 4, 5), F = (2,8)$: stratum of hyperelliptic curves (see \eqref{hypeq}).
\end{enumerate}

\begin{figure}[h!]
    \centering
\begin{tikzcd}[column sep = 6pt] \\
\text{codimension} & & \text{closure order} & \\[-6pt]
0& & \Psi_0  & \\
1& & \Psi_1 \arrow[hook]{u} & \\
2& \Sigma_2 \arrow[hook]{ur} & & \arrow[hook]{ul} Z \\
3& \Sigma_3 \arrow[hook]{u}
\end{tikzcd}
\hspace{1.75in}
\begin{tikzcd}
\text{our $\leq$ order} \\[-6pt]
\Psi_0 \\
\Psi_1 \arrow[no head]{u} \\
\Sigma_2 \arrow[no head]{u}  \\
\Sigma_3 \arrow[no head]{u}  \\
Z \arrow[no head]{u} 
\end{tikzcd}
\caption{Two partial orders on the genus $7$ strata}
\label{7fig}
\end{figure}
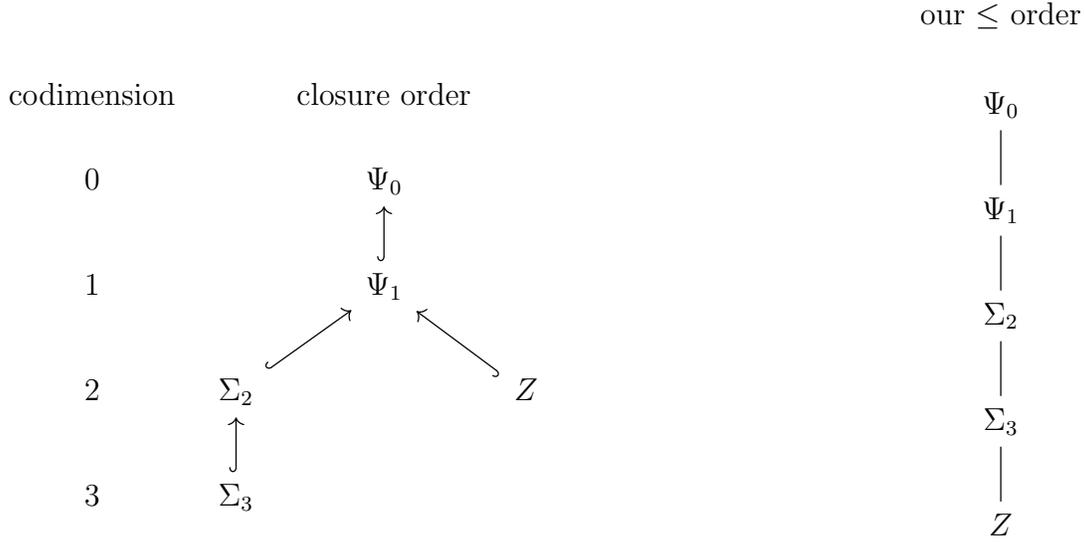

The table on the left of Figure \ref{7fig} lists the codimensions of strata (computed with \eqref{codimeq}). It also
indicates the partial order of which strata lie in the closure of others, which can be seen by considering \eqref{clofact}. This should be contrasted with the our partial ordering $\leq$, which is pictured on the right.

\begin{lem} \label{H47}
The Chow ring $A^*(\H_{4,7} \smallsetminus \beta^{-1}(\M_7^3))$ is generated by CE classes. Hence, all classes supported on $\M_7^4$ are tautological on $\M_7$.
\end{lem}
\begin{proof}
We implement Strategy \ref{strategysec}, starting at the bottom of the partial ordering. By \eqref{hypeq}, we have 
\[Z = \beta^{-1}(\M_7^2) = \beta^{-1}(\M_7^3).\]
The second equality follows because a genus $7$ curve cannot possess maps to $\pp^1$ of degrees $3$ and $4$, otherwise it would map to $\pp^1 \times \pp^1$ with image a curve of bidegree $(3, 4)$, which has genus $6$.

Next, we claim that, modulo classes supported on $Z$, the fundamental class of $\Sigma_3$ is expressible in terms of CE classes.
To see this, observe that $\Sigma_3 = \Sigma_{(2,3,5)}(\E)$. Moreover,
\[\codim \Sigma_3 = 3 = h^1(\pp^1, \E nd(\O(2, 3, 5))).\]
Thus, the claim follows from Lemma \ref{sigmaCE}. By Lemma \ref{distinctparts} \eqref{dp1}, we see $A^*(\Sigma_3)$ is generated by the restrictions of CE classes, so by the push-pull formula, every class supported on on $\Sigma_3 \subset \H_{4,7} \smallsetminus Z$ is expressible in terms of CE classes, i.e. is tautological.

Similarly, modulo classes supported on $Z$ and $\Sigma_3$, we claim that the fundamental class of $\Sigma_2$ is expressible in terms of CE classes. To see this, observe that $\Sigma_2 = \Sigma_{(2, 4, 4)}(\E)$ and
\[\codim \Sigma_2 = 2 = h^1(\pp^1, \E nd(\O(2, 4, 4))).\]
Thus, the claim again follows from Lemma \ref{sigmaCE}. By Lemma \ref{distinctparts} \eqref{dp1}, we see $A^*(\Sigma_2)$ is generated by restrictions of CE classes. Using the push-pull formula, along with the previous paragraph, we see that every class supported on $\Sigma_2 \cup \Sigma_3 \subset \H_{4,7} \smallsetminus Z$ is tautological.

By Proposition
\ref{Hprime}, we know $A^*(\Psi)$ is generated by tautological classes. Putting this together with the above, we find that $A^*(\Psi \cup \Sigma_2 \cup \Sigma_3) = A^*(\H_{4,7} \smallsetminus \beta^{-1}(\M_7^3))$ is generated by tautological classes. Applying Theorem \ref{pushforward}, every class supported on $\M_7^4 \smallsetminus \M_7^3$ is tautological in $\M_7 \smallsetminus \M_7^3$. classes supported on $\M_7^3$ are known to be tautological (see \eqref{Mg3}), so the result follows.
\end{proof}

\subsection{Genus $8$ tetragonal curves}\label{tet8}
Using \eqref{totaldeg} -- \eqref{conditional}, there are five allowed splitting types for the CE bundles, which give rise to the following stratification of $\H_{4,8}$. Again, we label a stratum with a $\Psi_i$ if it is contained within $\Psi$, equivalently if $2e_1 - f_2 \geq -1$.
\begin{enumerate}
    \item[($\Psi_0$)] $E = (3, 4, 4), F = (5,6)$: generic stratum; the associated scroll is smooth.
    \item[($\Psi_1$)] $E = (3, 4, 4), F = (4,7)$: associated scroll is smooth, $F$ unbalanced.
    \item[($\Psi_2$)] $E = (3, 3, 5), F = (5, 6)$: associated scroll is smooth, $E$ unbalanced.
    \item[($\Sigma_3$)] $E = (2, 4, 5), F = (4,7)$: associated scroll is a cone over $\mathbb{F}_1$. Bielliptic curves are a proper closed subvariety here, see \cite[Theorem 2.3]{CDC}.
    \item[($Z$)] $E = (1, 5, 5), F = (2,9)$: stratum of hyperelliptic curves (see \eqref{hypeq}).
\end{enumerate}

The table on the left of Figure \ref{8fig} lists the codimension of strata (see \eqref{codimeq}) and
indicates which strata are in the closure of others, which can be seen by considering \eqref{clofact}. This should be contrasted with our partial ordering $\leq$, which is pictured on the right.

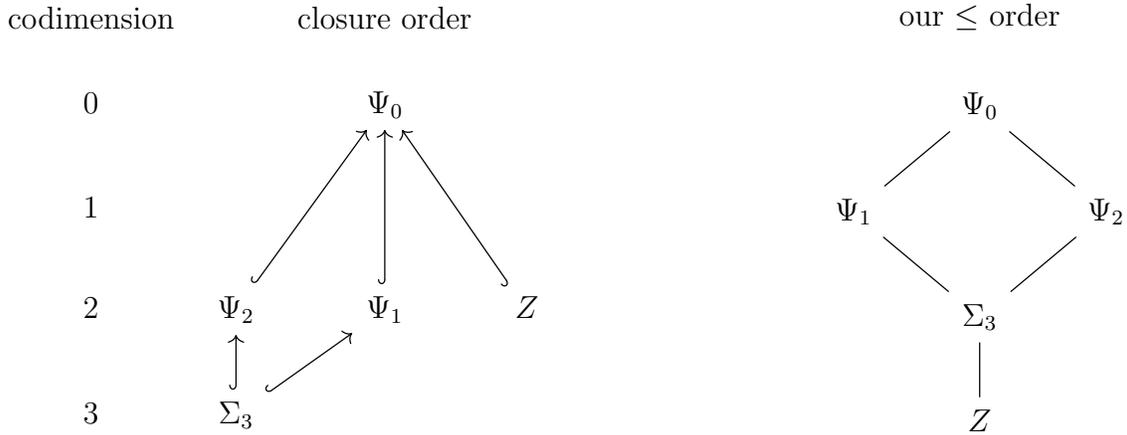
\begin{figure}[h!]
\centering
\begin{tikzcd}[column sep = 6pt]
\text{codimension} & & \text{closure order} & \\[-6pt]
0&& \Psi_0 & \\
1&&& \\
2&\Psi_2 \arrow[hook]{uur} & \Psi_1 \arrow[hook]{uu} & Z \arrow[hook]{uul} \\
3& \Sigma_3 \arrow[hook]{u} \arrow[hook]{ur}
\end{tikzcd}
\hspace{1.3in}
\begin{tikzcd}[column sep = 0pt]
& \text{our $\leq$ order} \\[-6pt]
&\Psi_0 & \\
\Psi_1 \arrow[no head]{ur} && \Psi_2 \arrow[no head]{ul}  \\
&\Sigma_3 \arrow[no head]{ur} \arrow[no head]{ul}  \\
&Z \arrow[no head]{u} 
\end{tikzcd}
\caption{Two partial orders on the genus $8$ strata}
\label{8fig}
\end{figure}

\begin{lem} \label{H48}
The Chow ring $A^*(\H_{4,8} \smallsetminus \beta^{-1}(\M_8^3))$ is generated by CE classes. Hence, all classes supported on $\M_8^4$ are tautological on $\M_8$.
\end{lem}
\begin{proof}
The proof is very similar to Lemma \ref{H47}. The lowest stratum is again the hyperelliptic locus:  $Z = \beta^{-1}(\M_8^2) = \beta^{-1}(\M_8^3)$. Then, we notice that $\Sigma_3 $ is equal to $\Sigma_{(2,4,5)}(\E)$ and
\[\codim \Sigma_3 = 3 = h^1(\pp^1, \E nd(\O(2, 4, 5))).\]
Thus, by Lemma \ref{sigmaCE}, the fundamental class of $\Sigma_3$ inside $\H_{4,8} \smallsetminus \beta^{-1}(\M_8^3)$ is tautological. By Lemma \ref{distinctparts} \eqref{dp1}, the Chow ring of the locally closed stratum $A^*(\Sigma_3)$ is generated by the restrictions of CE classes. By the push-pull formula, every class supported on $\Sigma_3 \subset \H_{4,8} \smallsetminus \beta^{-1}(\M_8^3)$ is tautological. 

Meanwhile, Proposition \ref{Hprime} shows that $A^*(\Psi) = A^*(\Psi_0 \cup \Psi_1 \cup \Psi_2)$ is generated by tautological classes. Putting this together with the previous paragraph, all of $A^*(\H_{4,8} \smallsetminus \beta^{-1}(\M_8^3))$ is generated by tautological classes. 
\end{proof}

\subsection{Genus $9$ tetragonal curves}\label{tet9}
Using \eqref{totaldeg} -- \eqref{conditional}, we find that the allowed splitting types in genus $9$ are as follows. Again, we label a stratum $\Psi_i$ if $2e_1 - f_1 \geq -1$.
\begin{enumerate}
\item[($\Psi_0$)] $E = (4, 4, 4), F = (6, 6)$: the general stratum, the associated scroll is $\pp^2 \times \pp^1$.
\item[($\Psi_1$)] $E = (4, 4, 4), F = (5, 7)$: codimension $1$, with $F$ unbalanced.
\item[($\Psi_2$)] $E = (3, 4, 5), F = (6, 6)$: codimension $1$, with $E$ unbalanced.
\item[($\Psi_3$)] $E = (3, 4, 5), F = (5, 7)$: codimension 2, both $E$ and $F$ unbalanced.

\item[($\Psi_4$)] $E = (4, 4, 4), F = (4, 8)$: codimension $3$, such curves have bidegree $(4, 4)$ on  $\pp^1 \times \pp^1$.

\item[($\Psi_5$)] $E = (3, 3, 6), F = (6, 6)$: codimension $4$.
\item[($\Sigma_6$)] $E = (3, 4, 5), F = (4, 8)$: codimension $3$, such curves posess a $g^2_6$ (Lemma \ref{hasg26}).
\item[($\Sigma_7$)] $E = (2, 5, 5), F = (4, 8)$: codimension $4$, all such covers factor through a degree $2$ cover of an elliptic curve (see Lemma \ref{bi}).
\item[($\Sigma_8$)] $E = (2, 4, 6), F = (4, 8)$: codimension $4$, the ``special bielliptics" live here as a proper closed locus of codimension $1$ (see Figure \ref{78pic}).
\item[($Z$)] $E = (1, 5, 6), F = (2, 10)$: codimension $2$, the hyperelliptic curves (see \eqref{hypeq}).
\end{enumerate}

In genus $9$, it is less clear which strata lie in the closure of others. However, for our purposes,
all we need is our $\leq$ order, pictured in Figure \ref{9order} below. 
\begin{figure}[h!]
    \centering
\begin{tikzcd}
&&\Psi_0 \\
&\Psi_2 \arrow[no head]{ur} & & \Psi_1 \arrow[no head]{ul} \\
\Psi_5 \arrow[no head]{ur} & &\Psi_3 \arrow[no head]{ul} \arrow[no head]{ur} & &\Psi_4 \arrow[no head]{ul} \\
&&& \Sigma_6 \arrow[no head]{lu} \arrow[no head]{ur} \\
&&& \Sigma_7 \arrow[no head]{u} \\
&&\Sigma_8 \arrow[no head]{ur} \arrow[no head]{uuull} \\
&& Z \arrow[no head]{u}
\end{tikzcd}
    \caption{Our $\leq$ order in genus $9$}
    \label{9order}
\end{figure}
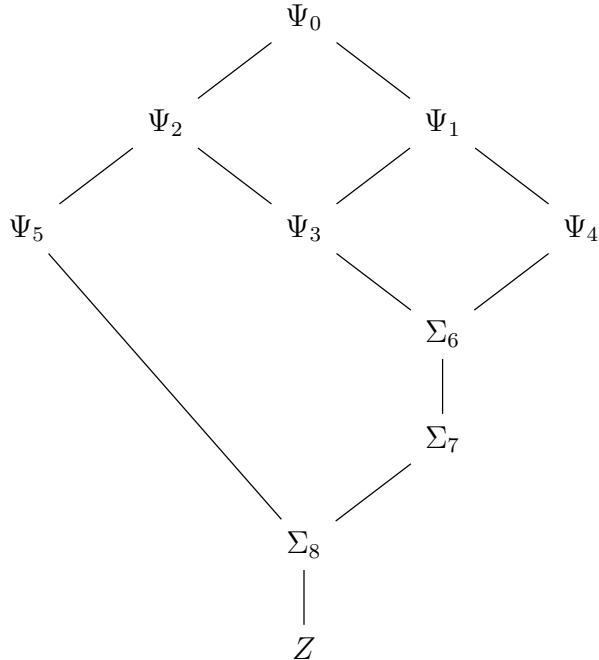

Note that $\Sigma_7$ and $\Sigma_8$ have the same dimension, so $\Sigma_8$ is not contained in the closure of $\Sigma_7$ (see Example \ref{wex} for a baby case of this phenomenon).
\begin{rem}
The two ``problem strata" mentioned in the introduction are $\Sigma_8$ and $\Sigma_6$. Note that these are more ``interesting" nodes in the diagram above: they live directly below two different strata (i.e. there are two lines coming out the tops of these nodes).
\end{rem}

The key to our argument is a good geometric understanding of $\Sigma_8, \Sigma_7$ and $\Sigma_6$. (We already know that $Z$ consists of hyperelliptic curves, so it is not a concern.)

We first describe the bielliptic locus, making use of the explicit description due to Casnati--Del Centina \cite{CDC}  for curves of genus $6 \leq g \leq 9$. 
Let $\B_9 \subset \M_9$ denote the locus of curves $C$ which are double covers of an elliptic curve $D$. By the Castelnuovo--Severi inequality,
$C \to D$ is the unique degree $2$ map of $C$ to an elliptic curve. By Riemann--Hurwitz, $\B_9$ is irreducible of dimension $16$.
Also by the Castelnuovo--Severi inequality, every degree $4$ map $C \to \pp^1$ from a bielliptic curve $C$
factors through the map $C \to D$. That is, $\beta^{-1}(\B_9) \subset \H_{4,9}$ consists of maps of the form $C \to D \to \pp^1$. Hence, $\beta^{-1}(\B_9) \subset \H_{4,9}$ is irreducible of dimension $17$ and $\beta^{-1}(\B_9) \to \B_9$ has $1$-dimensional fibers.
Recall that by Riemann-Hurwitz,
$\dim \H_{4,9} = 21$.

\begin{lem} \label{bi}
Every curve in stratum $\Sigma_7$ is bielliptic. Moreover, every degree $4$ cover that factors through an elliptic curve lives in the closure of $\Sigma_7$, i.e. $\overline{\Sigma}_7 = \beta^{-1}(\B_9)$.
\end{lem}
\begin{proof}
On $\Sigma_7$, we have $E = (2,5, 5)$ and $F = (4, 8)$. Thus, for degree reasons, we have $q_{1,2} =q_{1,3} =0$, so the conditions \cite[Theorem 2.3 (general case)]{CDC} are automatically satisfied. This says that $\Sigma_7 \subset \beta^{-1}(\B_9)$. Meanwhile, $\Sigma_7$ is irreducible of codimension $4$, hence dimension $17$. Since $\beta^{-1}(\B_9)$ is closed and irreducible of dimension $17$, we must have $\overline{\Sigma}_7 = \beta^{-1}(\B_9)$.
\end{proof}

Theorem 2.3 of \cite{CDC} shows that $\beta^{-1}(\B_9)$ meets precisely one other stratum, $\Sigma_8$, in codimension $1$ inside $\Sigma_8$. Casnati-Del Centina call the intersection $\overline{\Sigma}_7 \cap \Sigma_8$ the \emph{special bielliptics} (pictured in purple in Figure \ref{78pic}). The special bielliptics $C \xrightarrow{\varphi} D \xrightarrow{\sigma} \pp^1$ are characterized by the property that the branch locus of $\varphi$ is linearly equivalent to $\sigma^*\O_{\pp^1}(8)$ on the elliptic curve $D$.  Given a bielliptic curve $C \to D$, one may always choose a map $D \to \pp^1$ so that the composition $C \to D \to \pp^1$ is special. Hence, $\beta(\overline{\Sigma}_7) \subset \beta(\Sigma_8)$.
\begin{figure}[h!]
    \centering
    \includegraphics[width = 3.5in]{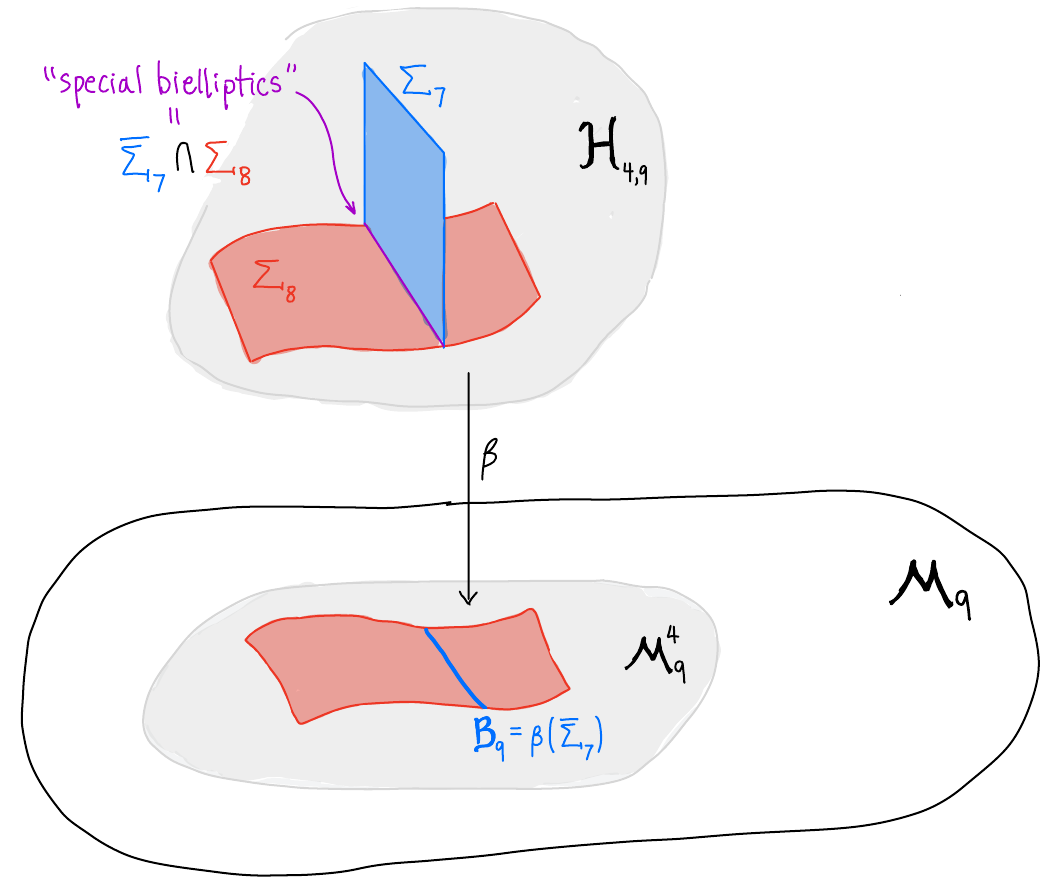}
    \caption{The map $\beta$ contracts $\Sigma_7$ and $\beta(\overline{\Sigma}_7) \subset \beta(\Sigma_8)$}
    \label{78pic}
\end{figure}

Recall that (see Figure \ref{9order}), the stratum $\Sigma_8$ is closed in $\H_{4,9} \smallsetminus Z = \H_{4,9} \smallsetminus \beta^{-1}(\M_9^3)$.
\begin{lem} \label{7and8} 
The push forward $\beta_*(\Sigma_8)$ is tautological on $\M_9 \smallsetminus \M_9^3$. Furthermore, the push forward of any class supported on $\Sigma_8 \cup \Sigma_7$ is tautological on $\M_9 \smallsetminus \M_9^3$.
\end{lem}
\begin{proof}
We first observe that $\Sigma_{(2, *, *)}(\E) = \Sigma_7 \cup \Sigma_8$ and has pure codimension $4$, which is the expected codimension. By Lemma \ref{2star} (see also Example \ref{246ex}), the fundamental class of $\Sigma_{(2, *, *)}(\E)$ is tautological in $\H_{4,9} \smallsetminus Z = \H_{4, 9} \smallsetminus \beta^{-1}(\M_9^3)$. Hence, by Theorem \ref{pushforward}, we have
\[\beta_*[\Sigma_{(2, *, *)}(\E)] = \beta_*[\Sigma_8] + \beta_*[\overline{\Sigma}_7]\]
is tautological on $\M_9 \smallsetminus \M_9^3$. But, by Lemma \ref{bi}, $\overline{\Sigma}_7 = \beta^{-1}(\B_9)$ maps to $\M_9$ with $1$-dimensional fibers (pictured in blue in Figure \ref{78pic}). Hence, $\beta_*[\overline{\Sigma}_7] = 0$, so $\beta_*[\Sigma_8]$ is tautological on $\M_9 \smallsetminus \M_9^3$. 

By Lemma \ref{distinctparts} \eqref{for8}, we know that $A^*(\Sigma_8)$ is generated by the pullbacks of $\kappa_1$ and $\kappa_2$. Hence, using the push-pull formula, the push forward of every class supported on $\Sigma_8$ is tautological. Since we are working with rational coefficients,
the pushforward map from $A^*(\Sigma_8)$ to $A^*(\beta(\Sigma_8))$ is surjective. Hence,
every class supported on $\beta(\Sigma_8)$ is tautological on $\M_9 \smallsetminus \M_9^3$. The last sentence now follows because $\beta(\Sigma_8 \cup \Sigma_7) = \beta(\Sigma_8)$.
\end{proof}

\begin{example}[Regarding the class of $\B_9$]
To further explicate the second paragraph of the above proof, we explain why the fundamental class of $\B_9 \subset \M_9 \smallsetminus \M_9^3$ is tautological. (Once known to be tautological, this class actually must be $0$ by a result of Looijenga \cite{Looijenga}.)
Let $S = \overline{\Sigma}_7 \cap \Sigma_8 \subset \Sigma_8$ be the locus of special bielliptics (pictured in purple in Figure \ref{78pic}). We know that $S$ maps finitely and surjectively onto $\B_9$. Thus, the fundamental class of $\B_9 \subset \M_9 \smallsetminus \M_9^3$ is a multiple of $\beta_*[S]$. By Lemma \ref{distinctparts} \eqref{for8}, the class of $S$ inside $A^*(\Sigma_8)$ is a multiple of $\beta^*\kappa_1$, so by the push-pull formula, $\beta_*[S]$ is a multiple of $\kappa_1 \cdot \beta_*[\Sigma_8]$.
\end{example}


Continuing up the partial order, we turn next to $\Sigma_6$.

\begin{lem} \label{hasg26}
Every curve in $\Sigma_6$ possesses a $g^2_6$ which is birational onto its image.
\end{lem}
\begin{proof}
On $\Sigma_6$, we have $2e_1 - f_2 < 0$ and $2e_1 - f_2 = 2$. Therefore, $C = V(p, q)$ meets the line $V(Y, Z) \subset \pp E^\vee$ in $2$ points (counted with multiplicity), say $p + q$. The line $V(Y, Z) \subset \pp E^\vee$ is dual to the canonical quotient $E \to \O(3)$. This line is sent to a line with degree $1$ under the map $\pp E^\vee \to \pp^8$ that factors the canonical embedding (which is given by $\O_{\pp E^\vee}(1) \otimes \gamma^*\omega_{\pp^1}$).

\begin{figure}[h!]
    \centering
    \includegraphics[width=6in]{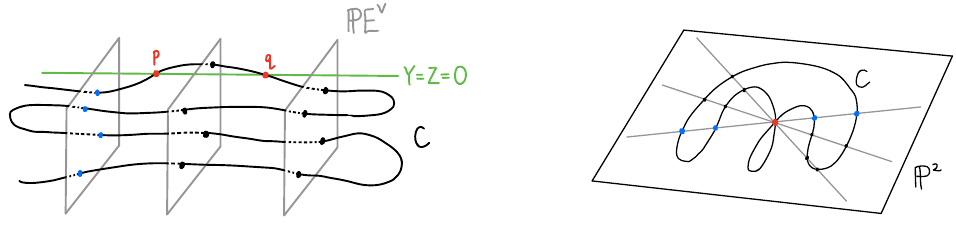}
    \caption{Curves in $\Sigma_6$ possess a $g^2_6$}
    \label{g26}
\end{figure}

As pictured on the left of Figure \ref{g26}, the line spanned by $p, q$ meets each plane spanned by the fibers of the $g^1_4$. Taking a fiber of the $g^1_4$ plus $p$ and $q$, we obtain $6$ points whose span under the canonical is $3$-dimensional. By Geometric Riemann--Roch, these six points constitute a $g^2_6$ (pictured on the right of Figure \ref{g26}). 

A $g^2_6$ is either (1) birational onto its image (2) a double cover of a degree $3$ plane curve or (3) a degree $3$ cover of a conic. A genus $9$ curve cannot have maps to $\pp^1$ of degrees $3$ and $4$ (if so it would map birationally to a curve of bidegree $(3, 4)$ on $\pp^1 \times \pp^1$, which has genus $6$), so we are not in case (3). Meanwhile, we have already established that everything in case (2) is contained in $\Sigma_8 \cup \Sigma_7$, which is disjoint from $\Sigma_6$. Thus, we must be in case (1).
\end{proof}

Let $\mathrm{PS} \subset \M_9$ denote the locus of plane sextics, i.e. curves of genus $9$ with a $g^2_6$ that is birational onto its image.
Let $\Delta^\circ \subset H^0(\pp^2, \O_{\pp^2}(6))$ be the locally closed subvariety of degree $6$ equations on $\pp^2$ whose vanishing locus has exactly one double point. Such curves have geometric genus $9$ and $\Delta^\circ/\GL_3$ maps surjectively onto $\mathrm{PS} \subset \M_9$. In particular, $\mathrm{PS}$ is irreducible and
\begin{equation} \label{dimps}
\dim \mathrm{PS} \leq h^0(\pp^2, \O_{\pp^2}(6)) - 1 -\dim \GL_3 = 28 - 1 - 9 = 18.
\end{equation}

\begin{lem} \label{678}
The push forward $\beta_*[\overline{\Sigma}_6]$ is tautological on $\M_9 \smallsetminus \M_9^3$. Hence, the push forward of any class supported on $\Sigma_8 \cup \Sigma_7 \cup \Sigma_6$ is tautological on $\M_9 \smallsetminus \M_9^3$.
\end{lem}

\begin{proof}
If $\Sigma_6 \to \M_9$ has positive dimensional fibers, then $\beta_*[\overline{\Sigma}_6] = 0$, which is tautological. So let us assume $\Sigma_6 \to \M_9$ is generically finite onto its image, in which case $\beta_*[\overline{\Sigma}_6]$ is a multiple of $[\beta(\overline{\Sigma}_6)]$.
By \ref{hasg26}, we have $\beta(\overline{\Sigma}_6) \subset \overline{\mathrm{PS}}$ (where the closure of $\mathrm{PS}$ is taken in $\M_9 \smallsetminus \M_9^3$).
It follows that $\dim \mathrm{PS} \geq \dim \Sigma_6 = 18$. By \eqref{dimps}, we conclude that $\dim \mathrm{PS} = 18$ and such curves possess finitely many $g^2_6$'s.
Now, both $\beta(\overline{\Sigma}_6)$ and $\overline{\mathrm{PS}}$ are irreducible of dimension $18$, so they must be equal. Therefore, we wish to show that $[\overline{\mathrm{PS}}]$ is tautological on $\M_9 \smallsetminus \M_9^3$.
We know by Lemma \ref{7and8} that all classes supported on $\B_9$ are tautological, so it suffices to work on the further open $\M_9 \smallsetminus (\M_9^3 \cup \B_9)$.

Let $\rho(g, r, d) := g - (r+1)(g - d + r)$ be the Brill--Noether number. In particular, $\rho(9, 2, 6) = -6$, so the plane sextics are ``expected" to occur in codimension $6$ on $\M_9$.
On the open $\M_9 \smallsetminus (\M_9^3 \cup \B_9)$, the locus of curves that possess a $g^2_6$ is $\mathrm{PS}$, which has dimension $18 = \dim \M_9 + \rho(9, 2, 6)$, so it is a Brill--Noether locus of the expected dimension. (Notice we need to work on the complement of $\M_9^3 \cup \B_9$, as curves of gonality less than or equal to $3$ possess a $g^2_6$, but $\dim \M_9^3 = 19$ is too large. In the end, this is not an issue because we already know all classes supported on $\M_9^3$ are tautological.)
Moreover, each curve in $\mathrm{PS}$ possesses finitely many $g^2_6$'s (this is where it is important we also work on the complement of $\B_9$).
We can therefore apply Faber's argument as in \cite[p. 15-16]{F} on the open $\M_9 \smallsetminus (\M_9^3 \cup \B_9)$ to see that $[\mathrm{PS}]$ is tautological in $A^*(\M_9 \smallsetminus (\M_9^3 \cup \B_9))$.
Since all classes supported on $\B_9$ have already been shown to be tautological, $[\overline{\mathrm{PS}}]$ is tautological in $A^*(\M_9 \smallsetminus \M_9^3)$ too.

By Lemma \ref{distinctparts} \eqref{for6}(ii), we know that $A^*(\Sigma_6)$ is generated by the pullback of $\kappa_1$ and $\kappa_2$. The second claim now follows by the push-pull formula and Lemma \ref{7and8}.
\end{proof}

We now complete the goal of this subsection.

\begin{lem}
All classes supported on $\M_9^4$ are tautological on $\M_9$.
\end{lem}
\begin{proof}
By Proposition \ref{Hprime}, the Chow ring of $\Psi = \Psi_0 \cup \cdots \cup \Psi_5$ is generated by CE classes.
Combining this with Theorem \ref{pushforward} and Lemma \ref{678}, we see that every class in $\H_{4,9} \smallsetminus \beta^{-1}(\M_9^3)$ pushes forward to a tautological class on $\M_9 \smallsetminus \M_9^3$. Such push forwards span all classes supported on $\M_9^4 \smallsetminus \M_9^3 \subset \M_9 \smallsetminus \M_9^3$. Finally, all classes supported on $\M_9^3$ are known to be tautological (see \eqref{Mg3}).
\end{proof}

\subsection{The end of our luck: bielliptics in genus $10$ and beyond} \label{noluck} 
We point out here one last coincidence in genus $10$, which allows us to see that the bielliptic locus on $\M_{10}$ is tautological. We then explain why these coincidences that drive our technique do not continue into higher genus.

For $g \geq 10$, the bielliptics completely fill the strata they occupy. Let $h = \lfloor \frac{g}{2} \rfloor$.
By \cite[Proposition 2.1]{CDC} and the sentence following it, for $g \geq 10$, we have
\begin{equation} \label{beq}
\beta^{-1}(\B_g) = \begin{cases} \Sigma_{(2, h, h+1)}(\E) \cap \Sigma_{(4, g - 1)}(\F) & \text{if $g$ even}\\
[\Sigma_{(2, h+1, h+1)}(\E) \cap \Sigma_{(4, g - 1)}(\F)] \cup  [\Sigma_{(2, h, h+2)}(\E) \cap \Sigma_{(4, g - 1)}(\F)]
&\text{if $g$ odd.}
\end{cases}
\end{equation}
Using similar techniques as in genus $7, 8,9$ we establish that the fundamental class of the bielliptic locus $\mathcal{B}_{10} \subset \M_{10}$ is tautological.
\begin{proof}[Proof of Theorem \ref{bielliptic10}]
Using \eqref{totaldeg} -- \eqref{conditional} in genus 10, one sees that $F =(4, 9)$ occurs only with $E = (2, 5, 6)$. Let $\Sigma$ be this splitting locus. 
By \eqref{beq}, we have $\Sigma = \beta^{-1}(\B_{10})$.
Meanwhile, $\codim \Sigma = 4$ is the expected codimension for $\Sigma_{(4, 9)}(\F)$, so by Lemma \ref{sigmaCE}, we see that $[\Sigma] = [\Sigma_{(4, 9)}(\F)]$ is tautological (modulo classes supported on $\beta^{-1}(\M_{10}^3)$).
By Lemma \ref{distinctparts} \eqref{dp1}, we know $A^*(\Sigma)$ is generated by restrictions of CE classes. By Theorem \ref{pushforward}, the push forward of every class supported on $\Sigma$ is tautological on $\M_{10} \smallsetminus \M_{10}^3$.
Since we are working with rational coefficients, the push forward map on Chow groups from $\Sigma$ to $\beta(\Sigma)$ is surjective. In particular, $[\beta(\Sigma)] = [\B_{10}]$ is tautological. The vanishing of $[\B_{10}]$ then follows from a theorem of Looijenga \cite{Looijenga}, which says that the tautological ring vanishes in codimension $d>g-2$.
\end{proof}

The codimension of $\beta^{-1}(\B_g) \subset \H_{4,g}$ is always $4$.
However, for $g \geq 11$, neither $\Sigma_{(4, g - 1)}(\F)$ nor $\Sigma_{(2,*,*)}(\E)$ has expected codimension $4$.
Thus, there is no way to realize the strata in \eqref{beq} as splitting loci of the expected dimension. 
As an example, in genus $12$, the bielliptics have $E = (2, 6, 7)$ and $F = (4, 11)$. In this case,  
\[h^1(\pp^1, \E nd(\O(2,6,7))) = 7 \qquad \text{and} \qquad h^1(\pp^1, \E nd(\O(4,11))) = 6, \]
so neither expected codimension is $4$.
In fact, in genus $12$, we claim van Zelm's result  \cite{VZ} that $[\B_{12}]$ is non-tautological on $\M_{12}$ implies $[\beta^{-1}(\B_{12})]$ is non-tautological on $\H_{4,12}$.
By Lemma \ref{distinctparts} \eqref{dp1}, we know that $\beta^{-1}(\B_{12})$ is generated by the restrictions of CE classes. Therefore, if $[\beta^{-1}(\B_{12})]$ were tautological, using the push-pull formula and  Theorem \ref{pushforward}, 
we would see that all classes supported on $\B_{12}$ were tautological, which is a contradiction.


\section{The Pentagonal Locus}\label{pent}
In this section, we show that $A^*(\H_{5,g}\smallsetminus \beta^{-1}(\M^4_g))$ is generated by tautological classes for $g = 7, 8, 9$. Given a degree $5$, genus $g$ cover $\alpha:C\rightarrow \p^1$, let $E = E_\alpha$ and $F = F_\alpha$ be the associated vector bundles on $\pp^1$ as in Section \ref{CEsec}. 
Let $\gamma: \pp E^\vee \to \pp^1$.
The line bundle $\O_{\pp E^\vee}(1) \otimes \gamma^* \omega_{\pp^1}$ defines a map $\pp E^\vee \to \pp^{g-1}$ such that the composition $C \subset \pp E^\vee \to \pp^{g-1}$ is the canonical embedding.
The bundles $E$ and $F$ split
\begin{align*}
E &= \O(e_1) \oplus \O(e_2) \oplus \O(e_3)\oplus \O(e_4) &\qquad &e_1 \leq e_2 \leq e_3\leq e_4
\intertext{and}
F &= \O(f_1) \oplus \O(f_2)\oplus \O(f_3)\oplus \O(f_4)\oplus \O(f_5) &\qquad &f_1 \leq f_2\leq f_3\leq f_4\leq f_5.
\end{align*}
As in the degree $4$ case, the splitting types of $E$ and $F$ give a stratification of $\H_{5,g}$.
This stratification was studied by Schreyer \cite{S} when $g = 7, 8$, and Sagraloff \cite{Sag} when $g=9$. (The translation between our notation and Schreyer's is that $a_i = f_i - 4$; the splitting type of $E$ determines the type of Schreyer's determinantal surface $Y$.)

The condition to be inside $\Psi = \H_{5,g} \smallsetminus \Supp(R^1\pi_* \E \otimes \det \E^\vee \otimes \wedge^2 \F)$ is that 
\begin{equation} \label{good}e_1 + f_1 + f_2 - (g+4) \geq -1 \qquad \Longleftrightarrow \qquad e_1 + f_1 + f_2 \geq \begin{cases} 10 & \text{if $g = 7$} \\ 11 & \text{if $g = 8$.} \\ 12 & \text{if $g=9$.} \end{cases}
\end{equation}
Just as in the degree $4$ case, there are several constraints on the splitting types. We collect some of these constraints below. Using these constraints, we recover the stratifications found by Schreyer \cite{S} in genus $7$ and $8$ and Sagraloff \cite{Sag} in genus $9$.

To begin,  we know that $\deg(E)=g+4$ and $\det E^{\otimes 2}\cong \det F $, so we have
\begin{gather}
e_1 + e_2 + e_3 + e_4 = g +4, \label{5eq1} \\
f_1 + f_2 + f_3 + f_4 + f_5 = 2g + 8.
\end{gather}
By \cite[Proposition 2.6]{DP}, we have
\begin{equation}\frac{g+4}{10} \leq e_1 \leq \frac{g+4}{4} \qquad \text{and} \qquad e_4 \leq \frac{2g + 8}{5}.
\end{equation}
Because $F$ is a subbundle of $\Sym^2 E$,
\begin{equation} \label{topf} f_5 \leq 2e_4. 
\end{equation}
Note that equations \eqref{5eq1}--\eqref{topf} always reduce us to a finite list of allowed splitting types.

Next, we introduce some notation. Every section in $H := H^0(E \otimes \det E^\vee \otimes \wedge^2 F)$ can be represented by a skew symmetric matrix
\begin{equation} \label{mat} M = \left(\begin{matrix} 0 & L_{1,2} & L_{1,3} & L_{1,4} & L_{1,5} \\ -L_{1,2} & 0 & L_{2,3} & L_{2,4} & L_{2,5} \\ -L_{1,3} & -L_{2,3} & 0 & L_{3,4} & L_{3,5} \\ -L_{1,4} & -L_{2,4} & -L_{3,4} & 0 & L_{4,5} \\ -L_{1,5} & -L_{2,5} & -L_{3,5} & -L_{4,5} & 0 \end{matrix} \right), \end{equation}
where $L_{i,j} \in H^0(\O(f_i + f_j) \otimes \O(\vec{e}) \otimes \O(-g-4))$.
The equations defining $D(\Phi(\eta)) \subset \pp E^\vee$ are the $5$ Pfaffian quadrics listed below:
\begin{align*}
Q_1 &= L_{2,5}L_{3,4} - L_{2,4}L_{3,5} + L_{2,3}L_{4,5} \\
Q_2 &= L_{1,5}L_{3,4} - L_{1,4}L_{3,5} + L_{1,3}L_{4,5} \\
Q_3 &= L_{1,5}L_{2,4} - L_{1,4}L_{2,5} + L_{1,2}L_{4,5} \\
Q_4 &= L_{1,5}L_{2,3} - L_{1,3}L_{2,5} + L_{1,2}L_{3,5} \\
Q_5 &= L_{1,4}L_{2,3} - L_{1,3}L_{2,4} + L_{1,2}L_{3,4}.
\end{align*}
Corresponding to the splitting of $E = \O(\vec{e})$, we can take coordinates $X_1,\dots,X_4$ on $\pp E^\vee$. The $L_{i,j}$ are linear homogeneous polynomials in the $X_k$ whose coefficients are elements of $H^0(\p^1,\O(f_i+f_j+e_k-g-4))$. We write these as:
\begin{align*}
    L_{i,j}&=a_{i,j}X_1+b_{i,j}X_2+c_{i,j}X_3+d_{i,j}X_4.
\end{align*}

If $L_{1,2}$ and $L_{1,3}$ were identically zero, then $Q_5$ would be reducible, which is impossible because $C$ is irreducible. Therefore, we must have
\begin{equation}
    f_1 + f_3 + e_4 \geq g+4
\end{equation}
If $X_4$ divides $L_{1,2}, L_{1,3},$ and $L_{1,4}$, then $Q_5$ is reducible. This will occur if $a_{1,4},b_{1,4}, c_{1,4}$ all identically vanish. In order for $a_{1,4}, b_{1,4},$ and $c_{1,4}$ to not all identically vanish, we must have
\begin{equation}
f_1+f_4+e_3 \geq g+4.
\end{equation}
Similarly, if $X_4$ divides $L_{1,2}$, $L_{1,3}$ and $L_{2,3}$, then $X_4$ divides $Q_5$ and $Q_5$ is reducible. To prevent this, we must have
\begin{equation}
    f_2 + f_3 + e_3  \geq g+4.
\end{equation}
Note that the curve $C$ cannot contain the section defined by $X_2=X_3=X_4=0$. Otherwise, it would be reducible. Therefore, at least one of the $Q_i$ must have a nonzero coefficient of $X_1^2$. If the coefficient of $X_1$ is zero  in $L_{1,2}, \ldots, L_{1,5}$ and $L_{2,3}, L_{2,4}$, and  $L_{2,5}$, then the coefficient of $X_1^2$ vanishes for all $Q_i$. Therefore, we must have that
\begin{equation}
f_2+f_5+e_1 \geq g+4.
\end{equation}
Similarly, we note that we must have
\begin{equation}
f_3+f_4+e_1 \geq g+4.
\end{equation}
Indeed, if not, then the coefficient of $X_1$ vanishes for all $L_{i,j}$ except possibly $j=5$. It follows that none of the quadrics have an $X_1^2$ term in them, and thus they contain the section $X_2=X_3=X_4=0$, so $C$ would be reducible.

If all $a_{1,j} = b_{1,j} = 0$, then the quadrics $Q_2, \ldots, Q_5$ all vanish on $V(X_3, X_4)$. The remaining equation 
 $Q_1$ then cuts out a divisor on the surface $V(X_3,X_4)$, so either $C$ is reducible or is entirely contained in $V(X_3,X_4)$. But this is impossible because then in the canonical embedding would send five points on $C$ to a common line, which means $C$ has a $g^3_5$. Projection from a point gives a $g^2_5$, so $C$ would have genus at most $6$.
To prevent this, 
\begin{equation}
f_1 + f_5 + e_2 \geq g+4.
\end{equation}
Another bad thing is if $L_{1,2}, L_{1,3}, L_{2,3}, L_{1,4}, L_{2,4}$ are all zero on $V(X_3, X_4)$. In this case,  the restriction of the quadrics to $V(X_3, X_4)$ is $Q_1 = L_{2,5}L_{3,4}, Q_2 = L_{1,5}L_{3,4}, Q_3 = Q_4 = Q_5 = 0$, so $V(Q_1)$ and $V(Q_2)$ share $1$-dimensional component inside $V(X_3, X_4)$. To prevent this, we need
\begin{equation} \label{beforeg26}
f_2 + f_4 + e_2 \geq g + 4.\end{equation}

Next, we note some conditions that imply $C$ has a special linear series.
If $a_{1,5} = 0$ and $a_{2,3} = 0$, then the curve meets the line $V(X_2, X_3, X_4)$ along $V(a_{2,5} a_{3,4} - a_{3,5}a_{2,4})$. 
The degree of $V(a_{2,5} a_{3,4} - a_{3,5}a_{2,4})$ is $\deg(a_{2,5}) + \deg(a_{3,4}) = 2e_1 - f_1$.
If $e_1 = 2$, the line $V(X_2, X_3, X_4)$ is contracted in the map $\pp E^\vee \to \pp^{g-1}$. If $2e_1 - f_1 = 1$, then $C$ meets $V(X_2, X_3, X_4)$ in a point $p$. Under the canonical, this point $p$ lies in the span of the five points in a fiber so $p$ plus the $g^1_5$ is a
 $g^2_6$ on $C$.
This yields the condition
\begin{equation} \label{ourg26}
\text{if $e_1 = 2$ and $e_1 + f_1 + f_5, e_1 + f_2 + f_3 < g+4$ and $2e_1 - f_1 = 1$, then $C$ has a $g^2_6$.}
\end{equation}

As another source of special linear series,
Schreyer shows \cite[p. 136]{S} that if $L_{1,2} =0$, then $C$ lies on a certain determinantal surface, which is birational to a Hirzebruch surface $\mathbb{F}_k := \pp(\O_{\pp^1} \oplus \O_{\pp^1}(k))$ for $k = f_2 - f_1$. Schreyer determines determines the class of the image of $C$ on this Hirzebruch surface in \cite[Theorem 5.7]{S}. In the case $k = 0$, we have $\mathbb{F}_0 \cong \pp^1 \times \pp^1$ and projection onto the other factor determines another pencil on $C$. Similarly, if $k = 1$, then $\mathbb{F}_1$ admits a map to $\pp^2$ and we obtain a $g^2_d$. The degree $d$ of these special linear series is given by intersecting Schreyer's class $C'$ with the $\O_{\mathbb{F}_k}(1)$ (which Schreyer calls $A$). This calculation is summarized nicely by Sagraloff in \cite[p. 65]{Sag} (to translate our splitting types, $f_i = a_{5-i} + 4$):
\begin{equation} \label{extrapencil}
\text{if $f_2 - f_1 = k$ and $L_{1,2} = 0$, then $C$ possesses a $g^{1+k}_{f_1}$.} \end{equation}

The final condition we note concerns the situation when
$e_1 + f_2 + f_4 - (g+4) < 0$. In this case, $a_{1,2} = a_{1,3} = a_{2,3} = a_{1, 4} = a_{2,4} = 0$.
Restricting the five quadrics $Q_1, \ldots, Q_5$ to the line $Z = V(X_2, X_3, X_4)$ we obtain
\begin{align*}
Q_1|_Z = a_{2,5} a_{3,4} X_1^2, \qquad \qquad Q_2|_Z = a_{1,5} a_{3,4} X_1^2, \qquad \qquad
Q_3|_Z = Q_4|_Z = Q_5|_Z = 0.
\end{align*}
In particular, if $a_{3,4} = 0$, then $C = D(\Phi(\eta))$ would contain the line $Z$, but such a curve would be reducible. Therefore,
\begin{equation} \label{a34nonzero}
\text{if $e_1 + f_2 + f_4 < g+4$, then $a_{3,4} \neq 0$.}
\end{equation}


\subsection{Strategy}
The strategy for the pentagonal locus is the same as, or even simpler than, the strategy for the tetragonal locus. 
\begin{enumerate}
    \item Use conditions \eqref{5eq1}--\eqref{beforeg26} to determine the allowed pairs of splitting types $\vec{e},\vec{f}$. The partial order on splitting types of Section \ref{slsec} induces a partial order on pairs of splitting types by $(\ee,\vec{f}')\leq (\vec{e},\vec{f})$ if $\ee\leq \vec{e}$ and $\vec{f}'\leq \vec{f}$.
    \item Starting with strata at the bottom of the partial ordering and working upwards, show that for each stratum outside of $\Psi$ at least one of the following is satisfied:
    \begin{enumerate}
        \item the stratum is contained in $\beta^{-1}(\M^4_g)$.
        \item its fundamental class in $\H_{5,g}\smallsetminus \beta^{-1}(\M^4_g)$ is tautological (modulo classes supported on strata below it in the partial order) \textit{and} the Chow ring of the locally closed stratum $\Sigma' := \Sigma \smallsetminus \beta^{-1}(\M_g^4)$ is generated by the restrictions of CE classes.
    \end{enumerate}
\end{enumerate}
This will establish that $A^*(\H_{5,g}\smallsetminus\beta^{-1}(\M^4_g))$ is generated by CE classes when $g=7,8,9$. In Section \ref{stratasec}, we show that the Chow rings of certain $\Sigma'$ are generated by restrictions of CE classes. Then, in Sections \ref{sec57}, \ref{sec58}, \ref{sec59}, we treat the cases $g = 7, 8, 9$ respectively.
\subsection{Chow rings of very unbalanced splitting strata} \label{stratasec}
In Lemma \ref{squo5}, we described each pair splitting locus as a quotient $\Sigma = [(U \times \gg_m)/G]$ where $G = \SL_2 \ltimes (\Aut(\vec{e}) \times \Aut(\vec{f}))$, and $U \subset H := H^0(E \otimes \det E^\vee \otimes \wedge^2 F)$ was the open subvariety of sections $\eta$ so that the Pfaffian locus $D(\Phi(\eta)) \subseteq \pp E^\vee$ is a smooth, irreducible curve. Let $U' \subset U$ be the further open where $C = D(\Phi(\eta))$ does not possess a $g^1_d$ for $d < 5$, so $\Sigma' := \Sigma \smallsetminus \beta^{-1}(\M_g^4) = [(U' \times \gg_m)/G]$.
We have a series of surjections
\begin{equation} \label{surj2}
A^*(BG) \twoheadrightarrow A^*([U'/G]) \twoheadrightarrow A^*(\Sigma'),
\end{equation}
The first map is induced by pullback from the structure map $v:[U'/G]\rightarrow BG$.
It will suffice to show that the images of generators on the left are expressible in terms of CE classes. To see this, we will need to know about the relations that come from the complement of $U' \subset H$.

We consider several different ``shapes" of splitting types that occur for pentagonal strata.

\begin{lem}\label{shape1}
Let $\vec{e} = (e_1, e_2, e_3, e_4)$ and $\vec{f} = (f_1, f_2, f_3, f_4, f_5)$ satisfy the following conditions:
\begin{enumerate}
    \item $e_1 < e_2 = e_3 < e_4$,
    \item $f_1 = f_2 < f_3 = f_4 < f_5$,
    \item $e_4+f_1+f_2=g+4$,
    \item $e_1+f_3+f_4=g+4$,
    \item \label{detpart} $-e_1f_1+e_2f_1-e_2f_3+e_4f_3+e_1f_5-e_4f_5\neq 0$.
\end{enumerate}
Let $\Sigma$ denote the corresponding stratum with splitting types $\vec{e}$ and $\vec{f}$. Then $A^*(\Sigma')$ is generated by the restrictions of CE classes.
\end{lem}
\begin{proof}
Set $G=\SL_2 \ltimes (\Aut(\vec{e})\times \Aut(\vec{f}))$ and let $\pi: \P \rightarrow BG$ be the $\pp^1$ bundle pulled back from $\BSL_2$.
The first part of the HN filtration for $\V(\vec{e})$ is
\begin{equation} \label{ve1}
0\rightarrow \pi^*L(e_4)\rightarrow \V(\vec{e})\rightarrow Q_1\rightarrow 0,
\end{equation}
where $L$ is rank $1$.
The next part is
\begin{equation} \label{ve2}
0\rightarrow \pi^*R(e_2)\rightarrow Q_1\rightarrow \pi^*T(e_1)\rightarrow 0,
\end{equation}
where $R$ is rank $2$ and $T$ is rank $1$. Similarly, we have the HN filtration for $\V(\vec{f})$:
\begin{equation} \label{vf1}
0\rightarrow \pi^*S(f_5)\rightarrow \V(\vec{f})\rightarrow W_1\rightarrow 0,
\end{equation}
and then
\begin{equation} \label{vf2}
0\rightarrow \pi^*M(f_3)\rightarrow W_1\rightarrow \pi^*N(f_1)\rightarrow 0.
\end{equation}
The bundle $S$ is of rank $1$, and $M$ and $N$ are of rank $2$. 
We denote the Chern classes of an HN bundle by the corresponding lowercase letter, with subscripts $i = 1, 2$ when the HN bundle has rank $2$. In particular, the right hand column vector in \eqref{mat51} below consists of the first Chern classes of the HN bundles.
From \eqref{ve1}--\eqref{vf2}, an application of the splitting principle gives the following expressions:
\begin{equation}\label{mat51}
\left(\small\begin{matrix}
 a_1 \\
 a_2' \\
 b_1 \\
 b_2' \\
\end{matrix} \right)=
\left(\small\begin{matrix}
1 & 1 & 1  & 0 & 0 & 0 \\
e_1+2e_2 & e_1+e_2+e_4 & 2e_2+e_4 & 0 & 0 & 0 \\
0 & 0 & 0 & 1 & 1 & 1 \\
0 & 0 & 0 & 2f_3+2f_1 & f_5+f_3+2f_1 & f_5+2f_3+f_1
\end{matrix} \right)
\left(\small\begin{matrix} \ell \\ r_1 \\ t \\ s \\ m_1 \\ n_1 \end{matrix}\right).
\end{equation}
Next, we show that the geometry of the curves in this stratum imposes some relations among $\ell, r_1, t, s, m_1, n_1$. We will show that modulo these relations, the Chern classes of the HN bundles are all expressible in terms of CE classes, finishing the proof.
Note that $BG$ has six generators in codimension $1$ (namely $\ell, r_1, t, s, m_1, n_1)$ but there are only four CE classes in codimension $1$, so we will need to show $A^1(BG) \to A^1([U'/G])$ has a kernel.

The two sources of relations are from the conditions \eqref{extrapencil} and \eqref{a34nonzero}. From \eqref{extrapencil}, we see that if $L_{1,2}=0$, then $C$ possesses a $g^1_4$. For degree reasons, we see that in this stratum
\[
L_{1,2}=d_{1,2}X_4.
\]
Therefore, the vanishing of $L_{1,2}$ is equivalent to the vanishing of $d_{1,2}$. Note that $\deg(d_{1,2}) = f_1 + f_2 + e_4 - (g+4) = 0$, so the vanishing of $d_{1,2}$ is one condition. Hence, the locus where $d_{1,2}$ vanishes is codimension $1$ inside the total space of sections $H$. We shall now describe more precisely the map that picks out the $d_{1,2}$ coefficient of an element in $H$.

Corresponding to the filtration on $\V(\vec{f})$, there is a quotient map
\[
\wedge^2 \V(\vec{f})\rightarrow \det(\pi^*N(f_1)).
\]
Tensoring by $\V(\vec{e})\otimes \det \V(\vec{e})^{\vee}$, we obtain a surjection
\[
\wedge^2 \V(\vec{f})\otimes\V(\vec{e})\otimes \det \V(\vec{e})^{\vee}\rightarrow \det(\pi^*N(f_1))\otimes\V(\vec{e})\otimes \det \V(\vec{e})^{\vee}.
\]
From the isomorphism
\[
\det \V(\vec{e})^{\vee}\cong \pi^*(\det R^{\vee}\otimes T^{\vee}\otimes L^{\vee})(-g-4),
\]
we obtain a surjection 
\begin{equation} \label{sur1}
\wedge^2 \V(\vec{f})\otimes\V(\vec{e})\otimes \det \V(\vec{e})^{\vee}\rightarrow \V(\vec{e})\otimes\pi^*(\det N\otimes\det R^{\vee}\otimes \det T^{\vee}\otimes L^{\vee} )(2f_1-g-4).
\end{equation}
Because $2f_1+e_i-g-4<0$ for $1\leq i \leq 3$ and $2f_1+e_4-g-4=0$, cohomology and base change implies that the push forward of \eqref{sur1} is given by
\begin{align*}
\pi_*(\wedge^2 \V(\vec{f})\otimes\V(\vec{e})\otimes \det \V(\vec{e})^{\vee})\rightarrow \pi_*(\pi^*(\det N\otimes\det R^{\vee}\otimes \det T^{\vee})(2f_1+e_4-g-4))\\
\cong \det N\otimes\det R^{\vee}\otimes \det T^{\vee}.
\end{align*}
The left-hand side above is the total space of $[H/G]$.
The above map corresponds to projection of $H$ onto the coefficient $d_{1,2}$.
Let $v: [U'/G] \to BG$ be the structure map, so $v^*: A^*(BG) \to A^*([U'/G])$ is the first map in \eqref{surj2}.
Since $d_{1,2}$ is non-vanishing on $U'$,
the pullback $v^*(\det N \otimes \det R^\vee \otimes \det T^\vee)$ admits a non-vanishing section on $[U'/G]$.
In particular we obtain the relation
\begin{equation}\label{n1}
    v^*(n_1-r_1-t)=0.
\end{equation}

Next, we turn to the condition \eqref{a34nonzero}. We want to write down a similar map that picks out the coefficient $a_{3,4}$. From the filtration on $\V(\vec{e})$, we have a surjection
\[
\V(\vec{e})\rightarrow \pi^*T(e_1).
\]
By tensoring with $\det \V(\vec{e})^{\vee}$, we obtain a surjection
\[
\V(\vec{e})\otimes \det \V(\vec{e})^{\vee}\rightarrow \pi^*(\det R^{\vee}\otimes L^{\vee})(e_1-g-4).
\]
Next, we note that we have a surjection $\wedge^2 \V(\vec{f})\rightarrow \wedge^2 W_1$, and so we obtain a surjection
\begin{equation} \label{sur3}
\wedge^2 \V(\vec{f})\otimes \V(\vec{e})\otimes \det \V(\vec{e})^{\vee}\rightarrow \wedge^2 W_1\otimes \pi^*(\det R^{\vee}\otimes L^{\vee})(e_1-g-4).
\end{equation}
There is a filtration of $\wedge^2 W_1$ with subquotients $\wedge^2 (\pi^*N(f_1))$, $\pi^*N(f_1)\otimes \pi^*M(f_3)$ and $\wedge^2(\pi^*M(f_3))$. Because $2f_1+e_1-g-4<0$, $f_1+f_3+e_1-g-4<0$, and $2f_3+e_1-g-4=0$, the $\pi$ push forward of \eqref{sur3} is given by
\[
\pi_*(\wedge^2 \V(\vec{f})\otimes \V(\vec{e})\otimes \det \V(\vec{e})^{\vee})\rightarrow \det M\otimes \det R^{\vee}\otimes \det L^{\vee}.
\]
This map corresponds to projecting onto $a_{3,4}$. Since $a_{3,4} \neq 0$ on $U'$,
 we obtain the relation 
\begin{equation}\label{m1}
    v^*(m_1-r_1-\ell)=0.
\end{equation}
Augmenting matrix \eqref{mat51} by the relations \eqref{m1} and \eqref{n1}, we obtain the matrix
\begin{equation}
\left(\small\begin{matrix}
 0\\
 0\\
 a_1 \\
 a_2' \\
 b_1 \\
 b_2' \\
\end{matrix} \right)=
\left(\small\begin{matrix}
0 & -1 & -1 & 0 & 0 & 1\\
-1 & -1 & 0 & 0 & 1 & 0\\
1 & 1 & 1  & 0 & 0 & 0 \\
e_1+2e_2 & e_1+e_2+e_4 & 2e_2+e_4 & 0 & 0 & 0 \\
0 & 0 & 0 & 1 & 1 & 1 \\
0 & 0 & 0 & 2f_3+2f_1 & f_5+f_3+2f_1 & f_5+2f_3+f_1
\end{matrix} \right)
\left(\small\begin{matrix} \ell \\ r_1 \\ t \\ s \\ m_1 \\ n_1 \end{matrix}\right).
\end{equation}
The determinant of the above $6\times 6$ matrix is the quantity in part \eqref{detpart} of the statement of the lemma.
By assumption, this determinant does not vanish, so the classes $\ell, r_1, t, s,m_1,n_1$ are expressible in terms of CE classes. 

Besides products of codimension $1$ generators, $BG$ has four codimension $2$ generators: $c_2, m_2, n_2, r_2$. By definition, $c_2$ is a CE class, so we just need to show that $m_2, n_2,$ and $r_2$ are expressible in terms of CE classes. Using the splitting principle on \eqref{ve1} and \eqref{ve2}, we obtain the following expression for $a_2$:
\[
a_2=r_2+r_1(\ell+t)+\ell t-(2e_1e_2+e_2^2+e_1e_4
     +2e_2e_4)c_2.
\]
Therefore, $r_2$ is expressible in terms of CE classes. Similarly, from \eqref{vf1} and \eqref{vf2}, we have the following expressions for $b_2$ and $b_3'$:
\begin{equation*}
\begin{split}
    b_2&=s(m_1+n_1)+m_1n_1+m_2+n_2-(2f_5f_3+f_3
      ^2+2f_5f_1+4f_3f_1+f_1^2)c_2,\\
    b_3'&=(f_3+2f_1)sm_1+(2f_3+f_1)sn_1+(f_5+f_3+f_1)m_1n_1+(f_5+2f_1)m_2+(f_5+2f_3)n_2 \\
    &\qquad -(f_5f_3^2+4f_5f_3f_1+2f_3^2f_1+f_5f_1^2+2f_3f_1^2)c_2.
    \end{split}
\end{equation*}
Since $f_1\neq f_3$, we see that $m_2$ and $n_2$ are expressible in terms of CE classes.
\end{proof}
\begin{lem}\label{forsigma2}
Let $\vec{e} = (e_1, e_2, e_3, e_4)$ and $\vec{f}=(f_1, f_2, f_3, f_4, f_5)$ satisfy the following conditions:
\begin{enumerate}
    \item $e_1<e_2<e_3=e_4$,
    \item $f_1<f_2=f_3<f_4=f_5$,
    \item $e_1+f_2+f_5=g+4$,
    \item $e_3+f_1+f_2=g+4$,
    \item \label{det2} $2e_2f_1-2e_3f_1-e_1f_2-3e_2f_2+4e_3f_2+e_1f_4+e_2f_4-2e_3f_4\neq 0$.
\end{enumerate}
Let $\Sigma$ denote the corresponding stratum with splitting types $\vec{e}$ and $\vec{f}$. Then $A^*(\Sigma')$ is generated by the restrictions of CE classes.
\end{lem}
\begin{proof}
Set $G=\SL_2 \ltimes (\Aut(\vec{e})\times \Aut(\vec{f}))$ and let $\pi: \P \rightarrow BG$ be the $\pp^1$ bundle pulled back from $\BSL_2$. The first part of the HN filtration for $\V(\vec{e})$ is 
\begin{equation}\label{E1}
0\rightarrow \pi^*R(e_3)\rightarrow \V(\vec{e})\rightarrow \pi^* W_1\rightarrow 0
\end{equation}
and then
\begin{equation}\label{E2}
0\rightarrow \pi^*S(e_2)\rightarrow W_1\rightarrow \pi^*L(e_1)\rightarrow 0,
\end{equation}
where $R$ is of rank $2$ and $S$ and $L$ are of rank $1$.
Similarly, for $\V(\vec{f})$ we have
\begin{equation}\label{F1}
0\rightarrow \pi^*M(f_4)\rightarrow \V(\vec{f})\rightarrow W_2\rightarrow 0
\end{equation}
and
\begin{equation}\label{F2}
0\rightarrow \pi^*N(f_2)\rightarrow W_2\rightarrow T(f_1)\rightarrow 0,
\end{equation}
where $T$ is rank $1$, and $M$ and $N$ are rank $2$. As usual, we denote the Chern classes of an HN bundle by the corresponding lowercase letter, with subscripts $i = 1, 2$ when the HN bundle has rank $2$.
Using the splitting principle and the definitions of CE classes, we obtain the following expressions:
\begin{equation} \label{mat1}
\left(\small\begin{matrix}
 a_1 \\
 a_2' \\
 b_1 \\
 b_2' \\
\end{matrix} \right)=
\left(\small\begin{matrix}
1 & 1 & 1  & 0 & 0 & 0 \\
e_2+2e_3 & e_1+e_2+e_3 & e_1+2e_3 & 0 & 0 & 0 \\
0 & 0 & 0 & 1 & 1 & 1 \\
0 & 0 & 0 & 2f_4+2f_2 & f_4+2f_2+f_1 & 2f_4+f_2+f_1
\end{matrix} \right)
\left(\small\begin{matrix} \ell \\ r_1 \\ s \\ t \\ m_1 \\ n_1 \end{matrix}\right).
\end{equation}

There are $6$ generators for $A^1(BG)$, but only $4$ codimension $1$ CE classes. Therefore, we will need to study the kernel of $A^1(BG) \to A^1([U'/G])$. Let $Z=V(X_2,X_3,X_4)\subset \p E^\vee$. For degree reasons, when we restrict the quadrics to $Z$, we see that they all vanish except for possibly $Q_1|_Z$, which takes the form
\[
Q_1|_{Z}=(a_{2,5}a_{3,4}-a_{2,4}a_{3,5})X_1^2.
\]
The coefficient of $X_1^2$ is of degree $0$ on $\p^1$. Note that if it vanishes, the curve becomes reducible, so the vanishing of this coefficient should impose a codimension $1$ relation. To find this relation, we need to find a way of picking out the quadric $Q_1$. 

Recall that the quadrics $Q_i$ cutting out the curve are obtained from the Pfaffians of the skew-symmetric matrix \eqref{mat}. 
Using the canonical identification $\wedge^4 F \cong F^\vee \otimes \det F$, these $5$ Pfaffians correspond to a global section of the rank $5$ vector bundle $\O_{\pp E}(2) \otimes (\det E)^{\otimes 2} \otimes F^\vee \otimes \det F$ on $\pp E^\vee$. Equivalently, this is a global section of $(\Sym^2 E^{\vee})\otimes (\det E)^{\otimes 2} \otimes  F^{\vee} \otimes \det F$ on $\pp^1$. Working on $\pi:\P\rightarrow BG$, the Pfaffians thus correspond to a section of $\Sym^2 \V(\vec{e})\otimes \det\V(\vec{e})^{\vee\otimes 2}\otimes \V(\vec{f})^{\vee}\otimes \det\V(\vec{f})$. From the HN filtration on $\V(\vec{e}),$ there is a quotient
{\small \begin{equation} \label{qmap}
\Sym^2 \V(\vec{e})\otimes \det\V(\vec{e})^{\vee\otimes 2}\otimes \V(\vec{f})^{\vee}\otimes \det \V(\vec{f})\rightarrow \pi^*L^{\otimes 2}(2e_1)\otimes\det\V(\vec{e})^{\vee\otimes 2}\otimes \V(\vec{f})^{\vee}\otimes \det \V(\vec{f}), 
\end{equation}}
corresponding to the $X_1^2$ parts of the Pfaffians.
Note that we have 
\[
\det \V(\vec{e})^{\vee\otimes2}\cong \pi^*(\det R^{\vee\otimes 2}\otimes S^{\vee\otimes 2}\otimes L^{\vee\otimes 2})(-4e_3-2e_2-2e_1) 
\]
and
\[
\det \V(\vec{f})=\pi^*(\det M\otimes \det N\otimes T)(f_1+2f_2+2f_4).
\]
By the assumptions on the splitting type $\vec{f}$ and cohomology and base change, the $\pi$ push forward of \eqref{qmap} is
\[
\pi_*(\Sym^2 \V(\vec{e})\otimes \det\V(\vec{e})^{\vee\otimes 2}\otimes \V(\vec{f})^{\vee}\otimes \det \V(\vec{f}))\rightarrow \det M\otimes \det N\otimes \det R^{\vee \otimes 2}\otimes S^{\vee\otimes 2}.
\]
This quotient map corresponds to evaluating the coefficient of $X_1^2$ in $Q_1$. The line bundle on the right thus admits a non-vanishing section when pulled back to $[U'/G]$. This gives a relation:
\begin{equation}\label{rel1}
v^*(m_1+n_1-2s-2r_1)=0 \in A^1([U'/G]).
\end{equation}

We need one more codimension $1$ relation.  Note that if $L_{1,2}$ and $L_{1,3}$ are linearly dependent then, after change of basis (within the $\O(f_2) \oplus \O(f_3)$ part of $F$), we can assume $L_{1,2} = 0$. Then \eqref{extrapencil} shows that the resulting curve would have a $g^2_4$, which is impossible. Therefore, $L_{1,2}$ and $L_{1,3}$ must be linearly independent. Conditions (1), (2) and (4) imply that the degrees of $a_{1,2}, a_{1,3}, b_{1,2}, b_{1,3}$ are negative, so $L_{1,2}$ and $L_{1,3}$ are dependent if and only if
\[
c_{1,2}d_{1,3}-c_{1,3}d_{1,2} = 0.
\]
Below, we construct a morphism of vector bundles whose vanishing locus is the locus where $L_{1,2}$ and $L_{1,3}$ become dependent. From the HN filtration, we have a series of surjections
\[
\wedge^2 \V(\vec{f})\rightarrow \wedge^2 W_2\rightarrow \pi^*(N\otimes T)(f_1+f_2).
\]
Tensoring with $\V(\vec{e})\otimes \det \V(\vec{e})^{\vee}$ and pushing forward, we have
\[
\pi_*(\wedge^2 \V(\vec{f})\otimes \V(\vec{e})\otimes \det \V(\vec{e})^{\vee})\rightarrow \pi_*(\pi^*N\otimes \pi^*T\otimes\V(\vec{e})\otimes \det \V(\vec{e})^{\vee})(f_1+f_2).
\]
Because $f_1+f_2-g-4+e_i < 0$  for $i < 3$ and
$f_1 + f_2 - g - 4 + e_3 = 0$,
by cohomology and base change, the above map is
\[
\pi_*(\wedge^2 \V(\vec{f})\otimes \V(\vec{e})\otimes \det \V(\vec{e})^{\vee})\rightarrow \det R^{\vee}\otimes S^{\vee}\otimes L^{\vee}\otimes R \otimes N \otimes T.
\]
This map corresponds to projection onto the tuple of coefficients $(c_{1,2}, d_{1,2}, c_{1,3}, d_{1,3})$. Note that $R\otimes \det R^{\vee}\cong R^{\vee}$, so we can identify the section we obtained from the above map as a morphism
\[
S\otimes L\otimes R\rightarrow N\otimes T.
\]
Taking the determinant of this morphism, we have
\[
S^{\otimes 2}\otimes L^{\otimes 2}\otimes \det R\rightarrow \det N\otimes T^{\otimes 2},
\]
and if this determinant morphism vanishes, then $L_{1,2}$ and $L_{1,3}$ are dependent. Therefore, we obtain the relation
\begin{equation}\label{rel2}
    v^*(-n_1-2t+r_1+2s+2\ell)=0.
\end{equation}
Augmenting matrix \eqref{mat1} by the relations \eqref{rel1} and \eqref{rel2}, we the following matrix
\begin{equation}\label{matrix1}
\left(\small\begin{matrix}
 0\\
 0\\
 a_1 \\
 a_2' \\
 b_1 \\
 b_2' \\
\end{matrix} \right)=
\left(\small\begin{matrix}
0 & -2 & -2 & 0 & 1 & 1 \\
2 & 1 & 2 & -2 & 0 & -1\\
1 & 1 & 1  & 0 & 0 & 0 \\
e_2+2e_3 & e_1+e_2+e_3 & e_1+2e_3 & 0 & 0 & 0 \\
0 & 0 & 0 & 1 & 1 & 1 \\
0 & 0 & 0 & 2f_4+2f_2 & f_4+2f_2+f_1 & 2f_4+f_2+f_1
\end{matrix} \right)
\left(\small\begin{matrix} \ell \\ r_1 \\ s \\ t \\ m_1 \\ n_1 \end{matrix}\right).
\end{equation}
The determinant of the above $6\times 6$ matrix is the quantity in part \eqref{det2} of the statement of the lemma,
which does not vanish by assumption. Hence, $\ell, r_1, s, t, m_1, n_1 \in A^1([U'/G])$ are expressible in terms of CE classes.

In addition to the products of codimension $1$ generators, $A^2(BG)$ has four codimension $2$ generators: $c_2,r_2,n_2,m_2$. By definition $c_2$ is a CE class, so it remains to show that $r_2, n_2$ and $m_2$ are expressible in terms of CE classes. From the HN filtrations and the splitting principle, we obtain the following expression for $a_2$:
\[
a_2=-(e_1e_2+2e_1e_3+2e_2
     e_3+e_3^2)c_2+\ell r_1+\ell s+r_1s+r_2,
\]
from which it follows that $r_2$ is expressible in terms of CE classes. Similarly, we obtain the following expressions for $b_2$ and $b_3'$: 

\begin{equation*}
    \begin{split}
        b_2&=-(f_4^2+4f_4
      f_2+f_2^2+2f_4f_1+2f_2f_1)c_2+tm_1+tn_1+ m_1n_1+m_2+n_2\\
        b_3'&=-(2f_4^2f_2+2f_4f_2^2+f_4^2f_1+4f_4f_2f_1+f_2^2f_1)c_2+(f_4+2f_2)tm_1+(2f_4+f_2)tn_1+ \\ &\qquad (f_4+f_2+f_1)m_1n_1+(2f_2+f_1)m_2+(2f_4+f_1)n_2.
    \end{split}
\end{equation*}
Because $f_2\neq f_4$, we see that $m_2$ and $n_2$ are expressible in terms of CE classes.
\end{proof}

We consider one more pair of shapes of splitting types.
\begin{lem}\label{forsigma3}
Let $\vec{e}=(e_1,e_2,e_3,e_4)$ and $\vec{f}=(f_1,f_2,f_3,f_4,f_5)$ satisfy the following conditions:
\begin{enumerate}
    \item $e_1<e_2=e_3<e_4$,
    \item $f_1<f_2=f_3<f_4=f_5$,
    \item $e_1+f_2+f_5=g+4$,
    \item $e_2+f_1+f_4=g+4$,
    \item \label{det3} $-2e_2f_1 + 2e_4f_1 + e_1f_2 - 2e_2f_2 + e_4f_2 - e_1f_4 + 4e_2f_4 - 3e_4f_4\neq 0$.
\end{enumerate}
Let $\Sigma$ be the $\vec{e}, \vec{f}$ splitting locus. Then $A^*(\Sigma')$ is generated by the restrictions of CE classes.
\end{lem}
\begin{proof}
Set $G=\SL_2 \ltimes (\Aut(\vec{e})\times \Aut(\vec{f}))$ and let $\pi: \P \rightarrow BG$ be the $\pp^1$ bundle pulled back from $\BSL_2$.
The HN filtration for $\V(\vec{f})$ is given by
\[
0\rightarrow \pi^*S(e_4)\rightarrow \V(\vec{e})\rightarrow W_1\rightarrow 0
\]
and
\[
0\rightarrow \pi^*R(e_2)\rightarrow W_1\rightarrow \pi^*L(e_1)\rightarrow 0.
\]
The filtration for $\V(\vec{f})$ is the same as in the previous Lemma \ref{forsigma2}. As a result, we have the following expressions for the Casnati--Ekedahl classes in terms of the generators of the Chow ring of this stratum:
\begin{equation} \label{mat2}
\left(\small\begin{matrix}
 a_1 \\
 a_2' \\
 b_1 \\
 b_2' \\
\end{matrix} \right)=
\left(\small\begin{matrix}
1 & 1 & 1  & 0 & 0 & 0 \\
2e_2+e_4 & e_1+e_2+e_4 & e_1+2e_2 & 0 & 0 & 0 \\
0 & 0 & 0 & 1 & 1 & 1 \\
0 & 0 & 0 & 2f_4+2f_2 & f_4+2f_2+f_1 & 2f_4+f_2+f_1
\end{matrix} \right)
\left(\small\begin{matrix} \ell \\ r_1 \\ s \\ t \\ m_1 \\ n_1 \end{matrix}\right).
\end{equation}

As in the previous lemma, there are $6$ codimension $1$ generators for $A^*(BG)$, but only $4$ codimension $1$ CE classes. We will need to show $A^1(BG) \to A^1([U'/G])$ has a kernel, meaning we have relations between the generators. The first relation is quite similar to the first relation from the previous Lemma \ref{forsigma2}. Not all of the quadrics cutting out the curve can vanish on $Z=V(X_2,X_3,X_4)$. We see that upon restriction to $Z$ all of the quadrics vanish, except for possibly $Q_1$, which is given by
\[
Q_1|_{Z}=(a_{2,5}a_{3,4}-a_{2,4}a_{3,5})X_1^2.
\]
As in Lemma \ref{forsigma2}, there is a quotient map 
\[
\Sym^2 \V(\vec{e})\otimes \det\V(\vec{e})^{\vee\otimes 2}\otimes \V(\vec{f})^{\vee}\otimes \det \V(\vec{f})\rightarrow \pi^*L^{\otimes 2}(2e_1)\otimes\det\V(\vec{e})^{\vee\otimes 2}\otimes \V(\vec{f})^{\vee}\otimes \det \V(\vec{f}),
\]
which corresponds to the coefficients of $X_1^2$ in the Pfaffians.
Note that we have 
\[
\det \V(\vec{e})^{\vee\otimes 2}\cong \pi^*(L^{\vee \otimes 2}\otimes S^{\vee\otimes 2}\otimes R^{\vee\otimes 2})(-2e_1-2e_4-4e_2)
\]
and
\[
\det \V(\vec{f})=\pi^*(\det M\otimes \det N\otimes T)(f_1+2f_2+2f_4).
\]
From cohomology and base change and the filtration on $\V(\vec{f})$, we see that the $\pi$ push forward of this map is
\[
\pi_*(\Sym^2 \V(\vec{e})\otimes \det\V(\vec{e})^{\vee\otimes 2}\otimes \V(\vec{f})^{\vee}\otimes \det \V(\vec{f}))\rightarrow\det M\otimes \det N\otimes S^{\vee\otimes 2}\otimes R^{\vee\otimes 2}.
\]
This quotient map corresponds to the coefficient of $X_1^2$ in $Q_1$. The non-vanishing of this coefficient
means that the pullback along $v:[U'/G] \to BG$ of the line bundle on the right has a non-vanishing section. This gives us the relation
\begin{equation}\label{relation1}
    v^*(m_1+n_1-2r_1-2s)=0 \in A^1([U'/G]).
\end{equation}

The next relation comes from considering the equations for the curve when restricted to $V(X_4)$. For degree reasons, $L_{1,2}$ and $L_{1,3}$ vanish when restricted to $V(X_4)$. Suppose that $L_{1,4}$ and $L_{1,5}$ are dependent. Then, the quadrics $Q_2,\dots, Q_5$ all vanish along $V(L_{1,4},X_4)=V(L_{1,5},X_4)$. It follows that $V(Q_1,L_{1,4},X_4)$ is contained in the curve. However, $\p E^{\vee}$ has dimension $4$, so the locus  $V(Q_1,L_{1,4},X_4)$ has dimension at least $1$. This means that $C$ would be contained in $V(X_4)$, which is impossible. Therefore, the restrictions of $L_{1,4}$ and $L_{1,5}$ to $V(X_4)$ must be independent. Because $e_1 + f_1 + f_5 - g - 4 < 0$, we have
\[L_{1,4}|_{V(X_4)} = b_{1,4} X_2 + c_{1,4} X_3 \qquad \text{and} \qquad 
L_{1,5}|_{V(X_4)} = b_{1,5} X_2 + c_{1,5} X_3.\]
Therefore, $L_{1,4}|_{V(X_4)}$ and $L_{1,5}|_{V(X_4)}$ are dependent if and only if
\[
b_{1,4}c_{1,5}-b_{1,5}c_{1,4}=0.
\]

As in Lemma \ref{forsigma2}, we construct a morphism of vector bundles whose vanishing locus is the locus where $L_{1,4}|_{V(X_4)}$ and $L_{1,5}|_{V(X_4)}$ are dependent. From the HN filtration and the corresponding filtration on $\wedge^2 \V(\vec{f})$, we have a surjection
\begin{equation} \label{vftok}
\wedge^2 \V(\vec{f})\rightarrow K,
\end{equation}
where $K$ is a vector bundle admitting a filtration
\[
0\rightarrow \pi^*M(f_4)\otimes \pi^*T(f_1)\rightarrow K\rightarrow \wedge^2 W_2\rightarrow 0.
\]
Tensoring \eqref{vftok} with the map $\V(\vec{e})\otimes \det \V(\vec{e})^{\vee}\rightarrow W_1\otimes \det \V(\vec{e})^{\vee}$, we obtain a surjection
\begin{equation}\label{tokw}
\wedge^2 \V(\vec{f})\otimes \V(\vec{e})\otimes \det \V(\vec{e})^{\vee} \rightarrow K\otimes W_1\otimes \det \V(\vec{e})^{\vee}.
\end{equation}
By cohomology and base change and the assumptions on the splitting types $\vec{e}$ and $\vec{f}$, the $\pi$ push forward of \eqref{tokw} is given by
\[
\pi_*(\wedge^2 \V(\vec{f})\otimes \V(\vec{e})\otimes \det \V(\vec{e})^{\vee})\rightarrow R\otimes \det R^{\vee}\otimes L^{\vee}\otimes  S^{\vee}\otimes M\otimes T.
\]
This map corresponds to projection onto the tuple of coefficients $(b_{1,4},c_{1,4},b_{1,5},c_{1,5})$. Since $R$ has rank $2$, we have $R \otimes \det R^\vee \cong R^\vee$.
The section obtained from the above map can be identified with a morphism
\[
 R\otimes L\otimes S\rightarrow M\otimes T.
\]
The associated determinant morphism
\[
\det R \otimes L^{\otimes 2}\otimes S^{\otimes 2}
\rightarrow
\det M\otimes T^{\otimes 2}
\]
vanishes precisely when $b_{1,4}c_{1,5}-b_{1,5}c_{1,4}=0$. 
Since this quantity is non-vanishing on $[U'/G]$, we obtain the relation
\begin{equation}\label{relation2}
    v^*(2t+m_1-2s-2\ell-r_1)=0.
\end{equation}

We augment the matrix \eqref{mat2} by the relations \eqref{relation1}, \eqref{relation2} to obtain
\begin{equation} \label{matrixaug1}
\left(\small\begin{matrix}
 0\\
 0\\
 a_1 \\
 a_2' \\
 b_1 \\
 b_2' \\
\end{matrix} \right)=
\left(\small\begin{matrix}
0 & -2 & -2 & 0 & 1 & 1\\
-2 & -1 & -2 & 2 & 1 & 0\\
1 & 1 & 1  & 0 & 0 & 0 \\
2e_2+e_4 & e_1+e_2+e_4 & e_1+2e_2 & 0 & 0 & 0 \\
0 & 0 & 0 & 1 & 1 & 1 \\
0 & 0 & 0 & 2f_4+2f_2 & f_4+2f_2+f_1 & 2f_4+f_2+f_1
\end{matrix} \right)
\left(\small \begin{matrix} \ell \\ r_1 \\ s \\ t \\ m_1 \\ n_1 \end{matrix}\right).
\end{equation}
The determinant of this matrix is the quantity in part \eqref{det3} of the statement of the lemma. It is non-vanishing by assumption, so on $[U'/G]$, the classes $\ell, r_1, s, t, m_1, n_1$ are expressible in terms of the CE classes.

Besides products of codimension $1$ classes, $A^*(BG)$ has $4$ codimension $2$ generators: $c_2, r_2, m_2, n_2$. Using the splitting principle and the HN filtrations, we obtain the following expression for $a_2$:
\[
a_2=-(2e_1e_2+e_2^2+e_1
      e_4+2e_2e_4)c_2+\ell r_1+\ell s+r_1s+r_2.
\]
It follows that $r_2$ is expressible in terms of CE classes. Similarly, we obtain the following expressions for $b_2$ and $b_3'$.
\begin{equation*}
    \begin{split}
       b_2&=-(f_4^2+4f_4
      f_2+f_2^2+2f_4f_1+2f_2f_1)c_2+tm_1+tn_1+ m_1n_1+m_2+n_2\\
        b_3'&=-(2f_4^2f_2+2f_4f_2^2+f_4^2f_1+4f_4f_2f_1+f_2^2f_1)c_2+(f_4+2f_2)tm_1+(2f_4+f_2)tn_1+ \\ &\qquad (f_4+f_2+f_1)m_1n_1+(2f_2+f_1)m_2+(2f_4+f_1)n_2.
    \end{split}
\end{equation*}
Because $f_4\neq f_2$, both $m_2$ and $n_2$ are expressible in terms of CE classes.
\end{proof}

\subsection{Genus $7$} \label{sec57}
Applying the constraints in \eqref{5eq1} -- \eqref{beforeg26}, one obtains a stratification of $\H_{5,7}$ based on the allowable splitting types of $\E$ and $\F$. This stratification was obtained by Schreyer \cite[p. 133]{S}, and we translate it here into our notation. The claimed special linear series (which are also listed in Schreyer's table) can be seen from \eqref{ourg26} and \eqref{extrapencil}. \begin{lem}[Schreyer]
Let $g = 7$.
There are $5$ allowed pairs of splitting types for the bundles $E$ and $F$. They give rise to the following stratification of $\H_{5,7}$:
\begin{enumerate}
    \item[($\Psi_0$)] $E = (2, 3, 3, 3), F = (4, 4, 4, 5, 5)$: the general stratum.
    \item[($Z_1$)] $E = (2, 2, 3, 4), F = (4, 4, 4, 5, 5)$: such curves possess a $g^1_4$.
    \item[($Z_2$)] $E = (2, 3, 3, 3), F = (3, 4, 5, 5, 5)$: such curves possess a $g^2_6$.
    \item[($Z_3$)] $E = (2, 2, 3, 4), F = (3, 4, 4, 5, 6)$: such curves possess a $g^2_6$.
    \item[($Z_4$)] $E = (2, 3, 3, 3), F = (3, 3, 5, 5, 6)$: such curves possess a $g^1_3$.
\end{enumerate}
\end{lem}
As our labeling suggests,
by a happy coincidence, all strata outside of the ``good open" $\Psi$ actually lie inside $\beta^{-1}(\M_7^4)$.
\begin{cor} \label{7pent}
The Chow ring of $\H_{5, 7} \smallsetminus \beta^{-1}(\M_7^4)$ is generated by the restrictions of tautological classes. Hence, all classes supported on $\M_7 \smallsetminus \M_7^4$ are tautological.
\end{cor}
\begin{proof}
It this case, we have $\Psi = \Psi_0$. Applying Proposition \ref{Hprime}, it suffices to show that all other strata $Z_i$ are contained in $\beta^{-1}(\M_7^4)$. This follows immediately for $Z_1$ and $Z_4$.
Suppose that a $7$ curve $C$ possesses a $g^2_6$. Degree $6$ plane curves have arithmetic genus $10$. If the $g^2_6$ sends $C$ birationally onto its image, then the image must have a double point (or worse). Projection from such a point gives a $g^1_4$ (or $g^1_k$ for $k < 4$). Otherwise, the $g^2_6$ sends $C$ with degree three onto a conic (so $C$ has a $g^1_3$) or with degree two onto a cubic. Every cubic admits a degree $2$ map to $\pp^1$ (by projecting from a point) so composing these two degree $2$ maps, we see that $C$ has a $g^1_4$. Thus, $Z_2$ and $Z_3$ are also contained in $\beta^{-1}(\M_7^4)$.
\end{proof}

Combining Corollary \ref{7pent} with Lemma \ref{H47}
completes the proof of Theorem \ref{main}. There is still some work to do in genus $8$ and $9$.

\subsection{Genus $8$} \label{sec58}
The constraints \eqref{5eq1}--\eqref{beforeg26} from the beginning of the section give a stratification of $\H_{5,8}$, which was first observed by Schreyer \cite[p. 133]{S}. The claimed linear series can be seen from \eqref{ourg26} and \eqref{extrapencil}. The codimensions of strata are determined by \eqref{codimeq5}.
\begin{lem}[Schreyer]
Let $g = 8$.
There are $7$ allowed pairs of splitting types for the bundles $E$ and $F$. They give rise to the following stratification of $\H_{5,8}$:
\begin{enumerate}
    \item[($\Psi_0$)] $E = (3, 3, 3, 3), F = (4, 5, 5, 5, 5)$: the general stratum.
    \item[($\Psi_1$)] $E = (2, 3, 3, 4), F = (4, 5, 5, 5, 5)$: codimension $1$. 
    \item[($\Sigma_2$)] $E = (2, 3, 3, 4), F = (4, 4, 5, 5, 6)$: codimension $2$. 
    \item[($Z_3$)] $E = (3, 3, 3, 3), F = (4, 4, 5, 5, 6)$: such curves possess a $g^1_4$.
    \item [($Z_4$)] $E = (2, 2, 4, 4), F = (4, 4, 4, 6, 6)$: such curves possess a $g^1_4$.
    \item [($Z_5$)] $E = (2, 3, 3, 4), F = (3, 4, 5, 6, 6)$: such curves possess a $g^2_6$.
    \item [($Z_6$)] $E = (3, 3, 3, 3), F = (3, 3, 6, 6, 6)$: such curves possess a $g^1_3$.
\end{enumerate}
\end{lem}

This time there is a stratum, $\Sigma_2$, which lives outside $\Psi$ and not inside $\beta^{-1}(\M_8^4)$. Nevertheless, using arguments similar to Lemmas \ref{H47} and \ref{H48}, we have the following.
\begin{lem} \label{H58}
The Chow ring $A^*(\H_{5,8} \smallsetminus \beta^{-1}(\M_8^4))$ is generated by CE classes. Hence, all classes supported on $\M_8 \smallsetminus \M_8^4$ are tautological.
\end{lem}
\begin{proof}
By a similar argument to the proof of Corollary \ref{7pent}, every genus $8$ curve possessing a $g^2_6$ also possesses a $g^1_k$ for $k \leq 4$. In particular,
we see that $Z=Z_3\cup Z_4\cup Z_5\cup Z_6$ is contained in $\beta^{-1}(\M^4_8)$. 
Next, on $\H_{5,8}\smallsetminus Z$, we have $\Sigma_2=\Sigma_{(4,4,5,5,6)}(\F).$
Hence, we have
\[
\codim \Sigma_2=2=h^1(\p^1,\E nd(\O(4,4,5,5,6))).
\]
Thus, by Lemma \ref{degenformulas}, the fundamental class of $\Sigma_2$ inside $\H_{5,8}\smallsetminus \beta^{-1}(\M^4_8)$ is tautological. By Lemma \ref{shape1}, $A^*(\Sigma_2')$ is generated by tautological classes. It then follows from the push-pull formula that every class supported on $\Sigma_2'\subset \H_{5,8}\smallsetminus \beta^{-1}(\M^4_8)$ is tautological. By Proposition \ref{good}, we know $A^*(\Psi)=A^*(\Psi_0\cup \Psi_1)$ is generated by tautological classes. It follows that all of $A^*(\H_{5,8}\smallsetminus \beta^{-1}(\M^4_8))$ is generated by tautological classes.
\end{proof}
Combining Lemma \ref{H58} with Lemma \ref{H48} completes the proof of Theorem \ref{main8}. The rest of the paper will deal with the case $g=9$.

\subsection{Genus $9$} \label{sec59}
There is a similar stratificaion in genus $9$, which was given by Sagraloff \cite{Sag}. Below, we translate Sagraloff's notation into ours. The stratification can be obtained from the conditions \eqref{5eq1}--\eqref{beforeg26}, and the claimed linear series can be seen from \eqref{ourg26} and \eqref{extrapencil}. The codimensions of strata are determined by \eqref{codimeq5}.
\begin{lem}[Sagraloff]
Let $g=9$. There are $7$ allowed pairs of splitting types for the bundles $E$ and $F$. They give rise to the following stratification of $\H_{5,9}$:
\begin{enumerate}
     \item[($\Psi_0$)] $E = (3, 3, 3, 4), F = (5, 5, 5, 5, 6)$: the general stratum.
    \item[($\Psi_1$)] $E=(3,3,3,4)$, $F=(4,5,5,6,6)$: codimension $2$. 
    \item[($\Sigma_2$)] $E=(2,3,4,4)$, $F=(4,5,5,6,6)$: codimension $2$. 
    \item[($\Sigma_3$)] $E=(2,3,3,5)$, $F=(4,5,5,6,6)$: codimension $4$.
    \item[($Z_4$)] $E=(3,3,3,4)$, $F=(4,4,6,6,6)$: such curves possess a $g^1_4$.
    \item[($Z_5$)] $E=(2,3,4,4)$, $F=(4,4,5,6,7)$: such curves possess a $g^1_4$.
    \item[($Z_6$)] $E=(2,3,4,4)$, $F=(3,4,6,6,7)$: such curves possess a $g^2_6$.
\end{enumerate}
\end{lem}
\begin{lem} \label{H59}
The Chow ring $A^*(\H_{5,9}\smallsetminus \beta^{-1}(\M^4_9))$ is generated by tautological classes. Hence, all classes supported on $\M_9^5$ are tautological.
\end{lem}
\begin{proof}
First, we see that $Z=Z_4\cup Z_5\cup Z_6$ is contained in $\beta^{-1}(\M^4_9)$. Then, note that on $\H_{5,9}\smallsetminus Z$, we have that $\Sigma_3=\Sigma_{(2,3,3,5)}(\E)$. Moreover, we see that
\[
\codim \Sigma_3=4=h^1(\p^1,\E nd(\O(2,3,3,5))).
\]
By Lemma \ref{degenformulas}, it follows that the fundamental class of $\Sigma_3$ is tautological. By Lemma \ref{forsigma3}, we see $A^*(\Sigma_3')$ is generated by tautological classes, so by the push-pull formula, every class supported on $\Sigma_3'\subset \H_{5,9}\smallsetminus \beta^{-1}(\M^4_9)$ is tautological.

Similarly, on $\H_{5,9}\smallsetminus \beta^{-1}(\M^4_9)$, we have $\Sigma_2=\Sigma_{(2,3,4,4)}(\E)$, and
\[
\codim \Sigma_2=2=h^1(\p^1,\E nd(\O(2,3,4,4))).
\]
Applying Lemma \ref{degenformulas}, the fundamental class of $\Sigma_2$ is tautological. Applying Lemma \ref{forsigma2}, we see that every class supported on $\Sigma_2'\subset \H_{5,9}\smallsetminus \beta^{-1}(\M^4_9)$ is expressible in terms of tautological classes. 
By Proposition \ref{good}, $A^*(\Psi_0\cup \Psi_1)$ is generated by tautological classes. Therefore, $A^*(\H_{5,9}\smallsetminus \beta^{-1}(\M^4_9))$ is generated by tautological classes. 
\end{proof}

\section{The General Genus 9 Curve}\label{Gen}
Mukai \cite{Muk} completely characterized canonical curves of genus $9$ without a $g^1_5$ as linear sections of a symplectic Grassmannian. We briefly recall his construction here.
Let $V$ be a six-dimensional vector space equipped with a symplectic form $\sigma$. The \emph{symplectic Grassmannian} $Sp(3, V) \subset G(3, V)$ parametrizes three-dimensional \emph{symplectic subspaces} $U \subset V$, i.e. subspaces such that $\sigma|_{U} = 0$.
The Grassmannian $G(3, V)$ embeds in $\pp(\wedge^3 V) \cong \pp^{19}$ via the Pl\"ucker embedding. Contracting with the symplectic form gives a map
\[\sigma^\sharp: \wedge^3 V \to V,\]
and the symplectic Grassmannian is the intersection of $G(3, V)$ with $\pp(\ker \sigma^\sharp) \subset \pp(\wedge^3 V)$. Note that the subspace $\pp(\ker \sigma^\sharp) \subset \pp(\wedge^3 V)$ corresponds to subspace of symmetric matrices in Mukai's description of the Pl\"ucker embedding \cite[p. 1544]{Muk}.

Recall that we use the subspace convention for Grassmannians and projective spaces. For example, given a globally generated rank $3$ vector bundle $E$ on $C$, the evaluation map $H^0(E) \to E$ determines a map $C \to G(3, H^0(E)^\vee)$ by considering the dual $E^\vee \rightarrow H^0(E)^\vee$. Similarly, the canonical embedding sends a curve $C \hookrightarrow \pp (H^0(\omega_C)^\vee)$.
The following is an amalgamation of Mukai's Theorems A, B, and C of \cite{Muk}.

\begin{thm}\label{Mukai}
Suppose $C$ is a smooth curve of genus $9$ with no $g^1_5$. Then there is a unique rank $3$ vector bundle $E$ on $C$ with the following properties:
\begin{enumerate}
    \item $\det E\cong \omega_C$.
    \item $h^0(C, E)=6$.
    \item $E$ is globally generated and for every $3$-dimensional subspace $U\subset H^0(E)$, the evaluation homomorphism $U\otimes \O_{C}\rightarrow E$ is injective or everywhere of rank $2$.
\end{enumerate}
The bundle $E$ induces a morphism $C\rightarrow G(3,H^0(E)^\vee)$ whose image is contained in the symplectic Grassmannian $Sp(3,H^0(E)^\vee) \subset G(3, H^0(E)^\vee)$. The curve $C$ in its canonical embedding is obtained by intersecting $Sp(3,H^0(E)^\vee)\subset \pp(\ker \sigma^\sharp) \cong \p^{13}$ with an eight dimensional linear subspace $\p^8 \subset\pp^{13}$. Such a linear space is unique up to the action of $\PGSp_6$, the subgroup of $\PGL_6$ fixing the one dimensional space spanned by a symplectic form.
\par Moreover, a canonical curve $C$ of genus $9$ is the intersection $\p^8\cap Sp(3,6)$ if and only if $C$ has no $g^1_5$.
\end{thm}

Let $\Delta\subset G(9,14) = \mathbb{G}(8, 13)$ be the locus of linear subspaces whose intersection with $Sp(3,6) \subset \pp(\ker \sigma^\sharp) = \pp^{13}$ is not a smooth genus $9$ curve.
The above theorem provides a morphism
\begin{equation} \label{phieq}
\phi:[G(9,14)\smallsetminus \Delta/\PGSp_6]\rightarrow \M_9\smallsetminus \M^5_9.
\end{equation}
We wish to show that $\phi$ is an isomorphism. The basic idea of our proof is modeled after \cite[Theorem 5.7]{PV}.
In particular, we make use of the following standard lemma, whose proof we include for completeness.

\begin{lem} \label{replem}
Let $f: \mathcal{X} \to \mathcal{Y}$ be a separated morphism of connected smooth Deligne--Mumford stacks that are of finite type over a field. 
Suppose that
\begin{enumerate}
\item the characteristic of the ground field is zero,
\item $f$ induces an isomorphism on stabilizer groups of geometric points, and
\item $f$ induces a bijection on geometric points.
\end{enumerate}
Then $f$ is an isomorphism.
\end{lem}
\begin{proof}
Let $a: Y \to \Y$ be a connected, smooth cover by a smooth scheme $Y$. 
Let $X := \X \times_{\Y} Y$ be the fiber product, so we have a diagram
\begin{center}
\begin{tikzcd}
X \arrow{d}[swap]{b} \arrow{r}{g} & Y  \arrow{d}{a}\\
\mathcal{X} \arrow{r}[swap]{f} & \mathcal{Y}.
\end{tikzcd}
\end{center}
Suppose $x: \Spec k \to X$ is a geometric point. The stabilizer group $G_x$ of $x$ is equal to the fiber product of stabilizer groups $G_{b(x)} \times_{G_{f(b(x))}} G_{g(x)}$. But $Y$ is a scheme, so $G_{g(x)}$ is trivial. Hence, $G_x = \ker(G_{b(x)} \to G_{f(b(x)})$, which is trivial by hypothesis.
By \cite[Theorem 2.2.5]{Conrad}, it follows that $X$ is an algebraic space.
Further, the map $f': X \to Y$ is quasi-finite and separated so by \cite[\href{https://stacks.math.columbia.edu/tag/03XX}{Tag 03XX}]{stacks-project}, we know $X$ is a scheme. 
Because $f$ induces an isomorphism on stabilizer groups of geometric points, the map $f': X \to Y$ is a bijection on geometric points.
Because $a$ is smooth and $\X$ is smooth and connected, we know $X$ is also smooth and connected.
Working in characteristic zero, the map $f'$ is generically smooth, hence birational. Now, Zariski's Main Theorem shows that $f': X \to Y$ is an isomorphism.
\end{proof}

\begin{lem} \label{stablem}
Suppose the characteristic of the ground field is zero. The map $\phi$ induces an isomorphism on stabilizer groups of geometric points.
\end{lem}
\begin{proof}
In characteristic zero, finite group schemes are smooth, so it suffices to show the map induces a bijection on the finite stabilizer groups.
Suppose $x: \Spec k \to [G(9, 14) \smallsetminus \Delta/\PGSp_6]$ is a geometric point. Such a point is the data of $(V, \sigma, W)$ where $V$ is a six-dimensional vector space, $\sigma$ is a symplectic form remembered up to scaling and $ W \subset \ker \sigma^\sharp \subset \wedge^3 V$ is a $9$ dimensional subspace.
The stabilizer group of $x$ is the subgroup of elements $\gamma \in \PGSp_6 \subset \PGL_6$ that send 
$W \subset \ker \sigma^\sharp \subset \wedge^3 V$ into itself.
The image $\phi(x)$ is the genus $9$ curve
\[C := \pp W \cap Sp(3, V) \subset \pp(\ker \sigma^\sharp) \subset \pp(\wedge^3 V).\]
The
automorphism group of $\phi(x)$ is the automorphism group of $C$.

To see that $\phi$ induces an injection on these stabilizer groups, suppose $\gamma \in \PGSp_6$ induces the identity on $C$. Let $E^\vee \to V$ be the restriction of the tautological sequence on $Sp(3, V)$ to $C$. 
By \cite[Section 4]{Muk},
the bundle $E$ is the unique rank $3$ bundle of Mukai's Theorem \ref{Mukai} and 
$E^\vee \to V$ is dual to the evaluation map $H^0(E) \to E$.
Let $\tilde{\gamma} \in \mathrm{GSp}_6 \subset \GL_6$ be a lift of $\gamma$.
Since $\gamma$ induces the identity on $C \subset Sp(3, V)$, there must exist an automorphism $\epsilon$ of $E$ on $C$ so that the diagram below commutes
\begin{equation} \label{eg}
\begin{tikzcd}
E^\vee \arrow{d}[swap]{\epsilon^\vee} \arrow{r} & V  \arrow{d}{\tilde{\gamma}} \\
E^\vee \arrow{r} & V.
\end{tikzcd}
\end{equation}
Above, the horizontal maps are the same (restricted from the tautological sequence on $Sp(3, V)$)
In \cite[Proposition 3.5(3)]{Muk}, Mukai showed that the only automorphisms of $E$ are scalars, so $\epsilon$ is a scalar. For \eqref{eg} to commute, the map $\tilde{\gamma}$ must be the dual of the effect of $\epsilon$ on global sections. Hence, $\tilde{\gamma}$ is also a scalar, so $\gamma$ is the identity.

To see that $\phi$ induces a surjection on stabilizer groups, suppose we have an automorphism $i: C \to C$. We need to find an element $\gamma \in \PGSp_6(S)$ that induces $i$ on $C \subset \pp(\ker \sigma^\sharp) \subset \pp(\wedge^3 V)$. 
Let $E$ on $C$ be the restriction of the tautological bundle on $Sp(3, V)$, which is the Mukai bundle.
The tautological surjection $V \to E$ induces an isomorphism $V \cong H^0(E)$.
The pullback $i^*E$ has the properties of Theorem \ref{Mukai}, so Mukai's uniqueness tells us $i^*E \cong E$,
and moreover, this isomorphism is unique up to scaling \cite[Proposition 3.5(3)]{Muk}. Now $i^*$ gives rise to an automorphism
\[\gamma^\vee: V^\vee \cong H^0(E) \xrightarrow{i^*} H^0(i^*E) \cong H^0(E) \cong V^\vee,\]
which is well-defined up to scaling, and preserves the symplectic form up to scaling. By construction, the dual of this element, $\gamma \in \PGSp_6$ induces the automorphism $i : C \to C$.
\end{proof}

\begin{lem}\label{sep}
The quotient $[G(9, 14) \smallsetminus \Delta/\PGSp_6]$ is separated.
\end{lem}
\begin{proof}
By Mukai \cite[Lemma 4.1]{Muk}, if the intersection $\pp^8 \cap Sp(3, 6)$ is smooth of the expected dimension $1$, then it is  a genus $9$ curve. Thus, $\Delta$ is the locus of linear spaces whose intersection with $Sp(3, 6)$ has a point with tangent space of dimension $2$ or more.
Considering the incidence correspondence
\[\{(p, \Lambda) \in Sp(3, 6) \times G(9, 14) : \dim (\pp \Lambda \cap \mathbb{T}_p Sp(3, 6)) \geq 2\}, \]
one sees that $\Delta$ is an irreducible divisor. Let $L = \O(\Delta)$ be the corresponding ample line bundle on $G(9, 14)$.

Let $V$ be a $6$-dimensional vector space equipped with a symplectic form $\sigma$.
The group $G:= \PGSp_6$ acts on $G(9, 14) = G(9, \ker \sigma^\sharp)$ via the $14$-dimensional representation $\ker \sigma^\sharp$.
Let $X = G(9, 14) \smallsetminus \Delta$.
We claim that the orbit of every point in $X$ is closed in $X$. Indeed suppose $x'$ is in the closure of the orbit of $x \in X$.  The orbit of $x$ corresponds to a constant family of a curve $[C] \in \M_9 \smallsetminus \M_9^5$. If $x' \in X$ is in the closure of the orbit of $x$, the intersection of the corresponding linear space with $Sp(3, 6)$ is a smooth curve $C'$ in the closure of the constant family of $C$, so $C' = C$. By Mukai's Theorem \ref{Mukai}, $x'$ is in the orbit of $x$. Because the orbits are closed, $X$ is contained in the stable locus of the action of $G$ on $G(9, 14)$ with respect to $L$ (see \cite[Definition 1.7(c)]{GIT}). 

By Lemma \ref{stablem}, the stabilizers of $G$ acting on $X$ are all finite. Therefore, \cite[Theorem 4.18]{edidin} shows that the action of $G$ on $X$ is proper. By \cite[Proposition 4.17]{edidin} the quotient stack $[X/G]$ is separated.
\end{proof}

\begin{cor}
Assume the characteristic of the ground field is $0$.
The map $\phi$ in \eqref{phieq} is an isomorphism.
\end{cor}
\begin{proof}
Mukai's Theorem \ref{Mukai} says that $\phi$ induces a bijection on geometric points. By Lemma \ref{sep}, the source of $\phi$ is separated, and hence the map $\phi$ is separated.
By Lemma \ref{stablem}, we know $\phi$ induces an isomorphism on stabilizer groups of geometric points. Thus, $\phi$
 is an isomorphism by Lemma \ref{replem}.
 \end{proof}

\begin{rem}
In positive characteristic,
the only way $\phi$ can fail to be an isomorphism is if $\phi$ induces a purely inseparable extension of function fields.
However, even if $\phi$ is a purely inseparable extension of degree $d$, the maps $\phi^*$ and $\frac{1}{d}\phi_*$ are mutually inverse. Therefore, we still have an isomorphism of Chow rings $A^*([G(9,14)\smallsetminus \Delta/\PGSp_6]) \cong A^*(\M_9\smallsetminus \M^5_9)$, which is actually all we need for our purposes.
\end{rem}

Our task is now to compute generators for the Chow ring of $[G(9,14)\smallsetminus \Delta/\PGSp_6]$ and show that they are tautological. First, note that there is an exact sequence
\[
1\rightarrow \mu_2\rightarrow \SP_6\rightarrow \PGSp_6\rightarrow 0.
\]
It follows that
\[
[G(9,14)\smallsetminus \Delta/\SP_6]\rightarrow [G(9,14)\smallsetminus \Delta/\PSp_6]
\]
is a $\mu_2$-banded gerbe. Hence, the two stacks have isomorphic Chow rings (with $\qq$-coefficients), so we may work with the $\SP_6$ quotient instead. The stack $[G(9,14)\smallsetminus \Delta/\SP_6]$ is an open substack of a Grassmann bundle over $\BSp_6$. Therefore, its Chow ring is generated by the Chern classes of the tautological subbundle $\mathcal{S}$ of the Grassmann bundle together with the (pullbacks of) generators of the Chow ring of $\BSp_6$. Totaro \cite[Section 15]{T} computed the Chow ring of $\BSp_n$.
\begin{prop}[Totaro]
The Chow ring $A^*(\BSp_{2n})$ is isomorphic to $\zz[c_2,c_4,c_6,\dots,c_{2n}]$ where $c_{2i}$ are the classes of the standard representation (induced via $\SP_{2n}\hookrightarrow \SL_{2n}$).
\end{prop}
As a result, we obtain generators for the Chow ring of $\M_9 \smallsetminus \M_9^5$.
\begin{lem} \label{gens}
The Chow ring of $\M_9\smallsetminus \M^5_9$ is generated by the Chern classes of $\mathcal{S}$ and the Chern classes $c_2(\V),c_4(\V),c_6(\V)$, where $\V$ is the standard representation of $\SP_6$. 
\end{lem} 
First, we deal with the Chern classes $c_i(\mathcal{S})$. Let $f:\C\rightarrow \M_9\smallsetminus \M^5_9$ be the universal curve.
\begin{lem} \label{Sbundle}
The Chern classes of $\mathcal{S}$ are tautological.
\end{lem}

\begin{proof}
By Mukai's theorem, the projectivization of dual of the tautological subbundle $\p \mathcal{S}^\vee$ is identified with projectivization of the Hodge bundle $\p(f_*\omega_{f})$. Therefore, $\mathcal{S}\cong (f_*\omega_{f})^\vee \otimes \mathcal{L}$ where $\mathcal{L}$ is some line bundle on $\M_9 \smallsetminus \M_9^5$.  By a theorem of Harer \cite{H} in characteristic $0$ and Moriwaki \cite{Moriwaki} in characteristic $p$, $\Pic(\Mg)$ and hence $\Pic(\M_9 \smallsetminus \M_9^5)$ is generated by $c_1(f_*\omega_{f})$. It follows from the splitting principle that the Chern classes of $\mathcal{S}$ are tautological.
\end{proof}

Next we deal with the Chern classes $c_i(\mathcal{V})$.
Writing $f:\C\rightarrow \M_9\smallsetminus \M^5_9$ for the universal curve, let $\mathcal{I}_2(\C)$ be defined by the exact sequence
\[
0\rightarrow \mathcal{I}_2(\C)\rightarrow \Sym^2 (f_*\omega_{f}) \rightarrow f_*(\omega_f^{\otimes 2})\rightarrow 0.
\]
The bundle $\mathcal{I}_2(\C)$ is a rank $21$ bundle parametrizing the quadrics vanishing on the curve under its canonical embedding. By Petri's theorem, a nontrigonal canonical curve of genus $9$ is exactly the common zero locus of these $21$ quadrics. 
\begin{lem} \label{quadricsbundle}
The bundle $\mathcal{I}_2(\C)$ is isomorphic to $\Sym^2 \V$.
\end{lem}
\begin{proof} Because a canonical curve of genus $9$ with no $g^1_5$ is a linear section of the symplectic Grassmannian, we see that we can identify the space of quadrics vanishing on the canonical curve with the restriction to $\p^8$ of the space of quadrics defining the symplectic Grassmannian $Sp(3, V) \subset \pp(\ker \sigma^\sharp) \cong \pp^{13}$.
The symplectic Grassmannian is the zero locus of $21$ quadrics in $\p^{13}$, see \cite[Equation 0.1]{Muk}.
That is, $\mathcal{I}_2(\C)$ is the corresponding $21$-dimensional representation of $\SP_6$. Following Mukai's notation on p. 1544, let $V$ be a six-dimensional vector space with a symplectic form and choose a decomposition $V \cong U_0 \oplus U_\infty$ for two symplectic subspaces $U_0, U_\infty$.
Then we identify the representations in \cite[Equation 0.1]{Muk} as follows. The first equation (representing $6$ quadrics) lives in a space of symmetric $3 \times 3$ matrices $\Sym_3k$ corresponding to $\Sym^2 U_0$, the second equation (representing another $6$ quadrics) lives in $\Sym_3 k \cong \Sym^2 U_\infty$ and the third equation (representing $9$ quadrics) lives in a space of $3 \times 3$ matrices, $\mathrm{Mat}_3 \cong U_0 \otimes U_\infty$. Together, we recognize $(\Sym^2 U_0) \oplus (\Sym^2 U_\infty) \oplus (U_0 \otimes U_\infty)$ as $\Sym^2 V$, which is also isomorphic to the adjoint representation of $\SP_6$.
\end{proof}


\begin{cor} \label{cV}
The Chern classes $c_2(\V),c_4(\V), c_6(\V)$ are all tautological.
\end{cor}
\begin{proof}
By the previous Lemma, there is an exact sequence
\[
0\rightarrow \Sym^2 \V \rightarrow \Sym^2 (f_*\omega_{f}) \rightarrow f_*(\omega_f^{\otimes 2})\rightarrow 0.
\]
By the splitting principle and Grothendieck--Riemann--Roch,
the Chern classes of $\Sym^2 (f_*\omega_{f})$ and $f_*(\omega_f^{\otimes 2})$ are tautological. Hence, the Chern classes of $\Sym^2 \V$ are tautological. By the splitting principle and the fact that the odd Chern classes of $\V$ vanish, we have
\[
c(\Sym^2 \V) = 1 + 8 c_2(\V) + [22c_2(\V)^2 + 14 c_4(\V)] + [28 c_2(\V)^3 + 54 c_2(\V) c_4(\V) + 38 c_6(\V)] + \ldots.
\]
It follows that $c_2(\V), c_4(\V), c_6(\V)$ are tautological.
\end{proof}

By Lemmas \ref{gens} and \ref{Sbundle} and Corollary \ref{cV}, we conclude that $A^*(\M_9 \smallsetminus \M_9^5)$ is tautological. Combining this with Lemma \ref{H59} completes the proof of Theorem \ref{main9}.

\bibliographystyle{amsplain}
\bibliography{refs}
\end{document}